\documentclass[review]{siamart220329}
\usepackage{diagbox}
\usepackage{subfig}

\usepackage{lipsum}
\usepackage{amsfonts}
\usepackage{graphicx}
\usepackage{epstopdf}
\usepackage{algorithmic}
\ifpdf
  \DeclareGraphicsExtensions{.eps,.pdf,.png,.jpg}
\else
  \DeclareGraphicsExtensions{.eps}
\fi

\newsiamremark{remark}{Remark}
\newsiamremark{hypothesis}{Hypothesis}
\crefname{hypothesis}{Hypothesis}{Hypotheses}
\newsiamthm{claim}{Claim}
\newsiamthm{assumption}{Assumption}

\headers{Adaptive angular domain compression scheme}{Qinchen Song, Jingyi Fu, Min Tang and Lei Zhang}

\title{An Adaptive Angular Domain Compression Scheme For Solving Multiscale Radiative Transfer Equation\thanks{Submitted to the editors DATE.
\funding{MT is
 supported by the Strategic Priority Research Program of Chinese Academy of Sciences Grant No.XDA25010401; NSFC12031013, Shanghai pilot innovation project 21JC1403500 and Mevion Medical Systems, Inc., Kunshan. LZ is partially supported by the Shanghai Municipal Science and Technology Project 22JC1401600, NSFC grant 12271360, and the Fundaffirmingmental Research Funds for the Centraffirmingl Universities.}}}

\author{Qinchen Song\thanks{Department of Mathematics, Institute of Natural Sciences,
Shanghai Jiao Tong University, Minhang, Shanghai 200240, P.R. China
  (\email{sqc9931@sjtu.edu.cn}, \email{nbfufu@sjtu.edu.cn}, \email{tangmin@sjtu.edu.cn}, \email{lzhang2012@sjtu.edu.cn}).}
\and Jingyi Fu\footnotemark[2]
\and Min Tang\footnotemark[2]
\and Lei Zhang\footnotemark[2]}

\usepackage{amsopn}
\DeclareMathOperator{\diag}{diag}

\nolinenumbers
\begin{document}

\maketitle

\begin{abstract}
When dealing with the steady-state multiscale radiative transfer equation (RTE) with heterogeneous coefficients, spatially localized low-rank structures are present in the angular space. This paper introduces an adaptive tailored finite point scheme (TFPS) for RTEs in heterogeneous media, which can adaptively compress the angular space. It does so by selecting reduced TFPS basis functions based on the local optical properties of the background media. These reduced basis functions capture the important local modes in the velocity domain. A detailed a posteriori error analysis is performed to quantify the discrepancy between the reduced and full TFPS solutions. Additionally, numerical experiments demonstrate the efficiency and accuracy of the adaptive TFPS in solving multiscale RTEs, especially in scenarios involving boundary and interface layers.
\end{abstract}

\begin{keywords}
radiative transfer equation (RTE); discrete ordinates method (DOM); heterogeneous media; tailored finite point scheme (TFPS); adaptive method; low-rank structure
\end{keywords}

\begin{MSCcodes}
35Q70, 65N22, 65N50, 65N06, 65G99
\end{MSCcodes}

\section{Introduction}
\label{sec:intro}
The radiative transport equation (RTE) describes the propagation and interaction of particles, such as photons and neutrons, with background media. This equation finds extensive applications in various fields, including nuclear engineering, atmospheric science, thermal transport and optical tomography. The steady state radiative transport equation reads as follows:
\begin{equation}\label{eq:sec1:1}
    \mathbf{u}\cdot\nabla \psi(\mathbf{z},\mathbf{u})+\sigma_{T}(\mathbf{z})\psi(\mathbf{z},\mathbf{u})=\sigma_{s}(\mathbf{z})\int_{S}\kappa(\mathbf{u},\mathbf{u}')\psi(\mathbf{z},\mathbf{u}')d\mathbf{u}'+q(\mathbf{z}),
\end{equation}
where $\mathbf{z}\in \Omega \subset\mathbb{R}^{3}$ and $\mathbf{u}\in\mathbb{S}$ represents the location and moving direction of the particles, with the set $S:=\{\mathbf{u}|\mathbf{u}\in\mathbb{R}^{2}, \vert \mathbf{u}\vert =1\}$ denoting all possible directions. The angular flux $\psi(\mathbf{z},\mathbf{u})$ represents the density of particles moving along the direction $\mathbf{u}$ at location $\mathbf{z}$. $\sigma_{T}$, $\sigma_{s}$, $\sigma_{a}:=\sigma_{T}-\sigma_{s}\geq 0$ and $q$ correspond to the total cross section, scattering cross section, absorption cross section and external source, respectively, which are assumed to be space dependent and piecewise smooth. Besides, the $L^{\infty}$ norm of their derivatives, $\nabla\sigma_{T}$, $\nabla\sigma_{a}$, and $\nabla q$, are piecewise bounded.
The kernel function $\kappa(\mathbf{u},\mathbf{u}')$ provides the transitional probability for particles moving in direction $\mathbf{u}'$ to be scattered into direction $\mathbf{u}$. One typical example is the Henyey-Greenstein (HG) function \cite{henyey1941diffuse}, which depends solely on the inner product $\mathbf{u}\cdot\mathbf{u}'$ of directions $\mathbf{u}$ and $\mathbf{u}'$ and is defined as follows:
\begin{equation}
\kappa(\mathbf{u},\mathbf{u}')=\frac{1-g^{2}}{(1+g^{2}-2g\mathbf{u}\cdot\mathbf{u}')^{3/2}}.
\end{equation}
Here the parameter $g\in[-1,1]$ is the anisotropy factor and is used to characterize the angular distribution of scattering. 
For the sake of simplicity, we adopt the Dirichlet boundary condition for equation \cref{eq:sec1:1}. This boundary condition specifies the inflow value at the physical boundary:
\begin{equation}\label{eq:sec1:2}
    \psi(\mathbf{z},\mathbf{u})=\Psi_{\Gamma^{-}}(\mathbf{z},\mathbf{u}),\quad z\in\Gamma_{\mathbf{u}}^{-}=\{\mathbf{z}\in\Gamma=\partial\Omega: \mathbf{u}\cdot\mathbf{n}<0 \},\quad \mathbf{u}\in\mathbf{S},
\end{equation}
where $\mathbf{n}$ is the outer normal direction of $\Omega$.

The problem of high dimensionality poses a significant challenge when numerically solving RTE. Specifically, in a spatial two-dimensional (2D) scenario, we encounter four dimensions in total: two for velocity directions and two for spatial coordinates. If we extend our analysis to three spatial dimensions real-world problem, the unknowns encompass a total of five dimensions. 

The numerical challenges are further compounded by the heterogeneity of material parameters. In various applications, particles travel through different mediums that possess distinct optical properties. These properties are typically characterized by parameters such as $\sigma_{T}$, $\sigma_{a}$, $\sigma_{s}$, and $\kappa$. Here are examples that illustrate how these parameters contribute to the characterization of different optical regimes from different perspectives:
\begin{enumerate}
    \item \emph{optically thick} and \emph{optically thin} regimes: Let $L$ denote the characteristic length. In the \textit{optically thick} regime, $\sigma_{T}L\gg 1$, indicating a high average number of interactions between particles and the background media over the length scale $L$. Conversely, in the \textit{optically thin} regime, $\sigma_{T}L\ll 1$.
    \item \emph{absorption dominated} and \emph{scattering dominated} regime: In the \textit{absorption-dominated} regime, $\sigma_{s}\ll\sigma_{T}$, indicating that particles are more likely to be absorbed than scattered when interacting with the background media. In the \textit{scattering dominated} regimes, $\sigma_{a}\ll\sigma_{T}$, indicating a higher probability of particle scattering than absorption. 
    \item \emph{isotropic} and \emph{anisotropic} regimes: In the \textit{isotropic regime}, particles are scattered uniformly to all directions ($\kappa \equiv 1$). In the anisotropic regime, particles exhibit preferential scattering directions, indicated by a non-zero $g$ and thus an anisotropic $\kappa$.
\end{enumerate}

In particular, the diffusive regime and transport regime are extensively studied in the literature, characterized by special scalings: $\sigma_{a}\ll 1$ and $\sigma_{T}\gg 1$ in the diffusive regime; $\sigma_{a}=O(1)$ and $\sigma_{T}=O(1)$ in the transport regime. Various schemes are available to handle these regimes. Asymptotic-preserving (AP) schemes are used when both regimes coexist and are difficult to separate. Conventional AP discretization schemes \cite{lewis1984computational, reed1971new, larsen1987asymptotic, larsen1989asymptotic, adams2001discontinuous,jin2010asymptotic} require the same mesh in both regimes, leading to a large linear system. Domain decomposition schemes \cite{klar1995domain, bal2002coupling, golse2003domain, li2015diffusion} can be applied when there is clear spatial separation between the diffusive and transport regimes. These schemes decouple the equations in each regime, utilizing the diffusive approximation in the diffusive regime for computational efficiency and directly solving the transport equation in the transport regime.

In this paper, our objective is to develop a scheme capable of handling a general multiscale RTE, possibly with boundary and interface layers. This scheme should allow for significant variations in parameters and cover a wide range of particle behaviors even when there exist transition regimes and no clear separation. On the other hand, only a lower dimensional problem needs to be solved in some particular regime, so that the computational complexity can be reduced. For example, only a diffusion equation needs to be solved in the diffusive regime. To achieve this, the main difficulty is to find the approximated lower dimensional model with required accuracy and find way to piece together models with different dimensions.

We will present our work using the discrete ordinate method (DOM) \cite{chandrasekhar2013radiative, kubelka1948new, schuster1905radiation, azmy2010advances, lewis1984computational} for angular discretization and the tailored finite point scheme (TFPS) \cite{han2014two} for spatial discretization. DOM represents the angular domain with carefully selected discrete points and approximates the integral term in the radiative transfer equation (RTE) using numerical quadrature. TFPS is a spatial discretization scheme designed for solving the steady state discrete ordinate transport equation. It uses fundamental solutions as local basis functions within each spatial cell, ensuring continuity at cell interfaces. TFPS has been proven to be an asymptotic preserving (AP) scheme \cite{wang2022uniform}, capable of handling the diffusive regime. It can handle absorption dominated or transition regimes with coarse spatial meshes due to its exact satisfaction of the homogeneous RTE with constant coefficients. 

We aim to enhance TFPS by reducing the degrees of freedom in the angular domain. In 1D, we observe a low-rank structure in the velocity space of multiscale RTEs, thanks to the exponential decay properties of the local basis functions. These basis functions can accurately solve the discrete ordinate RTE. In higher dimensions, TFPS utilizes similar exponential decaying basis functions and interface continuity conditions. These special basis functions contain low-rank information in the velocity space, allowing for further computational cost reduction. However, previous work did not utilize this low-rank structure, employing a fixed number of basis functions ($8M$, twice the number of discrete velocity directions) in each physical cell, regardless of the cell's regime. To effectively capture this low-rank structure and compress the velocity space, we introduce Adaptive TFPS.

Adaptive TFPS selects basis functions based on their relative contributions to the scalar flux at the cell center, using a tolerance parameter $\delta>0$. These basis functions determine the important modes in the velocity domain and vary from cell to cell due to different physical parameters. Continuity conditions at interfaces and boundaries are defined accordingly. By enforcing specially designed continuity conditions at the centers of cell edges relevant to these important modes, the local basis of Adaptive TFPS is then pieced together. The accuracy of the numerical solution away from cell interfaces is guaranteed by a posterior analysis in \cref{sec:analysis} and results from \cite{han2014two}.

We note that previous studies have explored compressing the angular space of RTEs using, e.g., reduced basis method (RBM) \cite{peng2022reduced} and proper orthogonal decomposition (POD) \cite{buchan2015pod, hughes2020discontinuous}. RBM and POD treat the angular variable as a parameter, selecting representative parameters and constructing a low-rank solution space. This significantly reduces computation time for online simulations of reduced solutions. Additionally, the idea of selecting local basis functions independently for each partition of the space-angle phase-space is proposed in \cite{hughes2022adaptive}, although with different partitioning and basis selection methods compared to our approach. Other related works include random singular value decomposition (RSVD) \cite{chen2020random, chen2021low}, proper generalized decomposition (PGD) \cite{alberti2020reduced, dominesey2019reduced, prince2019separated}, and dynamic mode decomposition (DMD) \cite{mcclarren2018acceleration, mcclarren2022data}.

The novelty of Adaptive TFPS lies in three aspects.
\begin{itemize}
    \item Adaptive TFPS achieves compression of the angular space for general multiscale RTEs by utilizing prior knowledge of the local optical properties of the background medium. By incorporating this information, the method gains valuable insights into the inherent characteristics of the low-dimensional structure present within the system.
    \item The proposed scheme introduces a tunable parameter, denoted as $\delta$, which provides the ability to control the accuracy of the numerical solution away from boundary layers and interface layers. Additionally, this parameter allows for fine-tuning the precision of recovering these specific layers of interest.
    \item To validate the accuracy of Adaptive TFPS, a comprehensive a posterior analysis is performed to assess the error between the compressed solution obtained using Adaptive TFPS and the uncompressed solution obtained with the full-order TFPS scheme. This analysis plays a critical role in evaluating the reliability and accuarcy of Adaptive TFPS.
\end{itemize}

\paragraph{Outline}
The paper is organized as follows: The low-rank structure in the angular domain is explored through 1D examples in \cref{sec:compress} and extended to the 2D case with the introduction of the TFPS spatial discretization scheme in \cref{sec:TFPS}. The construction of Adaptive TFPS, which leverages the low-rank structures, is described in \cref{sec:Adaptive TFPS}. The analysis of Adaptive TFPS and an upper bound for the approximation error are presented in \cref{sec:analysis}. The effectiveness of the algorithm is validated in \cref{sec:experiments} through various benchmark examples. The paper concludes with a summary of key findings and future research directions in \cref{sec:conclusion}. Some Supplementary information is provided in the appendices.

\section{Low rank structure in velocity domain: 1D illustration}
\label{sec:compress}

In this section, we demonstrate the low-rank structure in the velocity domain of the 1D RTE. We discretize the angular domain with DOM to obtain the discrete ordinate RTE. By deriving basis functions, i.e., analytical solutions of discrete ordinate RTEs, we observe the existence of a low-rank structure in 1D RTEs.

\subsection{DOM for 1D RTE}

In slab geometry, the RTE reads as follows\cite{lewis1984computational,jin1991discrete}:
\begin{eqnarray}\label{eq:sec2:21}
\begin{aligned}
\mu\partial_{z}\psi(z,\mu)+\sigma_{T}(z)\psi(z,\mu)=\sigma_{s}(z)\int_{-1}^{1}\kappa(\mu,\mu')\psi(z,\mu')d\mu'+q(z)\\
z\in[z_{l},z_{r}],\quad \mu\in[-1,1],\\
\end{aligned}
\end{eqnarray}
with boundary conditions: $\psi(z_{l},\mu)=\Psi_{l}(\mu),\ \mu>0$; $ \psi(z_{r},\mu)=\Psi_{r}(\mu),\ \mu<0$.
Here, we consider a simple case where the coefficients $\sigma_{T}$, $\sigma_{a}$, and $q$ are constant and $\kappa\equiv 1$. 

We then employ DOM to discretize the angular domain of the 1D RTE. Let $\{\mu_{m},\omega_{m}\}_{m\in\mathcal{M}^{\mathrm{1d}}}$ be the finite quadrature set with $\mathcal{M}^{\mathrm{1d}}=\{1,2,\dots,2M\}$, where $\mu_{m}$ is the $m$-th velocity direction, and $\omega_{m}$ is its corresponding weight.  Subsequently, we obtain the following \textit{discrete ordinate RTE}:
\begin{equation}\label{eq:sec2:22}
\mu_{m}\partial_{z}\psi_{m}(z)+\sigma_{T}\psi_{m}(z)=\sigma_{s}\sum_{n\in V}\omega_{n}\psi_{n}(z)+q,\quad z\in[z_{l},z_{r}],\  m\in\mathcal{M}^{\mathrm{1d}},
\end{equation}
with boundary conditions $\psi_{m}(z_{l})=\Psi_{l}(\mu_{m}),\ \mu_{m}>0$; $\psi_{m}(z_{r})=\Psi_{r}(\mu_{m}),\ \mu_{m}<0$. Here $\psi_{m}(z)$ is an approximation of $\psi(z,\mu_{m})$. With a little abuse of the notation, we denote $\psi(z)$ as a vector-valued function: $\psi=(\psi_{1}(z),\psi_{2}(z),\dots,\psi_{2M}(z))^{T}$.

\subsection{Basis functions as analytical solutions of discrete ordinate RTE} 

Clearly, \cref{eq:sec2:22} is an constant coefficient equation which can be solved analytically.
The special solution to Equation \cref{eq:sec2:22} is:
\begin{equation}
    \phi^{s}(z)=\frac{\sigma_{T}-\sigma_{s}}{q}
\end{equation}
and fundamental solutions to the homogeneous equation associated with \cref{eq:sec2:22} are:
\begin{eqnarray}\label{eq:sec2:24}
    \begin{aligned}
        \phi^{(k)}(z)=\xi^{(k)}\exp\{\lambda^{(k)}\sigma_{T}(z-z^{(k)})\},\quad 1\leq k\leq 2M,\\
        z^{(k)}=z_{l},\quad 1\leq k\leq M;\quad z^{(k)}=z_{r},\quad M+1\leq k\leq 2M.\\
    \end{aligned}
\end{eqnarray}    
Here $(\lambda^{(k)},\xi^{(k)})$ for $1\leq k\leq 2M$ is the eigenpair for the matrix $\displaystyle M=U^{-1}[\frac{\sigma_{s}}{\sigma_{T}}W^{\mathrm{1d}}-I]$, where $U$, and $W^{\mathrm{1d}}$ are defined as follows:
\[U=\diag\{\mu_{1},\mu_{2},\dots,\mu_{2M}\},\quad W^{\mathrm{1d}}=\diag\{\omega_{1},\omega_{2},\dots,\omega_{2M}\}.\]
The vectors $\xi^{(k)}$ are normalized such that $\Vert \xi^{(k)}\Vert_{\infty}=1$, and as has been proved in \cite{jin1991discrete}, we arrange the eigenvalues as $\lambda^{(1)}<\cdots<\lambda^{(M)}<0<\lambda^{(M+1)}<\cdots\lambda^{(2M)}$. 

Using $\{\phi^{(k)}(z)\}_{k=1,\dots,2M}$ as basis functions, the solution space to \cref{eq:sec2:22} can be constructed as follows:
\[
\setlength{\abovedisplayskip}{3pt}
\setlength{\belowdisplayskip}{3pt}
\mathcal{F}^{\mathrm{1d}}=\big\{\sum_{k=1}^{2M}\alpha^{(k)}\phi^{(k)}+\phi^{s}\big| \alpha^{(k)}\in\mathbb{R}\big\}.
\]

By imposing Dirichlet boundary conditions at the physical boundary, we can obtain the unique solution $\psi$ to equation \cref{eq:sec2:22} from the set $\mathcal{F}^{\mathrm{1d}}$.

\subsection{Low rank structure indicated by basis functions:}

As depicted in \cref{eq:sec2:24}, the basis functions $\phi^{(k)}$ exhibit exponential decay, with the decay rate governed by the value of $\lambda^{(k)}\sigma_{T}$. We show the detailed value of $\lambda^{(k)}$ for different \(\frac{\sigma_{s}}{\sigma_{T}}\) in  \cref{fig:lambdak}, demonstrating that aside from the two eigenvalues with the smallest magnitude ($\lambda^{(M)}$ and $\lambda^{(M+1)}$), the other eigenvalues remain relatively stable. Consequently, with a relatively large $\sigma_{T}$, we can expect to have basis functions that decay fast enough so that their impact away from the physical boundary becomes negligible. For example, when \(\sigma_{T}=10\) and \(\frac{\sigma_{s}}{\sigma_{T}}=0.5\), the basis function \(\phi_{C}^{(1)}\) should decay faster than \(\exp(-100z)\).



\begin{figure}[htbp]
    \vspace{-0.3cm}
    \setlength{\abovecaptionskip}{0.cm}
    \centering
    \subfloat[$\frac{\sigma_{s}}{\sigma_{T}}=0.5$]
    {
    \label{fig:lambdak1}
    \includegraphics[width=0.3\textwidth]{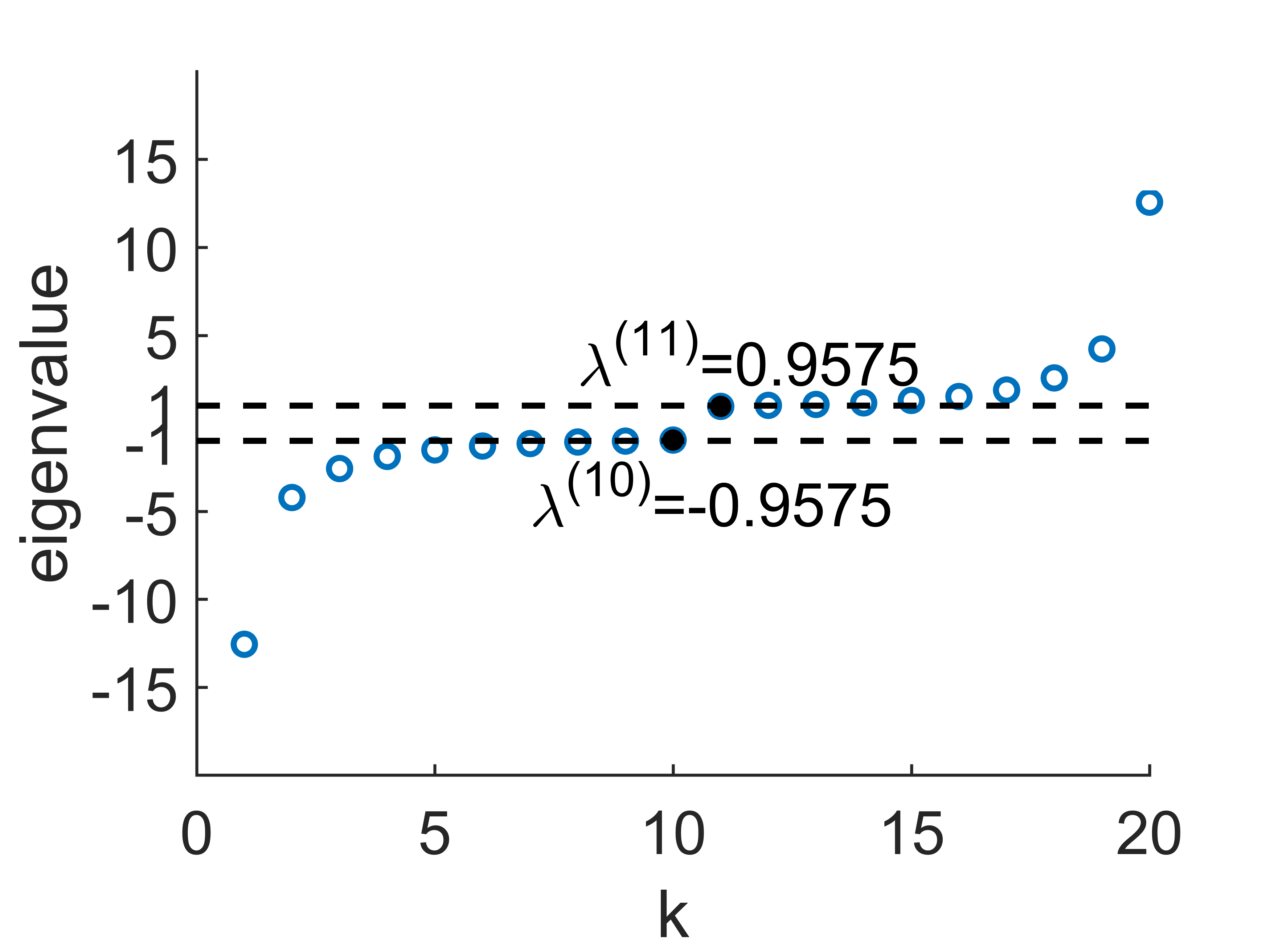}
    }
    \subfloat[$\frac{\sigma_{s}}{\sigma_{T}}=0.995$]
    {
    \label{fig:lambdak2}
    \includegraphics[width=0.3\textwidth]{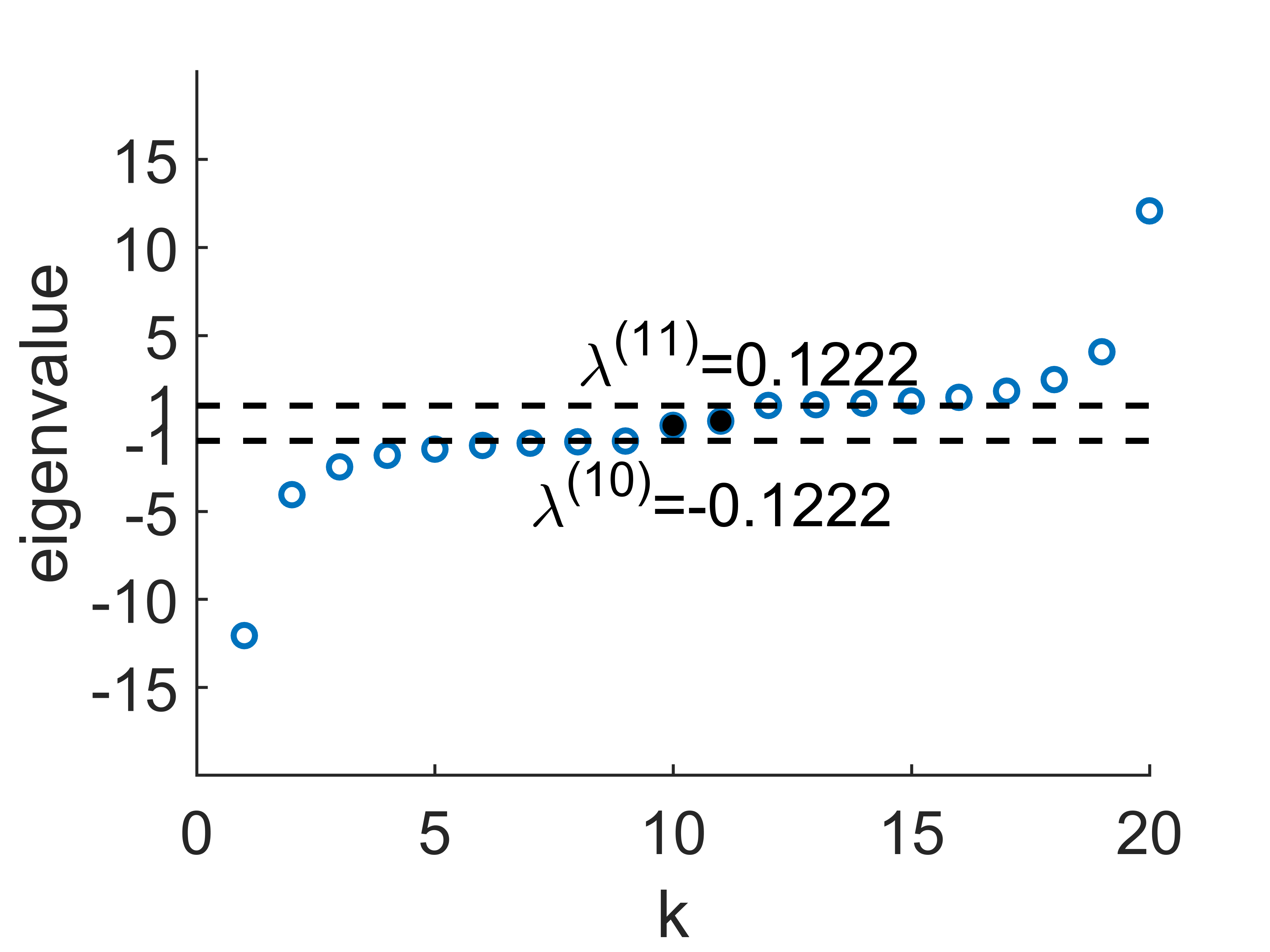}
    }
    \subfloat[$\frac{\sigma_{s}}{\sigma_{T}}=0.99995$]
    {
    \label{fig:lambdak3}
    \includegraphics[width=0.3\textwidth]{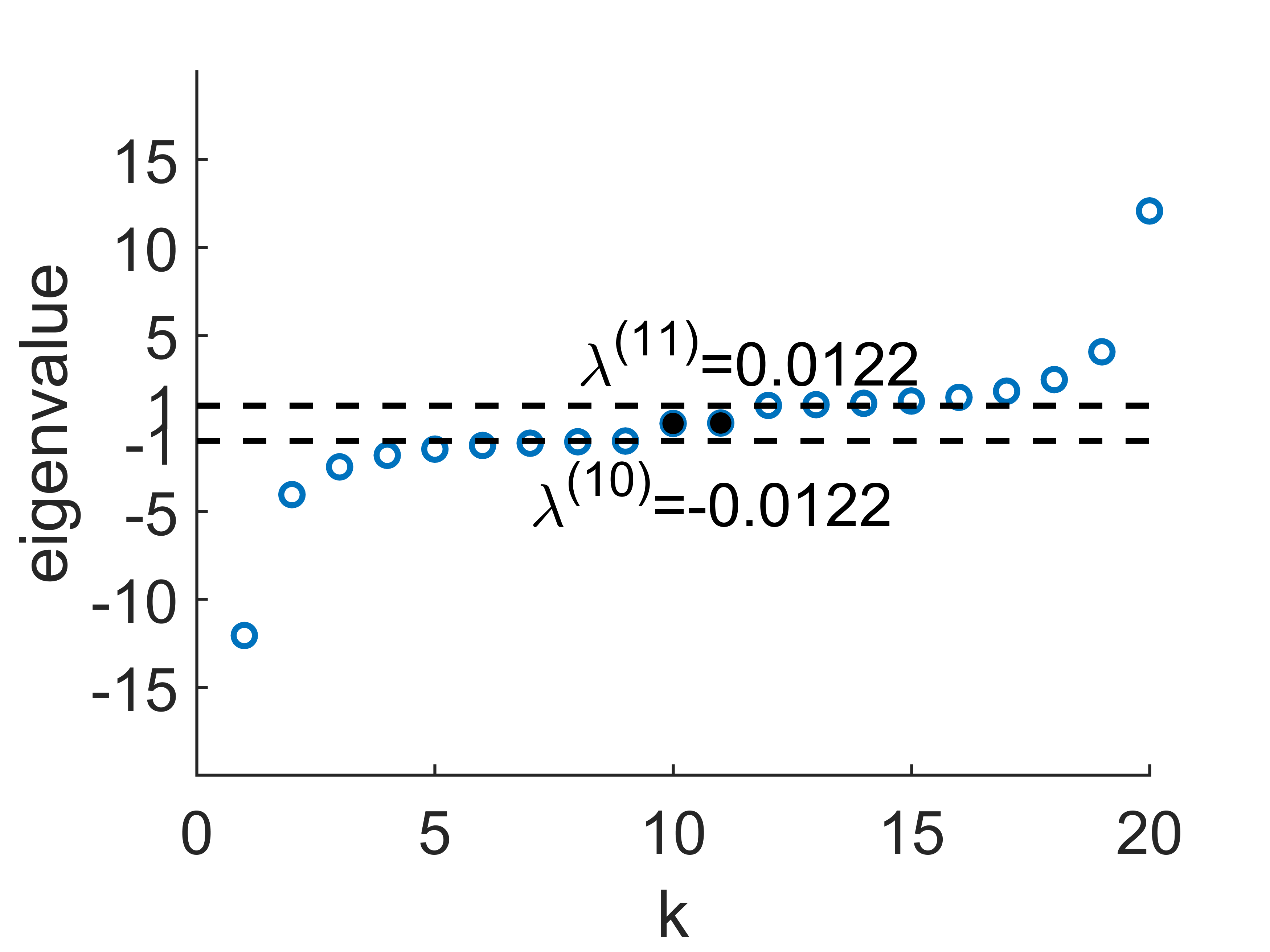}
    }
    \caption{The figure illustrates the values of the eigenvalue $\lambda^{(k)}$ for different choice of $\frac{\sigma_{s}}{\sigma_{T}}$. The eigenvalues with the smallest magnitude (\(\lambda^{(10)}\), \(\lambda^{(11)}\)) are also highlighted in the figure, corresponding to \(k=10\) and 11.}
    \label{fig:lambdak}
    \vspace{-0.3cm}
\end{figure}

Subsequently, by introducing a threshold $\delta>0$, we define $\phi_{\delta}$ as the truncated form of $\phi$ obtained by removing rapidly decaying basis functions whose values at the center of the spatial domain are less than $\delta$ in the infinity norm. The profiles of $\phi$ and $\phi_{\delta}$ for varying $\delta$ and fixed $\sigma_{T}$ and $\sigma_{a}$ are depicted in \cref{fig:solution1D}. As $\delta$ decreases from $10^{-2}$ to $10^{-5}$, $\phi_{\delta}$ approximates $\phi$ better. However, the number of reduced basis functions increases from 2 to 14, illustrating the trade-off between accuracy and computational complexity. By eliminating these rapidly decaying basis functions, we can construct a compressed solution space, thereby revealing the low-rank structure in the velocity space of the general 1D RTE \cite{jin2009uniformly}.


\begin{figure}[htbp]
    \setlength{\abovecaptionskip}{0.cm}
    \centering
    \includegraphics[width=0.9\textwidth]{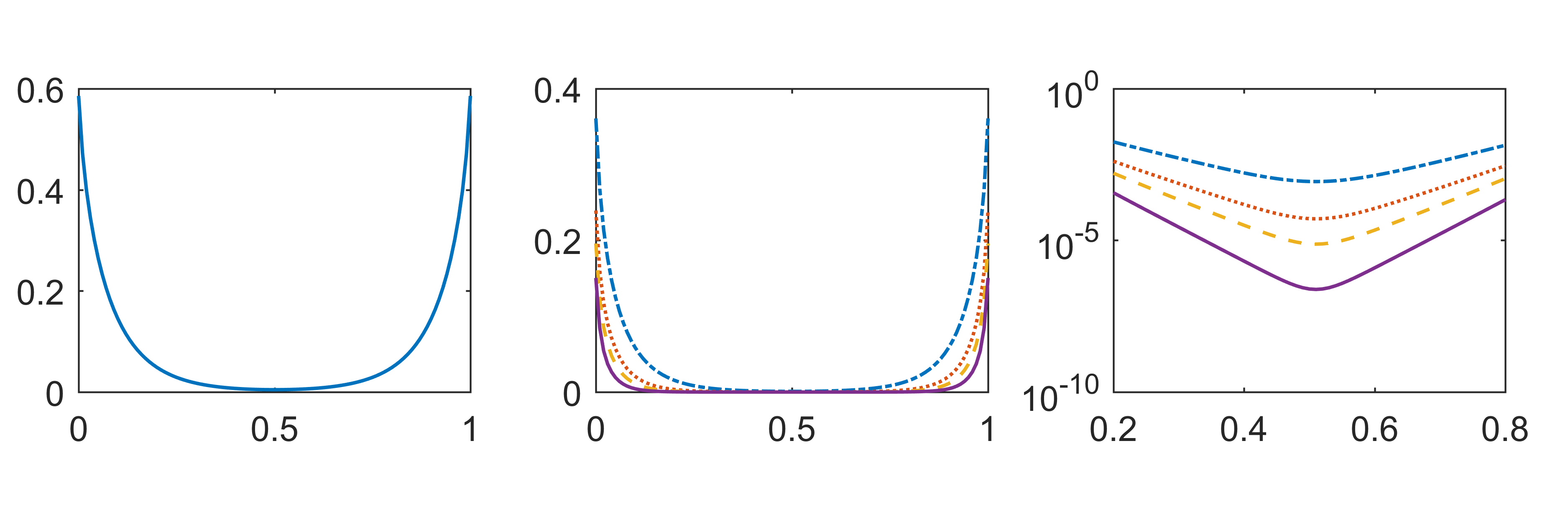}
    \caption{Left figure: the profile of $\phi$ when $\sigma_{s}=5$, $\sigma_{T}=10$. Middle and right figure: the difference between $\phi$ and $\phi_{\delta}$ for $\delta=10^{-2}$, $10^{-3}$, $10^{-4}$,  $10^{-5}$ when $\sigma_{s}=5$, $\sigma_{T}=10$. Here, the blue dot-dash line represents the case when $\delta=10^{-2}$, the red dotted line represents the case when $\delta=10^{-3}$, the dashed orange line represents the case when $\delta=10^{-4}$ and the purple solid line represent the case when $\delta=10^{-5}$.}
    \label{fig:solution1D}
    \vspace{-0.3cm}
\end{figure}

\section{Low rank structure in 2D case and TFPS}
\label{sec:TFPS}
Building upon the findings from the previous section, we apply the concept of exponential decay basis functions and the corresponding low rank structure, observed in the angular domain of 1D RTEs, to the two spatial dimentional case. This choice of basis functions, similar to the 1D case, enables the use of TFPS as a spatial discretization method. The two spatial dimensional case is a simplification of the real-world three-dimensional scenario. We adopt the x-y geometry as outlined in \cite{LewisMiller}, assuming that the angular flux exhibits mirror symmetry about the x-y plane. Consequently, the  angular domain is represented as a projection from the unit sphere $S^{2}$ onto the x-y plane.


\begin{figure}[htbp]
    \vspace{-0.3cm}
    \setlength{\abovecaptionskip}{0.cm}
    \centering
    \includegraphics[width=0.7\textwidth]{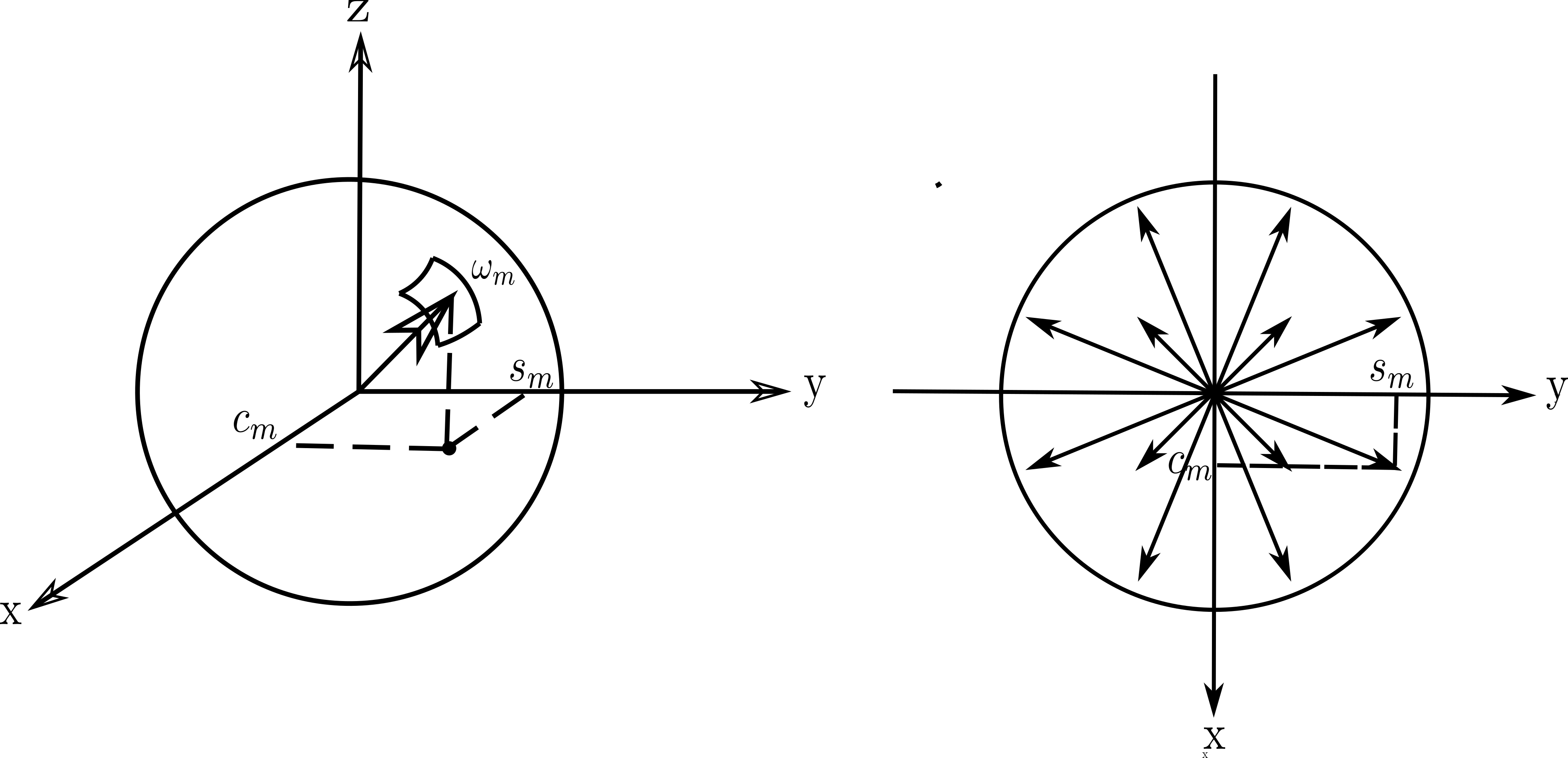}
    \caption{DOM in x-y geometry. Left figure: a quadrature point in three spatial dimensional case. Right figure: example of DOM ($S_{N}$) in x-y geometry with $N=2$.}
    \label{fig:DOM}
    \vspace{-0.3cm}
\end{figure}

\subsection{DOM for RTE in x-y geometry}

The DOM quadrature points in x-y geometry can be viewed as projections of the three-dimensional DOM quadrature points onto the two-dimensional x-y plane, as illustrated in \cref{fig:DOM}. We denote the quadrature set for DOM in x-y geometry as $\{\mathbf{u}_{m},\omega_{m}\}_{m\in\mathcal{M}_{\mathrm{2d}}}$ with $\mathcal{M}_{\mathrm{2d}}=\{1,2,\dots,4M\}$ being the index set for discrete velocity directions. Here $\mathbf{u}_{m}=(c_{m},s_{m})=((1-\zeta_{m}^{2})^{1/2}\cos(\theta_{m}), (1-\zeta_{m}^{2})^{1/2}\sin(\theta_{m}))$ represents the $m$-th velocity direction with $\zeta_{m}\in[0,1]$ and $\theta_m\in[0,2\pi)$, and $\omega_{m}$ represents the corresponding weight. 

We approximate the integral term in the RTE by its numerical quadrature with weights $\omega_{m}$, and denote the approximation of $\psi(x,y,\mathbf{u}_{m})$ as $\psi_{m}(x,y)$. Then we get the following 2D  \textit{discrete ordinate RTE} in x-y geometry:
\begin{equation}\label{eq:sec2:1}
    c_{m}\frac{\partial}{\partial x}\psi_{m}(x,y)+s_{m}\frac{\partial}{\partial y}\psi_{m}(x,y)+\sigma_{T}\psi_{m}(x,y)=\sigma_{s}\sum_{n\in V}\kappa_{mn}\psi_{n}(x,y)\omega_{n}+ q,
\end{equation}
for $m\in\mathcal{M}_{\mathrm{2d}}$. Here $\kappa_{mn}$ is an approximation of $\kappa(\mathbf{u}_{m},\mathbf{u}_{n})$ \cite{chen2018uniformly}. With a slight abuse of notation, we write $\psi$ as $\psi(x,y)=\Big(\psi_{1}(x,y),\psi_{2}(x,y),\dots,\psi_{4M}(x,y)\Big)^{T}$, 
which is a $4M$-dimensional vector function.

\subsection{Construction of exponential decaying basis functions and TFPS} 

For spatial discretization of the discrete ordinate RTE in x-y geometry \cref{eq:sec2:1}, we construct exponential decaying basis functions, similar to those used in the 1D case \cref{eq:sec2:24}. These basis functions are commonly employed in the Tailored Finite Point Scheme (TFPS) scheme \cite{han2014two}.

To be more specific, we consider the domain $\Omega=[0,1]\times[0,1]$ with grid points $x_i=ih$ ($0\leq i\leq I$) and $y_j=jh$ ($0\leq j\leq I$) as illustred in \cref{fig:mesh}, where $h=\frac{1}{I}$. Each cell is denoted as $C_{i,j}=[x_{i-1},x_{i}]\times[y_{j-1},y_{j}]$, with $i$ and $j$ as row and column indices. The set of all cells is denoted as $\mathcal{C}=\{C_{i,j}\}_{1\leq i,j\leq I}$. In subsequent sections, we also represent each cell $C\in\mathcal{C}$ as $C=[x^l_C, x^r_C]\times [y^b_C, y^t_C]$, or omit the subscript $C$ for simplicity when there is no confusion.

\begin{figure}[htbp]
    \setlength{\abovecaptionskip}{0.2cm}
    \centering
    \includegraphics[width=0.35\textwidth]{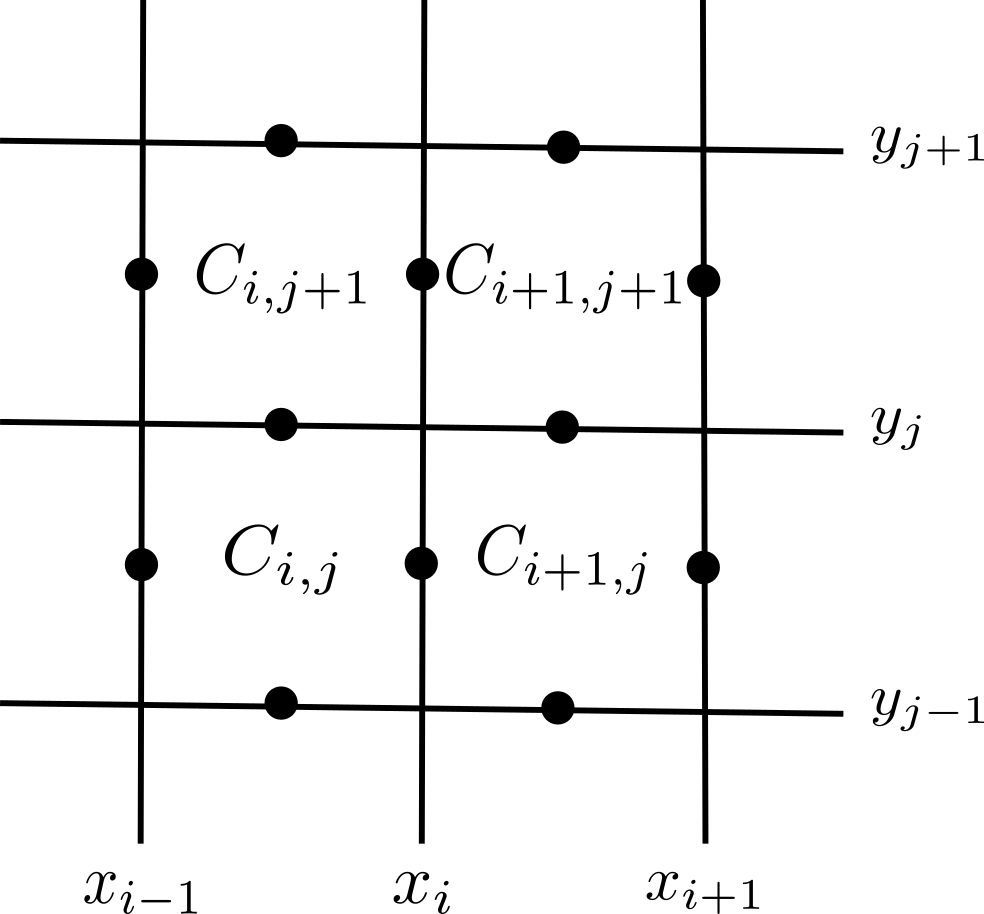}
    \caption{Spatial mesh}
    \label{fig:mesh}
    \vspace{-0.3cm}
\end{figure}

The following equation with piecewise constant coefficient is then utilized to approximate the discrete ordinate RTE (DORTE) in x-y geometry \cref{eq:sec2:1}:
\begin{equation}\label{eq:sec2:4}
\setlength{\abovedisplayskip}{3pt}
\setlength{\belowdisplayskip}{3pt}
    c_{m}\frac{\partial}{\partial x}\psi_{m}+s_{m}\frac{\partial}{\partial y}\psi_{m}+\bar{\sigma}_{T}\psi_{m}=\bar{\sigma}_s\sum_{n\in V}\kappa_{mn}\psi_{n}\omega_{n}+\bar{q},\quad m\in\mathcal{M}_{\mathrm{2d}},
\end{equation}
where $\sigma_{T}$, $\sigma_{a}$, $q$ are replaced by their cell averages, i.e. for any $C\in \mathcal{C}$:
\[
\setlength{\abovedisplayskip}{3pt}
\setlength{\belowdisplayskip}{3pt}
\bar{\sigma}_{T}(x,y)|_C=\sigma_{T,C}=\frac{1}{|C|}\int_C\sigma_{T}(x,y)dxdy,\]
\[
\setlength{\abovedisplayskip}{3pt}
\setlength{\belowdisplayskip}{3pt}
\bar{\sigma}_{s}(x,y)|_{C}=\sigma_{s,C}=\frac{1}{|C|}\int_C\sigma_{s}(x,y)dxdy,\]
\[
\setlength{\abovedisplayskip}{3pt}
\setlength{\belowdisplayskip}{3pt}
\bar{q}(x,y)|_{C}=q_{C}=\frac{1}{|C|}\int_C q(x,y)dxdy.\]
To form a global solution at the PDE level, we need to impose the interface and boundary conditions. The following continuity conditions are imposed at the interior cell interfaces, at any interior interface $\mathfrak{i}=C_+\cap C_-$, with $C_{+}, C_{-}\in\mathcal{C}$, we have 
\begin{equation}\label{eq:sec2:25}
    \psi|_{C_{+}}(x,y) = \psi|_{C_{-}}(x,y), \quad (x,y) \in \mathfrak{i},
\end{equation}
and boundary conditions are imposed as shown in \cref{eq:sec1:2}. 


Notice that within each cell $C\in \mathcal{C}$, $(x,y)\in \mathcal{C}$, \cref{eq:sec2:4} is an equation with constant coefficients
\begin{equation}\label{eq:sec2:5}
    c_{m}\frac{\partial}{\partial x}\psi_{m}+s_{m}\frac{\partial}{\partial y}\psi_{m}+\sigma_{T,C}\psi_{m}=\sigma_{s,C}\sum_{n\in V}\kappa_{m,n}\psi_{n}\omega_{n}+q_{C}, \quad m\in\mathcal{M}
\end{equation}
which can be solved analytically. The special solution to the local equation \cref{eq:sec2:5} is:
\begin{equation}\label{eqn:specialsolution}
\setlength{\abovedisplayskip}{3pt}
\setlength{\belowdisplayskip}{3pt}
    \phi_{C}^{s}(x,y)=\frac{q_{C}}{\sigma_{T,C}-\sigma_{s,C}},\quad (x,y)\in C.
\end{equation}
The homogeneous equation (when $q_{C}\equiv 0$) associated with \cref{eq:sec2:5} may have an infinite number of fundamental solutions. The following $8M$ fundamental solutions of particular forms are chosen as local basis functions in cell $C$.
\begin{eqnarray}\label{eq:sec2:6x}
    \begin{aligned}
        \phi_{C}^{(k)}(x,y)=\xi_{C}^{(k)}\exp\{\lambda_{C}^{(k)}\sigma_{T,C}(x-x_{C}^{(k)})\},\quad 1\leq k\leq 4M,\\
        x_{C}^{(k)}=x_C^l,\quad 1\leq k\leq 2M;\quad x_{C}^{(k)}=x_C^r,\quad 2M+1\leq k\leq 4M;\\
    \end{aligned}
\end{eqnarray}    
such that $\lambda_{C}^{(1)}<\cdots<\lambda_{C}^{(2M)}<0<\lambda_{C}^{(2M+1)}<\cdots\lambda_{C}^{(4M)}$, and
\begin{eqnarray}\label{eq:sec2:6y}
    \begin{aligned}        
        \phi_{C}^{(k)}(x,y)=\xi_{C}^{(k)}\exp\{\lambda_{C}^{(k)}\sigma_{T,C}(y-y_C^{(k)})\},\quad 4M+1\leq k\leq 8M,\\
        y^{(k)}=y_{C}^{b},\quad 4M+1\leq k\leq 6M;\quad y^{(k)}=y_{C}^{t},\quad 6M+1\leq k\leq 8M;
    \end{aligned}
\end{eqnarray}
such that  $\lambda_{C}^{(4M+1)}<\cdots<\lambda_{C}^{(6M)}<0<\lambda_{C}^{(6M+1)}<\cdots<\lambda_{C}^{(8M)}$. We note that these basis functions are the TFPS basis introduced in \cite{han2014two}. We refer to \cref{appendix:basis} for details of these functions and the identifications of $\lambda_C^{(k)}$.
\begin{remark}
In the derivation of basis functions and special solution, we assumed $\sigma_{a,C}\ne 0$ for simplicity. When $\sigma_{a,C}=0$, additional details about the basis functions and special solutions can be found in \cite{han2014two}.
\end{remark}

Denote $\mathcal{V}=\{1,2,\dots,8M\}$. Then the basis functions in $C$ are $\{\phi_{C}^{(k)}\}_{k\in \mathcal{V}}$, and the approximate solution to the piecewise constant coefficient discrete ordinate RTE \cref{eq:sec2:4} using TFPS is sought within the following solution space:
\[
\setlength{\abovedisplayskip}{3pt}
\setlength{\belowdisplayskip}{3pt}
\mathcal{F}=\big\{\sum_{C\in\mathcal{C}}(\sum_{k\in \mathcal{V}}\beta_{C}^{(k)}\phi_{C}^{(k)}+\phi_{C}^{s})\big| \beta_{C}^{(k)}\in\mathbb{R}\big\}.
\]

\subsection{Approximate Solution by TFPS}

We denote $\tilde{\psi}$ as the approximate solution to the discrete ordinate equation \cref{eq:sec2:4}, together with the interface condition \cref{eq:sec2:25} and boundary condition \cref{eq:sec1:2}, in the following form, 
\begin{equation}\label{eq:sec3:41}
    \tilde{\psi}=\sum_{C\in\mathcal{C}}(\sum_{k\in \mathcal{V}}\alpha_{C}^{(k)}\phi_{C}^{(k)})+\phi_{C}^{s}
\end{equation}
where $\alpha_{C}^{(k)}$ is the currently unknown coefficients for TFPS basis function $\phi_{C}^{(k)}$. For two and higher dimensions, we cannot impose continuity conditions for $\tilde{\psi}$ at all points of the interior cell interfaces and boundary conditions at all points of the spatial boundary. Instead, we impose continuity conditions only at the centers of the interior cell interfaces and boundary conditions only at the centers of the boundary interfaces.

Let $\mathcal{I}$ denote the set of cell interfaces for all cells in $\mathcal{C}$, with its subset $\mathcal{I}_{b}$ representing the interfaces at the spatial boundary and its subset $\mathcal{I}_{i}$ representing the interfaces inside the spatial domain. Besides, we denote $\mathbf{x}_\mathrm{mid}$ as the midpoint of interface $\mathfrak{i}\in\mathcal{I}$.

\begin{figure}[htbp]
    \vspace{-0.3cm}
    \setlength{\abovecaptionskip}{0.cm}
    \centering
    \subfloat[$\mathfrak{i}=C_{-}\cap C_{+}$]
    {
    \label{fig:interior_grid_point}
    \includegraphics[width=0.25\textwidth]{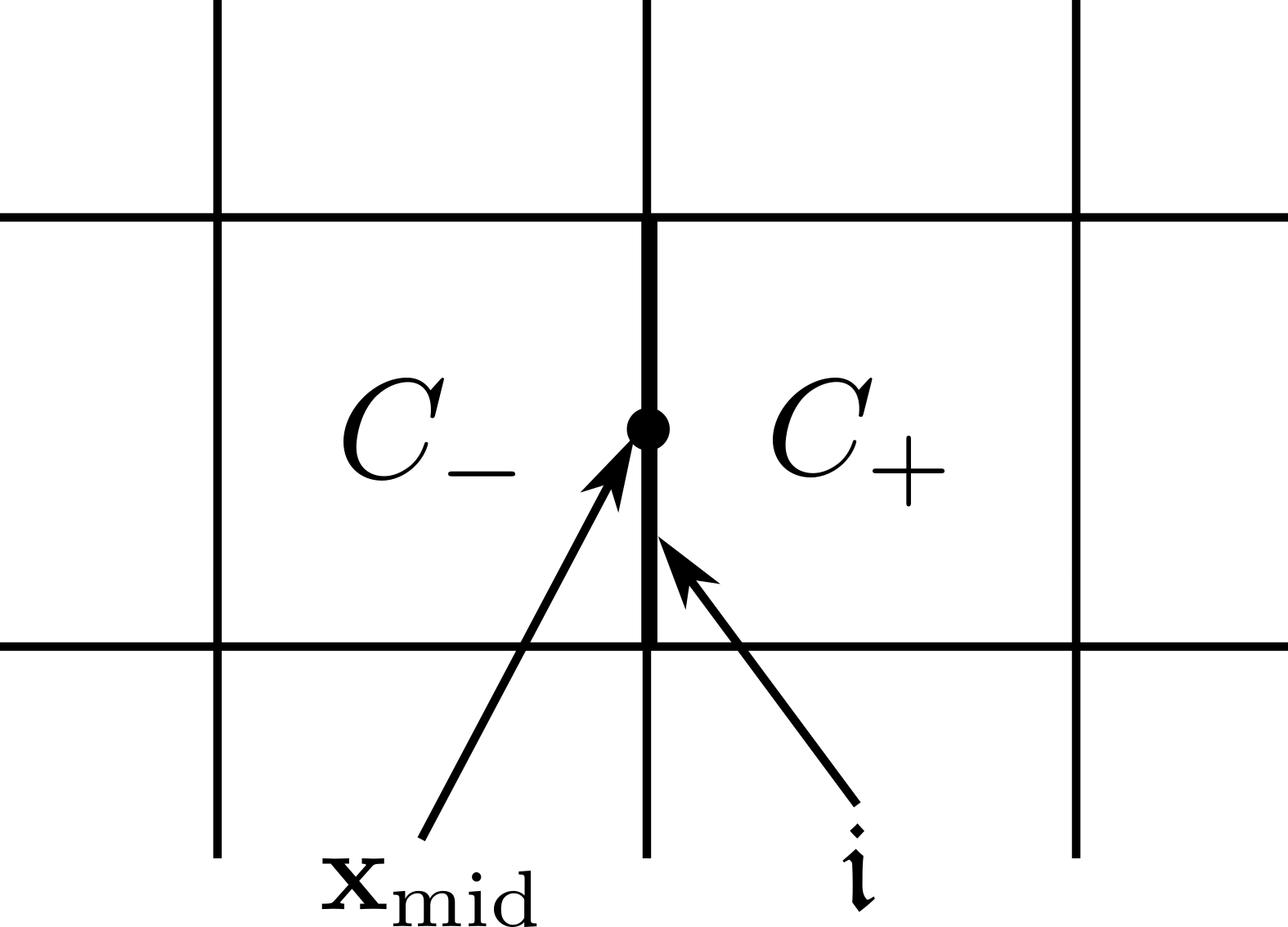}
    }
    \subfloat[$\mathfrak{i}=C\cap \partial\Omega$]
    {
    \label{fig:boundary_grid_point}
    \includegraphics[width=0.25\textwidth]{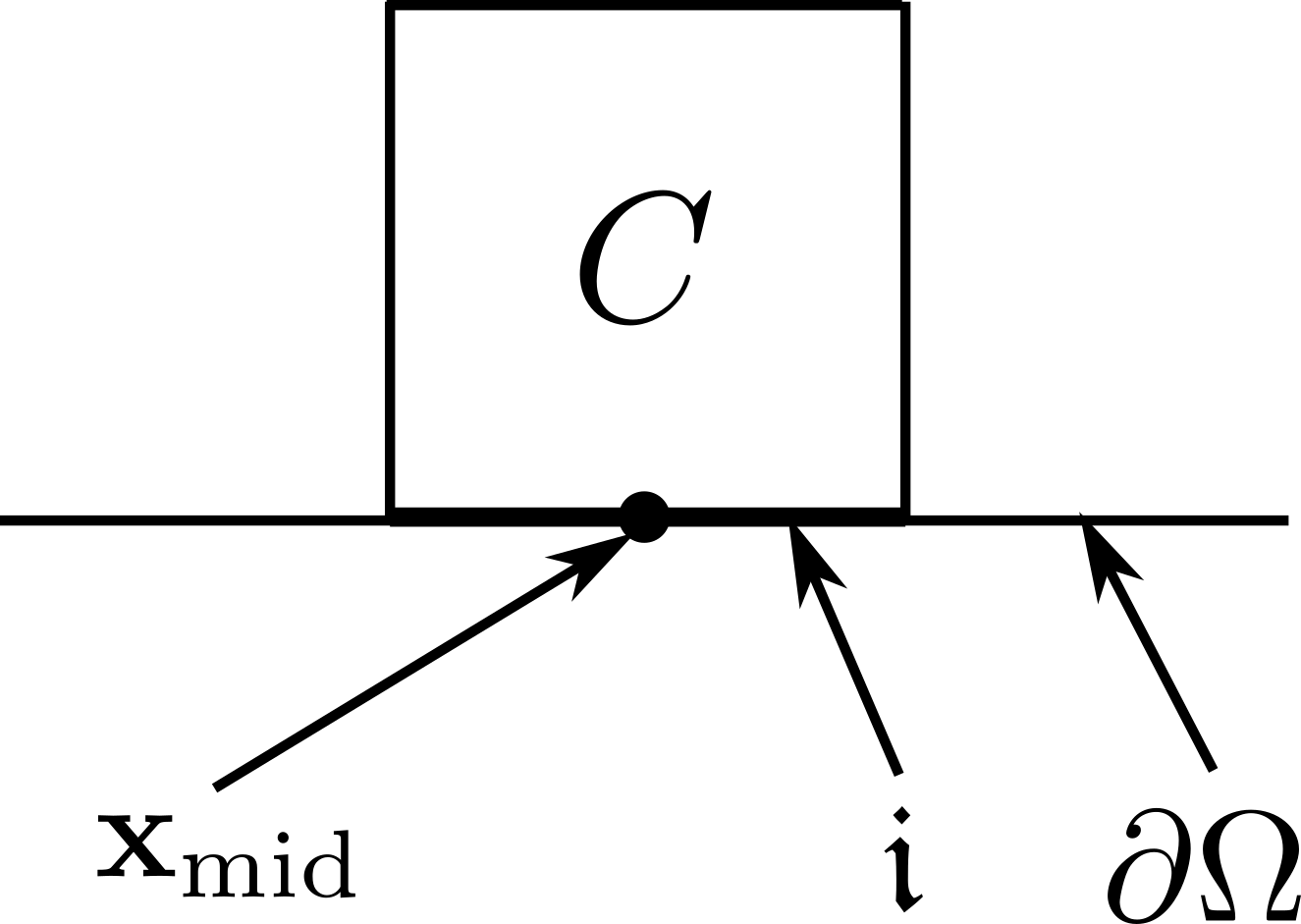}
    }
    \caption{Examples of interface $\mathfrak{i}$ located between two cells or between cell and physical boundary.}
    \label{fig:interior_boundary}
    \vspace{-0.3cm}
\end{figure}

When $\mathfrak{i}\in\mathcal{I}_{i}$, as shown in  \cref{fig:interior_grid_point}, the continuity condition at $\mathbf{x}_\mathrm{mid}$ is 
\begin{equation}\label{eqn:interfaceequations}
    \tilde{\psi}|_{C_{+}}(\mathbf{x}_\mathrm{mid})= \tilde{\psi}|_{C_{-}}(\mathbf{x}_\mathrm{mid}).
\end{equation}

When $\mathfrak{i}\in \mathcal{I}_{b}$, as shown on the right of \cref{fig:interior_grid_point}, the boundary condition at $\mathbf{x}_\mathrm{mid}$ is
\begin{equation}\label{eqn:boundaryequations}
\setlength{\abovedisplayskip}{3pt}
\setlength{\belowdisplayskip}{3pt}
    \tilde{\psi}_m(\mathbf{x}_\mathrm{mid}) = \Psi_\Gamma^-(\mathbf{x}_\mathrm{mid}, \mathbf{u}_m),\quad \mathbf{u}_{m}\cdot \mathbf{n}_{C,\mathbf{x}_{\mathrm{mid}}}<0,\  m\in\mathcal{M}_{\mathrm{2d}},
\end{equation}
where $\mathbf{n}_{C,\mathbf{x}_{\mathrm{mid}}}$ represents the outer normal direction of cell $C$ at $\mathbf{x}_{\mathrm{mid}}$.

Using $\{\alpha_{C}^{(k)}\}_{k\in \mathcal{V}, C\in\mathcal{C} }$ as unknowns, boundary conditions \cref{eqn:boundaryequations} and interface conditions \cref{eqn:interfaceequations} as constraints, we end up with the following linear system:
\begin{equation}\label{eq:sec2:19}
    A\alpha=b
\end{equation}
where $\alpha=\Big(\alpha_{C}^{(k)}\Big)_{k\in\mathcal{V},C\in\mathcal{C}}$. Clearly, there are $(I-1)I+(I-1)I=2I(I-1)$ internal interfaces, and $4I$ boundary interfaces. They altogether impose $4M\times2I(I-1)+2M\times4I=8MI^{2}$ constraints, which is exactly the same as the number of unknown coefficients. Therefore, the matrix $A\in\mathbb{R}^{8MI^{2}\times8MI^{2}}$ is a square matrix. Besides, the matrix $A$ in \cref{eq:sec2:19} is sparse, and its sparsity pattern is depicted in \cref{fig:MatrixA}. More details can be found in \cite{han2014two}.

\begin{figure}[htbp]
    \setlength{\abovecaptionskip}{0.cm}
    \centering
    \includegraphics[width=0.6\textwidth]{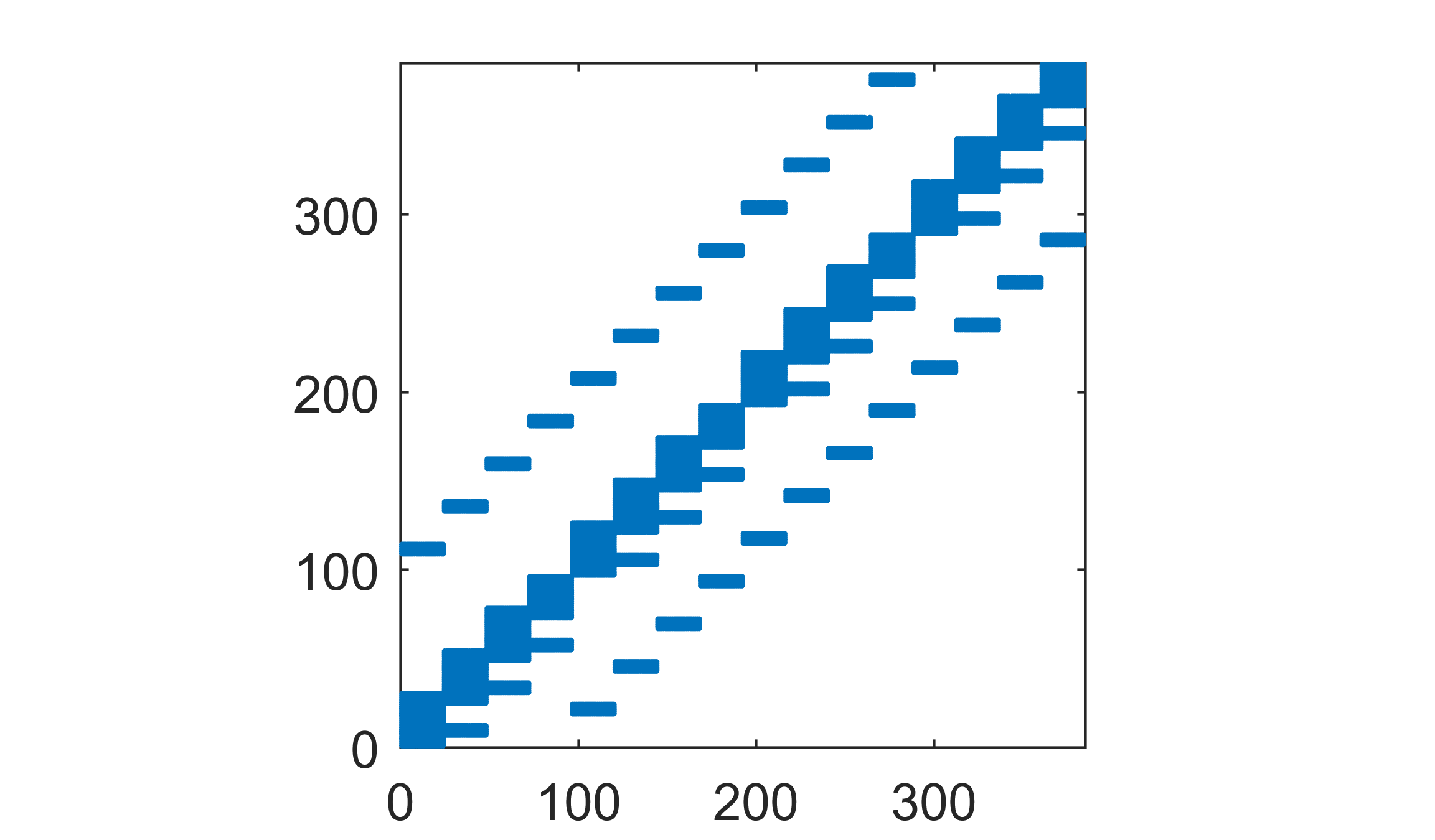}
    \caption{The sparsity pattern of the matrix $A$ for $I\times I=4\times 4$ grid with $M=3$}
    \label{fig:MatrixA}
    \vspace{-0.3cm}
\end{figure}

\begin{remark}
A series of 2D numerical experiments were conducted to validate the performance of TFPS in \cite{han2014two}. These experiments demonstrated that TFPS maintains uniformly second-order accuracy with respect to the mean free path, even in the presence of boundary layers or interface layers. Furthermore, those experiments provided numerical evidence supporting the unique solvability of the linear system \cref{eq:sec2:19}.
\end{remark}

\subsection{Low rank structure in the angular domain} 
Similarly, the low-rank structure in the angular domain of the RTE in x-y geometry can be unveiled by the exponential decay characteristics of basis functions \cref{eq:sec2:6x} and \cref{eq:sec2:6y}.

Denote the center of cell $C$ as $\mathbf{x}_C$. The quantity of interest is the angular flux at the cell center, $\psi(\mathbf{x}_C)$ for $C\in\mathcal{C}$. As in the 1D case, some functions in \cref{eq:sec2:6x} and \cref{eq:sec2:6y} decay rapidly, having negligible impact on $\psi(\mathbf{x}_C)$. By removing these functions, we can construct a compressed solution space. To be more specific, we introduce a threshold $\delta>0$, and for $C\in\mathcal{C}$, define the sets \begin{equation}\label{eq:sec2:26}
        \mathcal{V}_{\delta,C}=\big\{k\big|\Vert\phi_{C}^{(k)}(\mathbf{x}_C)\Vert_{\infty}>\delta\big\}
        =\big\{k\big|\exp\{-\frac{1}{2}\vert\lambda_{C}^{(k)}\vert\sigma_{T,C}h\}>\delta\big\},
\end{equation}
and $\bar{\mathcal{V}}_{\delta,C}=\mathcal{V}/\mathcal{V}_{\delta,C}$.

Then, for any $k\in\bar{\mathcal{V}}_{\delta,C}$,  $\phi_{C}^{(k)}(\mathbf{x}_C)<\delta$, making it negligible. By selecting basis functions in $\mathcal{V}_{\delta,C}$, we have the following compressed solution space:
\[
\setlength{\abovedisplayskip}{3pt}
\setlength{\belowdisplayskip}{3pt}
\mathcal{F}_{\delta}=\Big\{f=\sum_{C\in\mathcal{C}}\big(\sum_{k\in \mathcal{V}_{\delta,C}}\beta_{C}^{(k)}\phi_{C}^{(k)}\big)+\phi_{C}^{s})\Big| \beta_{C}^{(k)}\in\mathbb{R}\Big\}.
\]
According to the definition of $\mathcal{V}_{\delta,C}$ and $\bar{\mathcal{V}}_{\delta,C}$, $\big\{[f(\mathbf{x}_C)]_{C\in\mathcal{C}}\big|f\in\mathcal{F}_{\delta}\big\}$  should approximates $\big\{[f(\mathbf{x}_C)]_{C\in\mathcal{C}}\big|f\in\mathcal{F}\big\}$ well, where $[f(\mathbf{x}_C)]_{C\in\mathcal{C}}$ represents a vector of dimension $4MI^{2}$ that indicates the values of $f$ at the cell centers. This approximation allows us to select $\{\phi_{C}^{(k)}\}_{k\in \mathcal{V}_{\delta,C}}$ $(C\in\mathcal{C})$ as basis functions and seek an approximation of $\tilde{\psi}$ (the TFPS solution) in function space $\mathcal{F}_{\delta}$ which can accurately approximate $\tilde{\psi}$ at cell centers. This process highlights the low-rank structure in angular domain of RTE in x-y geometry.

\begin{remark}
In various regimes, a low-rank structure exists in the solution space of 2D RTE, as evidenced by the reduced number of basis functions given the threshold $\delta$. In the diffusive regime, asymptotic analysis reveals a significant gap between the absolute values of certain eigenvalues, $\lambda_{C}^{(2M)}$, $\lambda_{C}^{(2M+1)}$, $\lambda_{C}^{(6M)}$, $\lambda_{C}^{(6M+1)}$ and others. Therefore, in practice, it is often sufficient to select only these four basis functions: $\phi_{C}^{(2M)}$, $\phi_{C}^{(2M+1)}$, $\phi_{C}^{(6M)}$, and $\phi_{C}^{(6M+1)}$. In the transition region from diffusion regime to transport regime, the number of basis functions in each spatial cell gradually increases from 4 to $8M$. These findings can be further supported by experimental results presented in \cref{tab:Example1_basis} and \cref{fig:Example2_basis} in the numerical experiments section (\cref{sec:experiments}).
\end{remark}

\section{Adaptive TFPS: An adaptive angular domain compression scheme}
\label{sec:Adaptive TFPS}
In this section, we introduce Adaptive TFPS, a method that enables adaptive angular compression for multiscale RTEs. It selects local basis functions $\{\phi_{C}^{(k)}\}_{k\in \mathcal{V}_{\delta,C}}$ as in  \cref{eq:sec2:26} from the complete TFPS basis set, based on the accuracy threshold $\delta$ and the cell's optical properties, such as the scattering ratio $\gamma$, anisotropy factor $g$, and optical thickness $\sigma_{T}h$. These reduced basis functions identify the dominant modes in the velocity domain, forming a compressed solution space $\mathcal{F}_{\delta}$, which also includes a special solution defined in  \cref{eqn:specialsolution}. We denote $\tilde{\psi}_{\delta}$ as the approximate solution to \cref{eq:sec2:1}, \cref{eq:sec2:25}, and \cref{eq:sec1:2} using Adaptive TFPS, which has the following expression:
\begin{equation}\label{eq:sec3:42}
\setlength{\abovedisplayskip}{3pt}
\setlength{\belowdisplayskip}{3pt}
    \tilde{\psi}_{\delta}=\sum_{C\in\mathcal{C}}(\sum_{k\in \mathcal{V}_{\delta,C}}\tilde{\alpha}_{\delta,C}^{(k)}\phi_{C}^{(k)})+\phi_{C}^{s}
\end{equation}
where $\tilde{\alpha}_{\delta,C}^{(k)}$ is the currently unknown coefficients for basis function $\phi_{C}^{(k)}$. Clearly, we wish $\tilde{\psi}_{\delta}$ to be a good approximation of the truncated TFPS solution $\tilde{\psi}$, denoted as $\tilde{\psi}^{*}$:
\begin{equation}\label{eq:sec3:43}
\setlength{\abovedisplayskip}{3pt}
\setlength{\belowdisplayskip}{3pt}
    \tilde{\psi}^{*}=\sum_{C\in\mathcal{C}}(\sum_{k\in \mathcal{V}_{\delta,C}}\alpha_{C}^{(k)}\phi_{C}^{(k)})+\phi_{C}^{s}
\end{equation}

To construct the corresponding constrains for $\tilde{\alpha}_{\delta,C}^{(k)}$, we have to define the interface continuity and boundary conditions for this compressed solution space, analogous to the conditions in \cref{eqn:interfaceequations} and \cref{eqn:boundaryequations} for the full TFPS space. For a concise presentation, we need to first introduce the following partitions of the index set $\mathcal{V}_C=\mathcal{V}$, and further for the compressed index sets $\mathcal{V}_{\delta,C}$ and $\bar{\mathcal{V}}_{\delta,C}$. 

For any cell $C\in\mathcal{C}$, $k\in \mathcal{V}_C$, given an interior interface $\mathfrak{i}\in\mathcal{I}\cap C$, we call a basis $\phi_{C}^{(k)}$ is \textit{centered} on $\mathfrak{i}$ if it achieves maximum on $\mathfrak{i}$ (in fact, on the midpoint of $\mathfrak{i}$ by definitions \cref{eq:sec2:6x} and \cref{eq:sec2:6y}). Hence, we can partition $\mathcal{V}_{C}$, the index of set of basis function localized within cell $C$, as the following three parts, 
\begin{itemize}
    \item $\mathcal{V}_{C}^{\mathfrak{i}}$: The index set of basis functions \textit{centered} on $\mathfrak{i}$.
    \item $\mathcal{V}_C^{\mathfrak{i}^\parallel}$: The index set of basis functions \textit{centered} on interfaces opposite to $\mathfrak{i}$ (or simply say \textit{opposite to $\mathfrak{i}$}).
    \item $\mathcal{V}_C^{\mathfrak{i}^\perp}$: The index set of basis functions \textit{centered} on interfaces vertical to $\mathfrak{i}$ (or simply say \textit{vertical to $\mathfrak{i}$}).
\end{itemize}

\textbf{Example}: If $\mathfrak{i}$ is the left edge of $C$, then $\mathcal{V}_{C}^{\mathfrak{i}}=\{1, 2, \cdots, 2M\}$, $\mathcal{V}_{C}^{i^\parallel}=\{2M+1,2M+2,\dots,4M\}$ and $\mathcal{V}_{C}^{i^\perp}=\{4M+1,4M+2,\dots,8M\}$, according to the definition of basis in \cref{eq:sec2:6x} and \cref{eq:sec2:6y}. 

Furthermore, $\mathcal{V}_{\delta, C}$, the index set of reduced basis functions localized within cell $C$ can be accordingly partitioned as basis centered on $\mathfrak{i}$, opposite to $\mathfrak{i}$, and vertical to $\mathfrak{i}$ are denoted as
\[
\mathcal{V}_{\delta,C}^{\mathfrak{i}}=\mathcal{V}_{C}^{\mathfrak{i}}\cap \mathcal{V}_{\delta,C},\quad \mathcal{V}_{\delta,C}^{\mathfrak{i}^\parallel}=\mathcal{V}_{C}^{i^\parallel}\cap \mathcal{V}_{\delta,C},\quad \mathcal{V}_{\delta,C}^{\mathfrak{i}^\perp}=\mathcal{V}_{C}^{\mathfrak{i}^\perp}\cap \mathcal{V}_{\delta,C}.
\]

Similiarly, the index set of unselected basis functions localized within cell $C$ corresponding to centered on $\mathfrak{i}$, parallel to $\mathfrak{i}$, and vertical to $\mathfrak{i}$ are denoted as
\[
\bar{\mathcal{V}}^{\mathfrak{i}}_{\delta,C}=\mathcal{V}_{C}^{\mathfrak{i}}/\mathcal{V}_{\delta,C}^{\mathfrak{i}},\quad \bar{\mathcal{V}}_{\delta,C}^{i^\parallel}=\mathcal{V}_{C}^{i^\parallel}/\mathcal{V}_{\delta,C}^{i^\parallel},\quad \bar{\mathcal{V}}_{\delta,C}^{i^\perp}=\mathcal{V}_{C}^{i^\perp}/\mathcal{V}_{\delta,C}^{i^\perp}.
\]

\subsection{Interface condition and boundary condition}
\label{subsec:condition}

To obtain a solution in the compressed solution space, we need to formulate the interface continuity conditions at the interior grid points and the boundary conditions at the boundary grid points, respectively.

\subsubsection{Interface condition}

We now construct the new interface condition at interior grid points. Consider $\mathfrak{i}\in\mathcal{I}$ as the interface between two cells $C_{+}$ and $C_{-}$ such that $\mathfrak{i}=C_{+}\cap C_{-}$, and let $\mathbf{x}_{\mathrm{mid}}$ be the middle point at $\mathfrak{i}$, as is shown in \cref{fig:interior_grid_point}.

The original interface condition \cref{eqn:boundaryequations} in TFPS requires the continuity of the approximate solution $\tilde{\psi}$ at $\mathbf{x}_{\mathrm{mid}}$, namely,
\begin{equation}
\sum_{k\in \mathcal{V}_{C_{-}}}\alpha_{C_{-}}^{(k)}\phi_{C_{-}}^{(k)}(\mathbf{x}_{\mathrm{mid}})+\phi_{C_{-}}^{s}(\mathbf{x}_{\mathrm{mid}})
=\sum_{k\in \mathcal{V}_{C_{+}}}\alpha_{C_{+}}^{(k)}\phi_{C_{+}}^{(k)}(\mathbf{x}_{\mathrm{mid}})+\phi_{C_{+}}^{s}(\mathbf{x}_{\mathrm{mid}}).
\label{eqn:approxinterface}
\end{equation}

We notice that 
\[
\mathcal{V}_{\delta,C_{-}}\cup\bar{\mathcal{V}}_{\delta,C_{-}}^\mathfrak{i}\cup \bar{\mathcal{V}}_{\delta,C_{-}}^{\mathfrak{i}^\perp} \cup \bar{\mathcal{V}}_{\delta,C_{-}}^{\mathfrak{i}^\parallel}=\mathcal{V}_{C_{-}}=\mathcal{V}_{C_{+}}=\mathcal{V}_{\delta,C_{-}}\cup\bar{\mathcal{V}}_{\delta,C_{+}}^\mathfrak{i}\cup \bar{\mathcal{V}}_{\delta,C_{+}}^{\mathfrak{i}^\perp} \cup \bar{\mathcal{V}}_{\delta,C_{+}}^{\mathfrak{i}^\parallel}.
\]
The basis functions $\{\phi_{C_{-}}^{(k)}\}_{k\in \bar{\mathcal{V}}_{\delta,C_{-}}^{\mathfrak{i}^\perp}\cup\bar{\mathcal{V}}_{\delta,C_{-}}^{\mathfrak{i}^\parallel}}$ and $\{\phi_{C_{+}}^{(k)}\}_{k\in \bar{\mathcal{V}}_{\delta,C_{+}}^{\mathfrak{i}^\perp}\cup\bar{\mathcal{V}}_{\delta,C_{+}}^{\mathfrak{i}^\parallel}}$ decay rapidly and are not centered on $\mathfrak{i}$, so their contribution to the function value at $\mathbf{x}_\mathrm{mid}$ is small enough to be controlled.
Therefore we can decompose the contributions in the sum in \cref{eqn:approxinterface} into subgroups of $\mathcal{V}_{C_{\pm}}$,
\begin{equation}\label{eq:sec4:1}
    \begin{aligned}
    &\sum_{k\in \mathcal{V}_{\delta,C_{-}}\cup \bar{\mathcal{V}}_{\delta,C_{-}}^{\mathfrak{i}}}\alpha_{C_{-}}^{(k)}\xi_{C_{-}}^{(k)}\zeta_{C_{-}}(\mathbf{x}_\mathrm{mid},k)
    +\frac{q_{C_{-}}}{\sigma_{a,C_{-}}}\mathbf{1}\\
    =&\sum_{k\in\mathcal{V}_{\delta,C_{+}}\cup \bar{\mathcal{V}}_{\delta,C_{+}}^{\mathfrak{i}}}\alpha_{C_{+}}^{(k)}\xi_{C_{+}}^{(k)}\zeta_{C_{+}}(\mathbf{x}_\mathrm{mid},k)
    +\frac{q_{C_{+}}}{\sigma_{a,C_{+}}}\mathbf{1}+\tau_{\mathbf{x}_\mathrm{mid}}\\
    \end{aligned}
\end{equation}
with
\begin{equation}\label{eq:sec4:5}
\setlength{\abovedisplayskip}{3pt}
\setlength{\belowdisplayskip}{3pt}
    \tau_{\mathbf{x}_\mathrm{mid}}=-\sum_{k\in \bar{\mathcal{V}}_{\delta,C_{-}}^{i^\parallel}\cup \bar{\mathcal{V}}_{\delta,C_{-}}^{i^\perp}}\alpha_{C_{-}}^{(k)}\xi_{C_{-}}^{(k)}\zeta_{C_{-}}(\mathbf{x}_\mathrm{mid},k)+\sum_{k\in \bar{\mathcal{V}}_{\delta,C_{+}}^{\mathfrak{i}^\parallel}\cup \bar{\mathcal{V}}_{\delta,C_{+}}^{i^\perp}}\alpha_{C_{+}}^{(k)}\xi_{C_{+}}^{(k)}\zeta_{C_{+}}(\mathbf{x}_\mathrm{mid},k).
\end{equation}
Here $\zeta_{C_{-}}(\mathbf{x}_\mathrm{mid},k)$ and $\zeta_{C_{-}}(\mathbf{x}_\mathrm{mid},k)$ are the exponential components of $\phi_{C_{-}}^{(k)}(\mathbf{x}_\mathrm{mid})$ and $\phi_{C_{+}}^{(k)}(\mathbf{x}_\mathrm{mid})$ so that $\phi_{C_{-}}^{(k)}=\xi_{C_{-}}^{(k)}\zeta_{C_{-}}(\mathbf{x}_\mathrm{mid},k)$ and $\phi_{C_{+}}^{(k)}=\xi_{C_{+}}^{(k)}\zeta_{C_{+}}(\mathbf{x}_\mathrm{mid},k)$. 

Notice that, according to the selection rule, $\zeta_{C_{-}}(\mathbf{x}_\mathrm{mid},k)\leq \delta$ holds for any $k\in \bar{\mathcal{V}}_{\delta,C_{-}}^{i^\parallel}\cup \bar{\mathcal{V}}_{\delta,C_{-}}^{i^\perp}$ and $\zeta_{C_{+}}(\mathbf{x}_\mathrm{mid},k)\leq \delta$ holds for any $k\in \bar{\mathcal{V}}_{\delta,C_{+}}^{\mathfrak{i}^\parallel}\cup \bar{\mathcal{V}}_{\delta,C_{+}}^{i^\perp}$. Therefore, $\tau_{\mathbf{x}_\mathrm{mid}}$ is a sufficiently small quantity that can be neglected.

Upon analyzing equation \cref{eq:sec4:1}, it is clear that it incorporates not only the controllable small quantity $\tau_{\mathbf{x}_\mathrm{mid}}$, but also encodes information about the coefficients of unselected basis functions $\{\alpha_{C_{-}}^{(k)}\}_{k\in \bar{\mathcal{V}}_{\delta,C_{-}}^{\mathfrak{i}}}$ and $\{\alpha_{C_{+}}^{(k)}\}_{k\in \bar{\mathcal{V}}_{\delta,C_{+}}^{\mathfrak{i}}}$. However, these coefficients only pertain to interface layers, which are not relevant to our current focus. In order to eliminate this part of information, we introduce the following assumption and define the spaces $U_{\mathfrak{i}}$, $U_{\delta,\mathfrak{i}}$,  $\bar{U}_{\delta,\mathfrak{i}}$, and the projection operator $\mathrm{Proj}_{\delta,\mathfrak{i}}$.
\begin{assumption}\label{assump:linear_independent1}
For $\mathfrak{i}=C_{-}\cap C_{+}\in\mathcal{I}_{i}$, the vectors $\{\xi_{C_{-}}^{(k_-)},\xi_{C_{+}}^{(k_+)}\}_{k_-\in\mathcal{V}_{C_{-}}^{\mathfrak{i}},k_+\in\mathcal{V}_{C_{+}}^{\mathfrak{i}}}$ are linear independent.
\end{assumption}

\begin{remark}
According to the definitions of $\xi_{C_{-}}^{(k_{-})}$ and $\xi_{C_{+}}^{(k_{+})}$ provided in \cref{appendix:basis}, this assumption holds true if the optical properties within cell $C_{-}$ and cell $C_{+}$ are the same. Besides, if the optical properties of cell $C_{-}$ and cell $C_{+}$ are similar, the set $\{\xi_{C_{-}}^{(k_-)},\xi_{C_{+}}^{(k_+)}\}_{k_-\in\mathcal{V}_{C_{-}}^{\mathfrak{i}},k_+\in\mathcal{V}_{C_{+}}^{\mathfrak{i}}}$  should remain linearly independent, based on   perturbation arguments, which partly justifies the validity of \cref{assump:linear_independent1}. 
The assumption is justified numerically through experiments shown in supplementary material \ref{appendix:independence1}.
\end{remark}

We define $U_{\mathfrak{i}}$, $U_{\delta,\mathfrak{i}}$ and $\bar{U}_{\delta,\mathfrak{i}}$ as:
\[
U_{\delta,\mathfrak{i}}=\mathrm{span}\{\xi_{C}^{(k)}\}_{\mathfrak{i}\in C, k\in\mathcal{V}_{\delta,C}^{\mathfrak{i}}},\ \bar{U}_{\delta,\mathfrak{i}}=\mathrm{span}\{\xi_{C}^{(k)}\}_{\mathfrak{i}\in C, k\in\bar{\mathcal{V}}_{\delta,C}^{\mathfrak{i}}},\  U_{\mathfrak{i}}=\mathrm{span}\{\xi_{C}^{(k)}\}_{\mathfrak{i}\in C,k\in\mathcal{V}_{C}^{\mathfrak{i}}}
\]
Denote the index sets $\mathcal{V}^{\mathfrak{i}}$, $\mathcal{V}_{\delta}^{\mathfrak{i}}$, and $\bar{\mathcal{V}}_{\delta}^{\mathfrak{i}}$ as:
\[
\setlength{\abovedisplayskip}{3pt}
\setlength{\belowdisplayskip}{3pt}
\mathcal{V}_{\delta}^{\mathfrak{i}}=\mathcal{V}_{\delta,C_{-}}^{\mathfrak{i}}\cup \mathcal{V}_{\delta,C_{+}}^{\mathfrak{i}},\quad \bar{\mathcal{V}}_{\delta}^{\mathfrak{i}}=\bar{\mathcal{V}}_{\delta,C_{-}}^{\mathfrak{i}}\cup \bar{\mathcal{V}}_{\delta,C_{+}}^{\mathfrak{i}},\quad \mathcal{V}^{\mathfrak{i}}=\mathcal{V}_{C_{-}}^{\mathfrak{i}}\cup \mathcal{V}_{C_{+}}^{\mathfrak{i}}.
\]
and the orthonormal basis for $U_{\delta,\mathfrak{i}}$ and $\bar{U}_{\delta,\mathfrak{i}}$ as $\{\chi_{\mathfrak{i}}^{(k)}\}_{k\in \mathcal{V}_{\delta}^{\mathfrak{i}}}$ and $\{\chi_{\mathfrak{i}}^{(k)}\}_{k\in \bar{\mathcal{V}}_{\delta}^{\mathfrak{i}}}$ respectively.
Then, $U_{\mathfrak{i}}$, $U_{\delta,\mathfrak{i}}$ and $\bar{U}_{\delta,\mathfrak{i}}$ can be rewrite as follows:
\[
U_{\delta,\mathfrak{i}}=\mathrm{span}\{\chi_{\mathfrak{i}}^{(k)}\}_{k\in \mathcal{V}_{\delta}^{\mathfrak{i}}},\quad \bar{U}_{\delta,\mathfrak{i}}=\mathrm{span}\{\chi_{\mathfrak{i}}^{(k)}\}_{k\in \bar{\mathcal{V}}_{\delta}^{\mathfrak{i}}},\quad U_{\mathfrak{i}}=\mathrm{span}\{\chi_{\mathfrak{i}}^{(k)}\}_{k\in \mathcal{V}^{\mathfrak{i}}}
\]

According to \cref{assump:linear_independent1}, $\dim(U_{\mathfrak{i}})=4M$. Therefore, for any $l\in\mathbb{R}^{4M}$, we can express it linearly in terms of the orthonomal basis $\{\chi_{\mathfrak{i}}^{(k)}\}_{k\in\mathcal{V}^{\mathfrak{i}}}$ as: 
\begin{equation}\label{eq:sec4:6}    l=\sum_{k\in\mathcal{V}^{\mathfrak{i}}}\langle l,\chi_{\mathfrak{i}}^{(k)}\rangle_{\mathfrak{i}}\chi_{\mathfrak{i}}^{(k)}.
\end{equation}
where $\langle l,\chi_{\mathfrak{i}}^{(k)}\rangle_{\mathfrak{i}}$ denotes the coefficient in the linear representation. Subsequently, the \textit{projection} $\mathrm{Proj}_{\delta,\mathfrak{i}}$ from $\mathbb{R}^{4M}$ (or equivalently, $U_{\mathfrak{i}}$) to its subspace $U_{\delta,\mathfrak{i}}$ can be defined as follows:
\begin{equation}\label{eq:sec4:14}
    \forall l\in\mathbb{R}^{4M},\quad \mathrm{Proj}_{\delta,\mathfrak{i}}l=\sum_{k\in \mathcal{V}_{\delta}^{\mathfrak{i}}}\langle l,\chi_{\mathfrak{i}}^{(k)}\rangle_{\mathfrak{i}}\chi_{\mathfrak{i}}^{(k)}.
\end{equation}

By performing such projection on  \cref{eq:sec4:1}, we obtain the following equation,
\begin{equation}\label{eq:sec4:8}
\setlength{\abovedisplayskip}{3pt}
\setlength{\belowdisplayskip}{3pt}
    \begin{aligned}
    &\sum_{k\in \mathcal{V}_{\delta,C_{-}}}\alpha_{C_{-}}^{(k)}\big(\sum_{k'\in\mathcal{V}_{\delta}^{\mathfrak{i}}}\langle \xi_{C_{-}}^{(k)},\chi_{\mathfrak{i}}^{(k')}\rangle_{\mathfrak{i}}\chi_{\mathfrak{i}}^{(k')}\big)\zeta_{C_{-}}(\mathbf{x}_\mathrm{mid},k)+\frac{q_{C_{-}}}{\sigma_{a,C_{-}}}\big(\sum_{k'\in\mathcal{V}_{\delta}^{\mathfrak{i}}}\langle \mathbf{1},\chi_{\mathfrak{i}}^{(k')}\rangle_{\mathfrak{i}}\chi_{\mathfrak{i}}^{(k')}\big)\\
    &=\sum_{k\in\mathcal{V}_{\delta,C_{+}}}\alpha_{C_{+}}^{(k)}\big(\sum_{k'\in\mathcal{V}_{\delta}^{\mathfrak{i}}}\langle\xi_{C_{+}}^{(k)},\chi_{\mathfrak{i}}^{(k')}\rangle_{\mathfrak{i}}\chi_{\mathfrak{i}}^{(k')}\big)\zeta_{C_{+}}(\mathbf{x}_\mathrm{mid},k)
    +\frac{q_{C_{+}}}{\sigma_{a,C_{+}}}\big(\sum_{k'\in\mathcal{V}_{\delta}^{\mathfrak{i}}}\langle\mathbf{1},\chi_{\mathfrak{i}}^{(k')}\rangle_{\mathfrak{i}}\chi_{\mathfrak{i}}^{(k')}\big)\\
    &+\sum_{k'\in\mathcal{V}_{\delta}^{\mathfrak{i}}}\langle\tau_{\mathbf{x}_\mathrm{mid}},\chi_{\mathfrak{i}}^{(k')}\rangle_{\mathfrak{i}}\chi_{\mathfrak{i}}^{(k')}.\\
    \end{aligned}
\end{equation}

Exchanging the summation with respect to $k'$ and $k$, and since $\{\chi_{\mathfrak{i}}^{(k')}\}_{k'\in \mathcal{V}_{\delta}^{\mathfrak{i}}}$ are linearly independent, we obtain an equation for the coefficient corresponding to $\chi_{\mathfrak{i}}^{(k')}$ for all $k'\in \mathcal{V}_{\delta}^{\mathfrak{i}}$, namely,
\begin{equation}\label{eq:sec4:10}
\setlength{\abovedisplayskip}{3pt}
\setlength{\belowdisplayskip}{3pt}
    \begin{aligned}
    &\sum_{k\in \mathcal{V}_{\delta,C_{-}}}\alpha_{C_{-}}^{(k)}\langle \xi_{C_{-}}^{(k)},\chi_{\mathfrak{i}}^{(k')}\rangle_{\mathfrak{i}}\zeta_{C_{-}}(\mathbf{x}_\mathrm{mid},k)+\frac{q_{C_{-}}}{\sigma_{a,C_{-}}}\langle \mathbf{1},\chi_{\mathfrak{i}}^{(k')}\rangle_{\mathfrak{i}}\\
    &=\sum_{k\in\mathcal{V}_{\delta,C_{+}}}\alpha_{C_{+}}^{(k)}\langle\xi_{C_{+}}^{(k)},\chi_{\mathfrak{i}}^{(k')}\rangle_{\mathfrak{i}}\zeta_{C_{+}}(\mathbf{x}_\mathrm{mid},k)
    +\frac{q_{C_{+}}}{\sigma_{a,C_{+}}}\langle\mathbf{1},\chi_{\mathfrak{i}}^{(k')}\rangle_{\mathfrak{i}}+\langle\tau_{\mathbf{x}_\mathrm{mid}},\chi_{\mathfrak{i}}^{(k')}\rangle_{\mathfrak{i}}.\\
    \end{aligned}
\end{equation}

With the definition of $\tilde{\psi}^{*}$ as shown in \cref{eq:sec3:43}, we can rewrite equation \cref{eq:sec4:10} as follows: for all $k'\in \mathcal{V}_{\delta}^{\mathfrak{i}}$,
\begin{equation}
\setlength{\abovedisplayskip}{3pt}
\setlength{\belowdisplayskip}{3pt}
    \langle \tilde{\psi}^{*}|_{C_{-}}(\mathbf{x}_\mathrm{mid}),\chi_{\mathfrak{i}}^{(k')}\rangle_{\mathfrak{i}}=\langle\tilde{\psi}^{*}|_{C_{+}}(\mathbf{x}_\mathrm{mid}),\chi_{\mathfrak{i}}^{(k')}\rangle_{\mathfrak{i}}+\langle\tau_{\mathbf{x}_\mathrm{mid}},\chi_{\mathfrak{i}}^{(k')}\rangle_{\mathfrak{i}}.
\end{equation}
Therefore, as an approximation to $\tilde{\psi}^{*}$, we require that the Adaptive TFPS solution $\tilde{\psi}_{\delta}$ satisfies the following revised interface condition at $\mathbf{x}_\mathrm{mid}$: for all $k'\in \mathcal{V}_{\delta}^{\mathfrak{i}}$:
\begin{equation}\label{eq:sec4:11}
\setlength{\abovedisplayskip}{3pt}
\setlength{\belowdisplayskip}{3pt}
    \langle \tilde{\psi}_{\delta}|_{C_{-}}(\mathbf{x}_\mathrm{mid}),\chi_{\mathfrak{i}}^{(k')}\rangle_{\mathfrak{i}}=\langle\tilde{\psi}_{\delta}|_{C_{+}}(\mathbf{x}_\mathrm{mid}),\chi_{\mathfrak{i}}^{(k')}\rangle_{\mathfrak{i}}.
\end{equation}
\begin{remark}

The asymptotic analysis of the RTE \cite{jin2009uniformly,wang2022uniform} reveals that if cell $C_{-}$ and $C_{+}$ is in the diffusive regime, the eigenvectors associated with eigenvalues of smallest magnitude are nearly parallel to the vector $e$, where all elements are equal to 1. According to the selection rule of basis functions, these eigenvectors will be selected into $\mathcal{V}_{\delta,C_{-}}^{\mathfrak{{i}}}$ and $\mathcal{V}_{\delta,C_{+}}^{\mathfrak{i}}$, resulting in a high degree of linear dependence between  $\{\xi_{C_{-}}^{(k_{-})},\xi_{C_{+}}^{(k_{+})}\}_{k_{-}\in\mathcal{V}_{\delta,C_{-}}^{\mathfrak{{i}}},k_{+}\in\mathcal{V}_{\delta,C_{+}}^{\mathfrak{{i}}}}$. Therefore, to maintain numerical stability in computations, we opt to use the orthonormal basis $\{\chi_{\mathfrak{i}}^{(k)}\}_{k\in\mathcal{V}_{\delta}^{\mathfrak{i}}}$ instead of $\{\xi_{C_{-}}^{(k_{-})},\xi_{C_{+}}^{(k_{+})}\}_{k_{-}\in\mathcal{V}_{\delta,C_{-}}^{\mathfrak{{i}}},k_{+}\in\mathcal{V}_{\delta,C_{+}}^{\mathfrak{{i}}}}$ as basis for $U_{\delta,\mathfrak{i}}$.

\end{remark}
\subsubsection{Boundary condition}
We move on to define the new boundary condition. We now consider the edge $\mathfrak{i}$ of cell $C$ to be at the physical boundary, such that $\mathfrak{i}= \partial\Omega\cap C$ and let $\mathbf{x}_{\mathrm{mid}}$ be the middle point of $\mathfrak{i}$, as shown in \cref{fig:boundary_grid_point}.

The original boundary condition \cref{eqn:boundaryequations} requires the consistency of inflow angular flux with inflow boundary condition. We can express it in vector form as follows:
\begin{equation}
\setlength{\abovedisplayskip}{3pt}
\setlength{\belowdisplayskip}{3pt}
    \sum_{k\in \mathcal{V}}\alpha_{C}^{(k)}\phi_{C,\mathfrak{i}-}^{(k)}(\mathbf{x}_{\mathrm{mid}})+\phi_{C,\mathfrak{i}-}^{s}(\mathbf{x}_{\mathrm{mid}})=\Psi_{\Gamma,\mathfrak{i}-}(\mathbf{x}_{\mathrm{mid}}).
\end{equation}
Here, the subscript $\mathfrak{i}-$ indicates that the new vector or vector-valued function includes all inflow components (with respect to interface $\mathfrak{i}$) of the original vector or vector-valued function. For example, $\phi_{C,\mathfrak{i}-}^{(k)}(\mathbf{x}_{\mathrm{mid}})$ is a vector that includes all $\phi_{C,m}^{(k)}(\mathbf{x}_{\mathrm{mid}})$ satisfying $\mathbf{u}_{m}\cdot \mathbf{n}_{C,\mathbf{x}_{\mathrm{mid}}}<0$.

Notice that $\{\phi_{C,\mathfrak{i}-}^{(k)}\}_{k\in \bar{\mathcal{V}}_{\delta,C}^{\mathfrak{i}^\parallel}\cup\bar{\mathcal{V}}_{\delta,C}^{\mathfrak{i}^\perp}}$ decay rapidly and are not centered on $\mathbf{x}_{\mathrm{mid}}$, so that their contribution to the function value at $\mathbf{x}_{\mathrm{mid}}$ is small enough to  be controlled. Then we have:
\begin{equation}\label{eq:sec4:13}
\begin{aligned}
\setlength{\abovedisplayskip}{3pt}
\setlength{\belowdisplayskip}{3pt}
    \sum_{k\in \mathcal{V}_{\delta,C}\cup\bar{\mathcal{V}}_{\delta,C}^{\mathfrak{i}}}\alpha_{C}^{(k)}\xi_{C,\mathfrak{i}-}^{(k)}\zeta_{C}(\mathbf{x}_{\mathrm{mid}},k)+\frac{q_{C}}{\sigma_{a,C}}\mathbf{1}=\Psi_{\Gamma,\mathfrak{i}-}(\mathbf{x}_{\mathrm{mid}})+\tau(\mathbf{x}_{\mathrm{mid}})
    \end{aligned}
\end{equation}
with the small quantity $\tau(\mathbf{x}_{\mathrm{mid}})$ defined as
\begin{equation*}
\setlength{\abovedisplayskip}{3pt}
\setlength{\belowdisplayskip}{3pt}
    \tau(\mathbf{x}_{\mathrm{mid}})=-\sum_{k\in\bar{\mathcal{V}}_{\delta,C}^{\mathfrak{i}^\parallel}\cup \bar{\mathcal{V}}_{\delta,C}^{\mathfrak{i}^\perp}}\alpha_{C}^{(k)}\xi_{C,\mathfrak{i}-}^{(k)}\zeta_{C}(\mathbf{x}_{\mathrm{mid}},k)
\end{equation*}
Observing \cref{eq:sec4:13}, we can notice that, besides the small controllable quantity  $\tau(\mathbf{x}_{\mathrm{mid}})$, it also incorporates information regarding the coefficients of unselected basis functions $\{\alpha_{C}^{(k)}\}_{k\in \bar{\mathcal{V}}_{\delta,C}^{\mathfrak{i}}}$. To eliminate this part of information, we introduce the following assumption.

\begin{assumption}\label{assump:linear_independent2}
For $\mathfrak{i}=C\cap\partial\Omega\in\mathcal{I}_{b}$, the vectors $\{\xi_{C,\mathfrak{i}-}^{(k)}\}_{k\in\mathcal{V}_{C}^{\mathfrak{i}}}$ are linearly independent. 
\end{assumption}

\begin{remark}
Notice that $\{\xi_{C}^{(k)}\}_{k\in\mathcal{V}_{C}^{\mathfrak{i}}}$ are linearly independent because they are eigenvectors of the same matrix. However, as subvectors of $\xi_{C}^{(k)}$ ($\xi_{C}^{(k)}\in\mathbb{R}^{4M}$,  $\xi_{C,\mathfrak{i}-}^{(k)}\in\mathbb{R}^{2M}$), the linear independence of $\{\xi_{C,\mathfrak{i}-}^{(k)}\}_{k\in\mathcal{V}_{C}^{\mathfrak{i}}}$ cannot be simply inferred. When \( g = 0 \), we can easily establish the relationship between the eigenvectors in the \( x \)-\( y \) geometry case \cite{han2014two} and the eigenvectors in the slab geometry case \cite{jin2009uniformly}, as their basis functions are derived similarly. Consequently, the linear independence of \(\{\xi_{C,\mathfrak{i}-}^{(k)}\}_{k \in \mathcal{V}_{C}^{\mathfrak{i}}}\) can be inferred from the linear independence of the eigenvectors in the slab geometry. This observation partially justifies \cref{assump:linear_independent2}. 
Additionally, the assumption can be numerically validated through experiments presented in supplementary material \ref{appendix:independence2}.
\end{remark}

Similarly, we can also define the index set $\mathcal{V}_{\delta}^{\mathfrak{i}}$, $\bar{\mathcal{V}}_{\delta}^{\mathfrak{i}}$, $\mathcal{V}^{\mathfrak{i}}$, the vector spaces $U_{\delta,\mathfrak{i}}$, $\bar{U}_{\delta,\mathfrak{i}}$, $U_{\mathfrak{i}}$, their corresponding basis vector $\{\chi_{\mathfrak{i}}^{(k)}\}_{k\in\mathcal{V}_{\delta}^{\mathfrak{i}}}$, $\{\chi_{\mathfrak{i}}^{(k)}\}_{k\in\bar{\mathcal{V}}_{\delta}^{\mathfrak{i}}}$, $\{\chi_{\mathfrak{i}}^{(k)}\}_{k\in\mathcal{V}^{\mathfrak{i}}}$ , and the projection operator $\mathrm{Proj}_{\delta,\mathfrak{i}}$ from $U_{\mathfrak{i}}$ to $U_{\delta,\mathfrak{i}}$ for $\mathfrak{i}\in\mathcal{I}_{\mathrm{b}}$. By performing the projection on Equation \cref{eq:sec4:13}, we obtain the following equation that eliminates the information about coefficients of unselected basis functions in \cref{eq:sec4:13}:

\begin{equation}\label{eq:sec3:12}
\begin{aligned}
   &\sum_{k\in \mathcal{V}_{\delta,C}}\alpha_{C}^{(k)}\big(\sum_{k'\in\mathcal{V}_{\delta}^{\mathfrak{i}}}\langle\xi_{C,\mathfrak{i}-}^{(k)},\chi_{\mathfrak{i}}^{(k')}\rangle_{\mathfrak{i}}\chi_{\mathfrak{i}}^{(k')}\big)\zeta_{C}(\mathbf{x}_{\mathrm{mid}},k)+\frac{q_{C}}{\sigma_{a,C}}\big(\sum_{k'\in\mathcal{V}_{\delta}^{\mathfrak{i}}}\langle\mathbf{1},\chi_{\mathfrak{i}}^{(k')}\rangle_{\mathfrak{i}}\chi_{\mathfrak{i}}^{(k')})\\
    &=\sum_{k'\in\mathcal{V}_{\delta}^{\mathfrak{i}}}\langle\Psi_{\Gamma,\mathfrak{i}-}(\mathbf{x}_{\mathrm{mid}}),\chi_{\mathfrak{i}}^{(k')}\rangle_{\mathfrak{i}}\chi_{\mathfrak{i}}^{(k')}+\sum_{k'\in\mathcal{V}_{\delta}^{\mathfrak{i}}}\langle\tau(\mathbf{x}_{\mathrm{mid}}),\chi_{\mathfrak{i}}^{(k')}\rangle_{\mathfrak{i}}\chi_{\mathfrak{i}}^{(k')}
    \end{aligned}
\end{equation}

Since $\{\chi_{\mathfrak{i}}^{(k')}\}_{k'\in\mathcal{V}_{\delta}^{\mathfrak{i}}}$ are linear independent, the coefficient corresponding to $\chi_{\mathfrak{i}}^{(k')}$ should be same. Then we obtain, for all $k'\in \mathcal{V}_{\delta}^{\mathfrak{i}}$:
\begin{equation}\label{eq:sec4:4}
\begin{aligned}
   &\sum_{k\in \mathcal{V}_{\delta,C}}\alpha_{C}^{(k)}\langle\xi_{C,\mathfrak{i}-}^{(k)},\chi_{\mathfrak{i}}^{(k')}\rangle_{\mathfrak{i}}\zeta_{C}(\mathbf{x}_{\mathrm{mid}},k)+\frac{q_{C}}{\sigma_{a,C}}\langle\mathbf{1},\chi_{\mathfrak{i}}^{(k')}\rangle_{\mathfrak{i}}\\
    &=\langle\Psi_{\Gamma,\mathfrak{i}-}(\mathbf{x}_{\mathrm{mid}}),\chi_{\mathfrak{i}}^{(k')}\rangle_{\mathfrak{i}}+\langle\tau(\mathbf{x}_{\mathrm{mid}}),\chi_{\mathfrak{i}}^{(k')}\rangle_{\mathfrak{i}}.
    \end{aligned}
\end{equation}

With the definition of $\tilde{\psi}^{*}$ in \cref{eq:sec3:43}, we rewrite \cref{eq:sec4:4} as follows:
\begin{equation}    \langle\tilde{\psi}_{\mathfrak{i}-}^{*}(\mathbf{x}_{\mathrm{mid}}),\chi_{\mathfrak{i}}^{(k')}\rangle_{\mathfrak{i}}=\langle\Psi_{\Gamma,\mathfrak{i}-}(\mathbf{x}_{\mathrm{mid}}),\chi_{\mathfrak{i}}^{(k')}\rangle_{\mathfrak{i}}+\langle\tau(\mathbf{x}_{\mathrm{mid}}),\chi_{\mathfrak{i}}^{(k')}\rangle_{\mathfrak{i}}.
\end{equation}
Therefore, as an approximation to $\tilde{\psi}^{*}$, we require that $\tilde{\psi}_{\delta}$ satisfies the following revised boundary condition at $\mathbf{x}_{\mathrm{mid}}$: for all $k'\in \mathcal{V}_{\delta}^{\mathfrak{i}}$:
\begin{equation}\label{eq:sec3:8}    \langle\tilde{\psi}_{\delta,\mathfrak{i}-}(\mathbf{x}_{\mathrm{mid}}),\chi_{\mathfrak{i}}^{(k')}\rangle_{\mathfrak{i}}=\langle\Psi_{\Gamma,\mathfrak{i}-}(\mathbf{x}_{\mathrm{mid}}),\chi_{\mathfrak{i}}^{(k')}\rangle_{\mathfrak{i}}.
\end{equation}

\subsection{Compressed linear system}
\label{subsec:LinearSystem}

Using the coefficients of the reduced basis functions $\tilde{\alpha}_{\delta,C}^{(k)}$ in Adaptive TFPS solution $\tilde{\psi}_{\delta}$ as the unknowns, and the new interface condition \cref{eq:sec4:11} and new boundary condition \cref{eq:sec3:8} as the equations or constraints, we can assemble the following linear system:
\begin{equation}\label{eq:sec3:9}
\setlength{\abovedisplayskip}{3pt}
\setlength{\belowdisplayskip}{3pt}
\tilde{A}_{\delta}\tilde{\alpha}_{\delta}=\tilde{b}_{\delta}
\end{equation}
where $\tilde{\alpha}_{\delta}=\big(\tilde{\alpha}_{\delta,C}^{(k)}\big)_{k\in\mathcal{V}_{\delta,C},C\in\mathcal{C}}$.

According to the relationship between the index set of reduced basis functions for each cell, denoted as $\{\mathcal{V}_{\delta,C}\}_{C\in\mathcal{C}}$, and the index set of directions for selecting constraints, denoted as $\{\mathcal{V}_{\delta}^{\mathfrak{i}}\}_{\mathfrak{i}\in\mathcal{I}}$, the total number of reduced basis functions is equal to the total number of selected constraints. This implies that $\tilde{A}_{\delta}$ is a square matrix. Additionally, $\tilde{A}_{\delta}$ has a similar sparsity pattern to the matrix $A$ in \cref{eq:sec2:19} (shown in \cref{fig:MatrixA}). However, due to the reduced number of basis functions and constraints in Adaptive TFPS, the size of $\tilde{A}_{\delta}$ is smaller than that of $A$. Specifically, $A$ has dimensions of $8MI^{2} \times 8MI^{2}$, while $\tilde{A}_{\delta}$ has dimensions of $\sum_{C \in \mathcal{C}} \vert \mathcal{V}_{\delta,C} \vert \times \sum_{C \in \mathcal{C}} \vert \mathcal{V}_{\delta,C} \vert$. The accuracy of $\tilde{\psi}_{\delta}$ in approximating the TFPS solution, $\tilde{\psi}$, at the physical cell centers will be validated through a posteriori analysis in the next section.

\begin{remark}
The constrains \cref{eq:sec4:11} and \cref{eq:sec3:8} in Adaptive TFPS, can be understood as the continuity of the angular flux $\tilde{\psi}_{\delta}$ at the edge centers in the locally important velocity modes $\{\chi_{\mathfrak{i}}^{(k)}\}_{k\in \mathcal{V}_{\delta}^{\mathfrak{i}}}$.
In comparison, the constrains \cref{eqn:interfaceequations} and \cref{eqn:boundaryequations}, require the continuity of $\tilde{\psi}$, the TFPS solution, at the edge centers in all velocity modes $\{\chi_{\mathfrak{i}}\}_{k\in\mathcal{V}^{\mathfrak{{i}}}}$.
\end{remark}

\begin{remark}\label{remark:2}
When $\delta > 0$ is sufficiently small, the solvability of the linear system \cref{eq:sec2:19} implies the solvability of the linear system \cref{eq:sec3:9}. This observation will be proven through perturbation analysis in the next section.
\end{remark}


\section{A posteriori analysis}
\label{sec:analysis}
In this section, we rigorously evaluate how well $\tilde{\psi}_{\delta}$, the Adaptive TFPS solution, approximates $\tilde{\psi}$, the full TFPS solution, through a posterior analysis for the angular flux at cell centers, thereby justifying the accuracy of Adaptive TFPS. 

Based on equations \cref{eqn:interfaceequations} and \cref{eqn:boundaryequations}, the TFPS solution $\tilde{\psi}$ satisfies the following constraints at interior edge centers:  for all $\mathfrak{i}\in\mathcal{I}_{i}$,  $k'\in \mathcal{V}^{\mathfrak{i}}$, 
\begin{equation}\label{eq:sec5:4}
    \langle\tilde{\psi}|_{C_{-}}(\mathbf{x}_{\mathrm{mid}}),\chi_{\mathfrak{i}}^{(k')}\rangle_{\mathfrak{i}}=\langle\tilde{\psi}|_{C_{+}}(\mathbf{x}_{\mathrm{mid}}),\chi_{\mathfrak{i}}^{(k')}\rangle_{\mathfrak{i}}.
\end{equation}
and the following constraints at boundary edge centers:  for all $ \mathfrak{i}\in\mathcal{I}_{b}$, $k'\in \mathcal{V}^{\mathfrak{i}}$,
\begin{equation}\label{eq:sec5:5}
    \langle\tilde{\psi}_{\mathfrak{i}-}(\mathbf{x}_{\mathrm{mid}}),\chi_{\mathfrak{i}}^{(k')}\rangle_{\mathfrak{i}}=\langle\Psi_{\Gamma,\mathfrak{i}-}(\mathbf{x}_{\mathrm{mid}}),\chi_{\mathfrak{i}}^{(k')}\rangle_{\mathfrak{i}}.
\end{equation}
Using the coefficients of basis functions in $\tilde{\psi}$ as unknowns and equations \cref{eq:sec5:4} and \cref{eq:sec5:5} as constraints, we can formulate the following linear system:
\begin{equation}\label{eq:sec5:6}
    \tilde{A}\alpha=\tilde{b}.
\end{equation}
It is worth noting that the linear systems \cref{eq:sec2:19} and \cref{eq:sec5:6} share the same unknown vector $\alpha$.

With a certain permutation, the structure of $\tilde{A}$ and $\tilde{b}$ can be written as follows:
\begin{equation}
\setlength{\abovedisplayskip}{3pt}
\setlength{\belowdisplayskip}{3pt}
    \tilde{A}=
    \begin{pmatrix}
        \tilde{A}_{\delta} & \tilde{B}_{\delta}\\
        \tilde{C}_{\delta} & \tilde{D}_{\delta}\\
    \end{pmatrix}
    ,\quad \tilde{b}=
    \begin{pmatrix}
        \tilde{b}_{\delta}\\
        \tilde{d}_{\delta}
    \end{pmatrix}
    \label{eqn:Adelta}
\end{equation}

The blocks in the matrix $\tilde{A}$ have the following interpretations: the rows of the first block correspond to the continuity conditions of selected directions, and the rows of the second block correspond to the continuity conditions of unselected directions. Meanwhile, the columns of the first block correspond to reduced basis functions, and the columns of the second block correspond to unselected basis functions.

\begin{assumption}\label{assump:bound}
Let $E_{\delta,\mathfrak{i}}=(\chi_{\mathfrak{i}}^{(k)})_{k\in\mathcal{V}^{\mathfrak{i}}}$, and $C_{\gamma,g,M,\delta,2}$ be defined as the maximum $L_{2}$ norm of the inverse of each matrix $E_{\delta,\mathfrak{i}}$ over all interfaces $\mathfrak{i}$, i.e.
\begin{equation}\label{def:C2}
    \setlength{\abovedisplayskip}{3pt}
\setlength{\belowdisplayskip}{3pt}C_{\gamma,g,M,\delta,2}=\mathop{\max}_{\mathfrak{i}\in\mathcal{I}} \Vert E_{\delta,\mathfrak{i}}^{-1}\Vert_{2}
\end{equation}
Then $C_{\gamma,g,M,\delta,2}$ is uniformly bounded with respect to $\gamma$, $g$, $M$, $\delta$ and any choice of $\mathcal{I}$.
\end{assumption}

\begin{remark}
     According to the definition provided in equation \cref{def:C2}, $C_{\gamma,g,M,\delta,2}$ represents the maximum $l_2$ norm of the coordinate of a vector $l$ in the new coordinate system defined by the basis vectors $\{\chi_{\mathfrak{i}}^{(k)}\}_{k\in\mathcal{V}^{\mathfrak{i}}}$, given that $l$ belongs to $\mathbb{R}^{4M}$ and has a length of 1 in the Cartesian coordinate system. In another word, $C_{\gamma,g,M,\delta,2}$ indicates the degree of orthogonality of vectors $\{\chi_{\mathfrak{i}}^{(k)}\}_{k\in\mathcal{V}^{\mathfrak{i}}}$. The validity of \cref{assump:bound} can be numerically justified in supplementary material \ref{appendix:boundedness}.
\end{remark}  

\begin{lemma}\label{lemma:Cinf}
    Define $C_{\gamma,g,M,\delta,\infty}$ as follows:
   \begin{equation}\label{def:Cinf}
\setlength{\abovedisplayskip}{3pt}
\setlength{\belowdisplayskip}{3pt}    C_{\gamma,g,M,\delta,\infty}=\mathop{\max}_{\mathfrak{i}\in\mathcal{I}} \Vert E_{\delta,\mathfrak{i}}^{-1}\Vert_{\infty}
    \end{equation}
    Then $\langle l,\chi_{\mathfrak{i}}^{(k)}\rangle_{\mathfrak{i}}\leq C_{\gamma,g,M,\delta,\infty}\Vert l\Vert_{\infty}$ for any $\mathfrak{i}\in\mathcal{I}$,  $l\in U_{\mathfrak{i}}$, and $k\in\mathcal{V}^{\mathfrak{i}}$. Additionally, $1\leq C_{\gamma,g,M,\delta,\infty}\leq \sqrt{4M}C_{\gamma,g,M,\delta,2}$. 
\end{lemma}

\begin{proof}
The first property, $\langle l,\chi_{\mathfrak{i}}^{(k)}\rangle_{\mathfrak{i}}\leq C_{\gamma,g,M,\delta,\infty}\Vert l\Vert_{\infty}$ can be straightforwardly derived from the definition of $C_{\gamma,g,M,\delta,\infty}$ in \cref{def:Cinf}. Besides, 
for any $k\in\mathcal{V}^{\mathfrak{i}}$ and $\mathfrak{i}\in\mathcal{I}$, we have $\Vert E_{\delta,\mathfrak{i}}^{-1}\Vert_{\infty}\geq \frac{\Vert E_{\delta,\mathfrak{i}}^{-1}\chi_{\mathfrak{i}}^{(k)}\Vert_{\infty}}{\Vert\chi_{\mathfrak{i}}^{(k)}\Vert_{\infty} }=\frac{1}{\Vert\chi_{\mathfrak{i}}^{(k)}\Vert_{\infty} }\geq\frac{1}{\Vert\chi_{\mathfrak{i}}^{(k)}\Vert_{2}}=1 $. Therefore, we can conclude that  $C_{\gamma,g,M,\delta,\infty}\geq 1$. Moreover, since $\Vert E_{\delta,\mathfrak{i}}^{-1}\Vert_{\infty}\leq \sqrt{4M}\Vert E_{\delta,\mathfrak{i}}^{-1}\Vert_{2}$, we also have $C_{\gamma,g,M,\delta,\infty}\leq \sqrt{4M}C_{\gamma,g,M,\delta,2}$.
\end{proof}

With the properties of $C_{\gamma,g,M,\delta,\infty}$, we can derive estimates for the infinity norm of $\tilde{B}_{\delta}$, $\tilde{C}_{\delta}$, $\tilde{D}_{\delta}$, and $\tilde{b}$ as detailed in \cref{lemma:normBCD} below. The proof of this result can be found in \cref{appendix:BCD}.


\begin{lemma}\label{lemma:normBCD} 
The following inequalities hold true, 
\begin{gather*}
    \Vert \tilde{B}_{\delta}\Vert_{\infty}\leq 12MC_{\gamma,g,M,\delta,\infty}\delta,\quad \Vert \tilde{C}_{\delta}\Vert_{\infty}\leq 12MC_{\gamma,g,M,\delta,\infty},\\
    \Vert \tilde{D}_{\delta}-I_\delta\Vert_{\infty}\leq 12MC_{\gamma,g,M,\delta,\infty}\delta,\quad
    \Vert\tilde{b}\Vert_{\infty}\leq C_{\gamma,g,M,\delta,\infty}\big(2\big\Vert\frac{q}{\sigma_{a}}\big\Vert_{\infty}+\Vert\Psi_{\Gamma^{-}}\Vert_{\infty}\big),
\end{gather*}
where $I_\delta$ denotes the identity matrix of the same size as the matrix $\tilde{D}_{\delta}$.
\end{lemma}

Besides, we can establish an estimate for the infinity norm of $\tilde{A}^{-1}$ using the infinity norm of $\tilde{A}_{\delta}^{-1}$ and the estimations in \cref{lemma:normBCD}.

\begin{lemma}\label{lemma:posteriorA-1}
The infinity norm of $\tilde{A}^{-1}$ can be bounded by $\tilde{A}_{\delta}^{-1}$ as follows:\\
    $\exists\delta_{0}=1\big/\big(24MC_{\gamma,g,M,\delta,\infty}(24MC_{\gamma,g,M,\delta,\infty}\Vert\tilde{A}^{-1}\Vert_{\infty}+1)\big)$, when $\delta\leq\delta_{0}$:
    \[\Vert\tilde{A}^{-1}\Vert_{\infty}\leq 24MC_{\gamma,g,M,\delta,\infty}\Vert\tilde{A}_{\delta}^{-1}\Vert_{\infty}+2\]
\end{lemma}
\begin{proof}
We first establish an upper bound on $\Vert\tilde{A}_{\delta}^{-1}\Vert_{\infty}$ in terms of $\Vert\tilde{A}^{-1}\Vert_{\infty}$ using \cref{lemma:normBCD}. When $\delta\leq\delta_{0}$,
\begin{equation}
\setlength{\abovedisplayskip}{3pt}
\setlength{\belowdisplayskip}{3pt}
        \begin{aligned}
     &\Vert\tilde{A}_{\delta}^{-1}\Vert_{\infty}\leq
     \Big\Vert
    \begin{pmatrix}
        \tilde{A}_{\delta}^{-1} & 0 \\
        -\tilde{C}_{\delta}\tilde{A}_{\delta}^{-1} & I_{\delta} \\
    \end{pmatrix}
     \Big\Vert_{\infty}
    =\Big\Vert
    \Big(\tilde{A}-
    \begin{pmatrix}
        0 & \tilde{B}_{\delta}\\
        0 & \tilde{D}_{\delta}-I_{\delta}
    \end{pmatrix}
    \Big)^{-1}\Big\Vert_{\infty}\\
    &\leq
    \frac{\Vert\tilde{A}^{-1}\Vert_{\infty}}{1-\Vert\tilde{A}^{-1}\Vert_{\infty}\Big\Vert
    \begin{pmatrix}
        0 & \tilde{B}_{\delta}\\
        0 & \tilde{D}_{\delta}-I_{\delta}
    \end{pmatrix}
    \Big\Vert_{\infty}}
    \leq \frac{\Vert\tilde{A}^{-1}\Vert_{\infty}}{1-\Vert\tilde{A}^{-1}\Vert_{\infty}12MC_{\gamma,g,M,\delta,\infty}\delta}\\
    &\leq \frac{\Vert\tilde{A}^{-1}\Vert_{\infty}}{1-\Vert\tilde{A}^{-1}\Vert_{\infty}12MC_{\gamma,g,M,\delta,\infty}\delta_{0}}
    \leq \frac{\Vert\tilde{A}^{-1}\Vert_{\infty}}{1-\frac{1}{24}}\leq 2\Vert\tilde{A}^{-1}\Vert_{\infty}.\\
    \end{aligned}
\end{equation}
Since
\begin{equation}
\setlength{\abovedisplayskip}{3pt}
\setlength{\belowdisplayskip}{3pt}
    \begin{aligned}
    \Big\Vert
    \begin{pmatrix}
        \tilde{A}_{\delta} & 0 \\
        \tilde{C}_{\delta} & I_{\delta} \\
    \end{pmatrix}
    ^{-1}
    \Big\Vert_{\infty}=
    \Big\Vert
    \begin{pmatrix}
        \tilde{A}_{\delta}^{-1} & 0 \\
        -\tilde{C}_{\delta}\tilde{A}_{\delta}^{-1} & I_{\delta} \\
    \end{pmatrix}
    \Big\Vert_{\infty}
    \leq 12MC_{\gamma,g,M,\delta,\infty}\Vert\tilde{A}_{\delta}^{-1}\Vert_{\infty}+1\\
    \leq 24MC_{\gamma,g,M,\delta,\infty}\Vert\tilde{A}^{-1}\Vert_{\infty}+1,
    \end{aligned}
\end{equation}
we finally get
\begin{equation}
\setlength{\abovedisplayskip}{3pt}
\setlength{\belowdisplayskip}{3pt}
    \begin{aligned}
        &\Vert\tilde{A}^{-1}\Vert_{\infty}=\Big\Vert\Big(
        \begin{pmatrix}
            \tilde{A}_{\delta} & 0\\
            \tilde{C}_{\delta} & I_{\delta}\\
        \end{pmatrix}
        -
        \begin{pmatrix}
            0 & -\tilde{B}_{\delta} \\
            0 & I_{\delta}-\tilde{D}_{\delta} \\
        \end{pmatrix}
        \Big)^{-1}\Big\Vert_{\infty}\\
        \leq&\frac{12MC_{\gamma,g,M,\delta,\infty}\Vert\tilde{A}_{\delta}^{-1}\Vert_{\infty}+1}{1-(24MC_{\gamma,g,M,\delta,\infty}\Vert\tilde{A}^{-1}\Vert_{\infty}+1)12MC_{\gamma,g,M,\delta,\infty}\delta}\\
        \leq &\frac{12MC_{\gamma,g,M,\delta,\infty}\Vert\tilde{A}_{\delta}^{-1}\Vert_{\infty}+1}{1-\frac{1}{2}}=24MC_{\gamma,g,M,\delta,\infty}\Vert\tilde{A}_{\delta}^{-1}\Vert_{\infty}+2.
    \end{aligned}
\end{equation}

\end{proof}

With the infinity norm estimates from \cref{lemma:normBCD} and \cref{lemma:posteriorA-1}, we can now evaluate the difference between $\tilde{\psi}$ and $\tilde{\psi}_{\delta}$ at the cell centers. To quantify this difference, we employ the following norm for any $f \in \mathcal{F}$, 
\begin{equation}\label{eq:norm}
\setlength{\abovedisplayskip}{3pt}
\setlength{\belowdisplayskip}{3pt}
    \Vert f \Vert = \max\limits_{m \in \mathcal{M}, C \in \mathcal{C}} \vert f_{m}(\mathbf{x}_{C}) \vert.
\end{equation}

Besides, we introduce the following intermediate angular flux, denoted as $\tilde{\psi}_{\delta}^{*}$, to bridge the gap between TFPS solution $\tilde{\psi}$ and Adaptive TFPS solution $\tilde{\psi}_{\delta}$,
\begin{equation}
\setlength{\abovedisplayskip}{3pt}
\setlength{\belowdisplayskip}{3pt}
\tilde{\psi}_{\delta}^{*}=\sum_{C\in\mathcal{C}}(\sum_{k\in \mathcal{V}}\tilde{\alpha}_{C}^{(k)}\phi_{C}^{(k)}+\phi_{C}^{s}).
\end{equation}

Here, $\tilde{\psi}_{\delta}^{*}$ encompasses the full set of basis functions, with its coefficients for the reduced basis functions being the same as those of \(\tilde{\psi}_{\delta}\), while its coefficients for the unselected basis functions are derived from the coefficients of the reduced basis functions. Denote $\tilde{\alpha}=\Big((\tilde{\alpha}_{1,1})^{T},(\tilde{\alpha}_{2,1})^{T},\dots,(\tilde{\alpha}_{I,I})^{T}\Big)^{T}$ and $\tilde{\alpha}_{i,j}=\Big(\tilde{\alpha}_{i,j}^{(1)},\dots,\tilde{\alpha}_{i,j}^{(8M)}\Big)^{T}$. Then, after a certain permutation as in \cref{eqn:Adelta}, $\tilde{\alpha}$ can be expressed as $\tilde{\alpha}=\Big((\tilde{\alpha}_{\delta})^{T},(\tilde{d}_{\delta}-\tilde{C}_{\delta}\tilde{\alpha}_{\delta})^{T}\Big)^{T}$. We will now derive the bound for $\Vert\tilde{\psi}-\tilde{\psi}_{\delta}\Vert$ by evaluating $\Vert\tilde{\psi}-\tilde{\psi}_{\delta}^{}\Vert$ and $\Vert\tilde{\psi}_{\delta}^{}-\tilde{\psi}_{\delta}\Vert$ in the following \cref{lemma:posteriorerror1} and \cref{lemma:posteriorerror2}.

\begin{lemma}\label{lemma:posteriorerror1}
\[\setlength{\abovedisplayskip}{3pt}
\setlength{\belowdisplayskip}{3pt}
\Vert \tilde{\psi}-\tilde{\psi}_{\delta}^{*}\Vert\leq 96M^{2}(24MC\Vert\tilde{A}_{\delta}^{-1}\Vert_{\infty}+2)4M\delta (\Vert\Psi_{\Gamma^{-}}\Vert_{\infty}+2\big\Vert\frac{q}{\sigma_{a}}\big\Vert_{\infty}+12M\Vert\tilde{\alpha}_{\delta}\Vert_{\infty})\]
\end{lemma}

\begin{proof}
By the definition of $\tilde{\alpha}$, we have
\[
\begin{pmatrix}
    \tilde{A}_{\delta} & 0\\
    \tilde{C}_{\delta} & I_\delta \\
\end{pmatrix}
\tilde{\alpha}=
\begin{pmatrix}
    \tilde{A}_{\delta} & 0\\
    \tilde{C}_{\delta} & I_\delta \\
\end{pmatrix}
\begin{pmatrix}
    \tilde{\alpha}_{\delta}\\
    \tilde{d}_{\delta}-\tilde{C}_{\delta}\tilde{\alpha}_{\delta}\\
\end{pmatrix}
=
\begin{pmatrix}
    \tilde{b}_{\delta}\\
    \tilde{d}_{\delta}\\
\end{pmatrix}
=\tilde{b},
\]
and
\begin{equation*}
\setlength{\abovedisplayskip}{3pt}
\setlength{\belowdisplayskip}{3pt}
    \begin{aligned}
        \tilde{A}(\alpha-\tilde{\alpha})&=\tilde{A}\alpha-
\begin{pmatrix}
    \tilde{A}_{\delta} & 0\\
    \tilde{C}_{\delta} & I_\delta \\
\end{pmatrix}
\tilde{\alpha}-
\begin{pmatrix}
    0 & \tilde{B}_{\delta}\\
    0 & \tilde{D}_{\delta}-I\\
\end{pmatrix}
\tilde{\alpha}=-
\begin{pmatrix}
    0 & \tilde{B}_{\delta}\\
    0 & \tilde{D}_{\delta}-I\\
\end{pmatrix}
\tilde{\alpha}\\
&=-
\begin{pmatrix}
    \tilde{B}_{\delta}\\
    \tilde{D}_{\delta}-I_\delta\\
\end{pmatrix}
(\tilde{d}_{\delta}-\tilde{C}_{\delta}\tilde{\alpha}_{\delta}).
    \end{aligned}
\end{equation*}
With the following estimate,
\begin{displaymath}
\setlength{\abovedisplayskip}{3pt}
\setlength{\belowdisplayskip}{3pt}
    \begin{aligned}
        &\Big\Vert-
        \begin{pmatrix}
            \tilde{B}_{\delta}\\
            \tilde{D}_{\delta}-I\\
        \end{pmatrix}
        (\tilde{d}_{\delta}-\tilde{C}_{\delta}\tilde{\alpha}_{\delta})
        \Big\Vert_{\infty} 
        \leq  \Big\Vert
        \begin{pmatrix}
            \tilde{B}_{\delta}\\
            \tilde{D}_{\delta}-I\\
        \end{pmatrix}
        \Big\Vert_{\infty}\Vert \tilde{d}_{\delta}-\tilde{C}_{\delta}\tilde{\alpha}_{\delta}\Vert_{\infty}\\
        \leq& 12MC_{\gamma,g,M,\delta,\infty}\delta(\Vert \tilde{b}\Vert_{\infty}+\Vert \tilde{C}_{\delta}\tilde{\alpha}_{\delta}\Vert_{\infty})\\
        \leq& 12MC_{\gamma,g,M,\delta,\infty}^{2}\delta (\Vert\Psi_{\Gamma^{-}}\Vert_{\infty}+2\big\Vert\frac{q}{\sigma_{a}}\big\Vert_{\infty}+12M\Vert\tilde{\alpha}_{\delta}\Vert_{\infty}),
    \end{aligned}    
\end{displaymath}
we derive the bound,
\begin{displaymath}
\begin{aligned}
    &\Vert\alpha-\tilde{\alpha}\Vert_{\infty}
    \leq 
    12M\Vert\tilde{A}^{-1}\Vert_{\infty}C_{\gamma,g,M,\delta,\infty}^{2}\delta (\Vert\Psi_{\Gamma^{-}}\Vert_{\infty}+2\big\Vert\frac{q}{\sigma_{a}}\big\Vert_{\infty}+12M\Vert\tilde{\alpha}_{\delta}\Vert_{\infty})\\
    \leq  &12M(24MC_{\gamma,g,M,\delta,\infty}\Vert\tilde{A}_{\delta}^{-1}\Vert_{\infty}+2)C_{\gamma,g,M,\delta,\infty}^{2}\delta (\Vert\Psi_{\Gamma^{-}}\Vert_{\infty}+2\big\Vert\frac{q}{\sigma_{a}}\big\Vert_{\infty}+12M\Vert\tilde{\alpha}_{\delta}\Vert_{\infty}).
\end{aligned}
\end{displaymath}
The lemma concludes using the fact $\Vert\phi_{C}^{(k)}\Vert_{\infty}=1$,
\begin{displaymath}
         \Vert \tilde{\psi}-\tilde{\psi}_{\delta}^{*}\Vert\leq\Vert \tilde{\psi}-\tilde{\psi}_{\delta}^{*}\Vert_{\infty}
         =\Vert \sum_{C\in\mathcal{C}}\sum_{k\in \mathcal{V}}(\alpha_{C}^{(k)}-\tilde{\alpha}_{C}^{(k)})\phi_{C}^{(k)}\Vert_{\infty}\leq 8M\Vert\alpha-\tilde{\alpha}\Vert_{\infty}    
\end{displaymath}
\end{proof}

\begin{lemma}\label{lemma:posteriorerror2}
    \[\Vert\tilde{\psi}_{\delta}^{*}-\tilde{\psi}_{\delta}\Vert\leq 8MC_{\gamma,g,M,\delta,\infty}\delta (\Vert\Psi_{\Gamma^{-}}\Vert_{\infty}+2\big\Vert\frac{q}{\sigma_{a}}\big\Vert_{\infty}+12M\Vert\tilde{\alpha}_{\delta}\Vert_{\infty}).
    \]
\end{lemma}
\begin{proof}
For any $\mathbf{x}_{C}$, the center of a cell $C\in\mathcal{C}$, we have
\begin{equation}
    \tilde{\psi}_{\delta}^{*}(\mathbf{x}_{C})-\tilde{\psi}_{\delta}(\mathbf{x}_{C})=\sum_{k\in\bar{\mathcal{V}}_{\delta,C}}\tilde{\alpha}_{C}^{(k)}\phi_{C}^{(k)}(\mathbf{x}_{C}).
\end{equation}
The selection rule of basis functions in Adaptive TFPS requires that,
\[\Vert\phi_{C}^{(k)}(\mathbf{x}_{C})\Vert_{\infty}\leq \delta,\quad \forall k\in \bar{\mathcal{V}}_{\delta,C},C\in\mathcal{C},
\]
which leads to the following inequality,
\begin{equation}
    \Vert (\tilde{\psi}_{\delta}^{*}-\tilde{\psi}_{\delta})(\mathbf{x}_{C})\Vert_{\infty}\leq 8M\delta \mathop{\max}\limits_{k\in \bar{\mathcal{V}}_{\delta,C}}\{\tilde{\alpha}_{C}^{(k)}\}\leq 8M\delta\Vert \tilde{d}_{\delta}-\tilde{C}_{\delta}\tilde{\alpha}_{\delta}\Vert_{\infty}
\end{equation}
Finally, we obtain
\begin{displaymath}
\setlength{\abovedisplayskip}{3pt}
\setlength{\belowdisplayskip}{3pt}
\begin{aligned}
\Vert\tilde{\psi}_{\delta}^{*}-\tilde{\psi}_{\delta}\Vert & \leq 8M\delta\Vert \tilde{d}_{\delta}-\tilde{C}_{\delta}\tilde{\alpha}_{\delta}\Vert_{\infty}\\
& \leq 8MC_{\gamma,g,M,\delta,\infty}\delta(\Vert\Psi_{\Gamma^{-}}\Vert_{\infty}+2\big\Vert\frac{q}{\sigma_{a}}\big\Vert_{\infty}+12M\Vert\tilde{\alpha}_{\delta}\Vert_{\infty}).
\end{aligned}
\end{displaymath}

\end{proof}

\begin{theorem}[A posteriori analysis]
    Denote the full TFPS solution as $\tilde{\psi}$, the Adaptive TFPS solution as $\tilde{\psi}_{\delta}$, and the  coefficients of the basis functions in $\tilde{\psi}_{\delta}$ as $\tilde{\alpha}_{\delta}$. Then the error between $\tilde{\psi}$ and $\tilde{\psi}_{\delta}$ at cell centers can be bounded as follows:
    \begin{equation}
    \begin{aligned}
            \Vert\tilde{\psi}-\tilde{\psi}_{\delta}\Vert\leq 96M^{2}(24MC_{\gamma,g,M,\delta,\infty}\Vert\tilde{A}_{\delta}^{-1}\Vert_{\infty}+3)C_{\gamma,g,M,\delta,\infty}^{2}\delta\\
            (\Vert\Psi_{\Gamma^{-}}\Vert_{\infty}+2\big\Vert\frac{q}{\sigma_{a}}\big\Vert_{\infty}+12M\Vert\tilde{\alpha}_{\delta}\Vert_{\infty})
    \end{aligned}
    \end{equation}
\end{theorem}
\begin{proof}
    The proof follows from \cref{lemma:posteriorerror1} and Lemma \cref{lemma:posteriorerror2}.
\end{proof}

\section{Numerical Experiments}
\label{sec:experiments}

In this section, we initiate with numerical experiments to demonstrate the presence of a low-rank structure in the angular domain of the RTE in x-y geometry. This structure can be indicated by the number of local basis functions utilized in Adaptive TFPS. Furthermore, we illustrate the accuracy of Adaptive TFPS by comparing the discrepancy between $\tilde{\psi}_{\delta}$ (the angular flux computed using the reduced-order scheme, Adaptive TFPS) and $\tilde{\psi}$ (the angular flux computed using the full-order scheme, TFPS) in these experiments. It is pertinent to note that the discrepancy between $\tilde{\psi}$ and $\tilde{\psi}_{\delta}$ at cell centers has been examined through posterior error analysis in the preceding section, and the numerical examples in this section serve to substantiate this analysis.

\subsection{Lattice problem}

Firstly, we consider the lattice case of a \(4 \times 4\) checkerboard in \(\Omega=[0,1] \times [0,1]\), which includes both diffusive and transport regions, as illustrated in \cref{fig:checkboard}. Each subregion within the checkerboard is of identical size. The coefficients in the diffusive regions are: $
\sigma_{T} = 1000, \quad \sigma_{s} = 999.9995, \quad g = 0, \quad q = 0 $;
and in the transport regime are:
$
\sigma_{T} = 1, \quad \sigma_{s} = 0.5, \quad g = 0, \quad q = 0$.
Additionally, boundary conditions are specified as follows,
\[
\psi_{m}(0,y)=1,\quad c_{m}>0;\quad \psi_{m}(1,y)=1,\quad c_{m}<0;
\]
\[\psi_{m}(x,0)=1,\quad s_{m}>0;\quad \psi_{m}(x,1)=1,\quad s_{m}<0.\]

\begin{figure}[htbp]
    \vspace{-0.3cm}
    \setlength{\abovecaptionskip}{0.cm}
    \centering
    \includegraphics[width=0.4\textwidth]{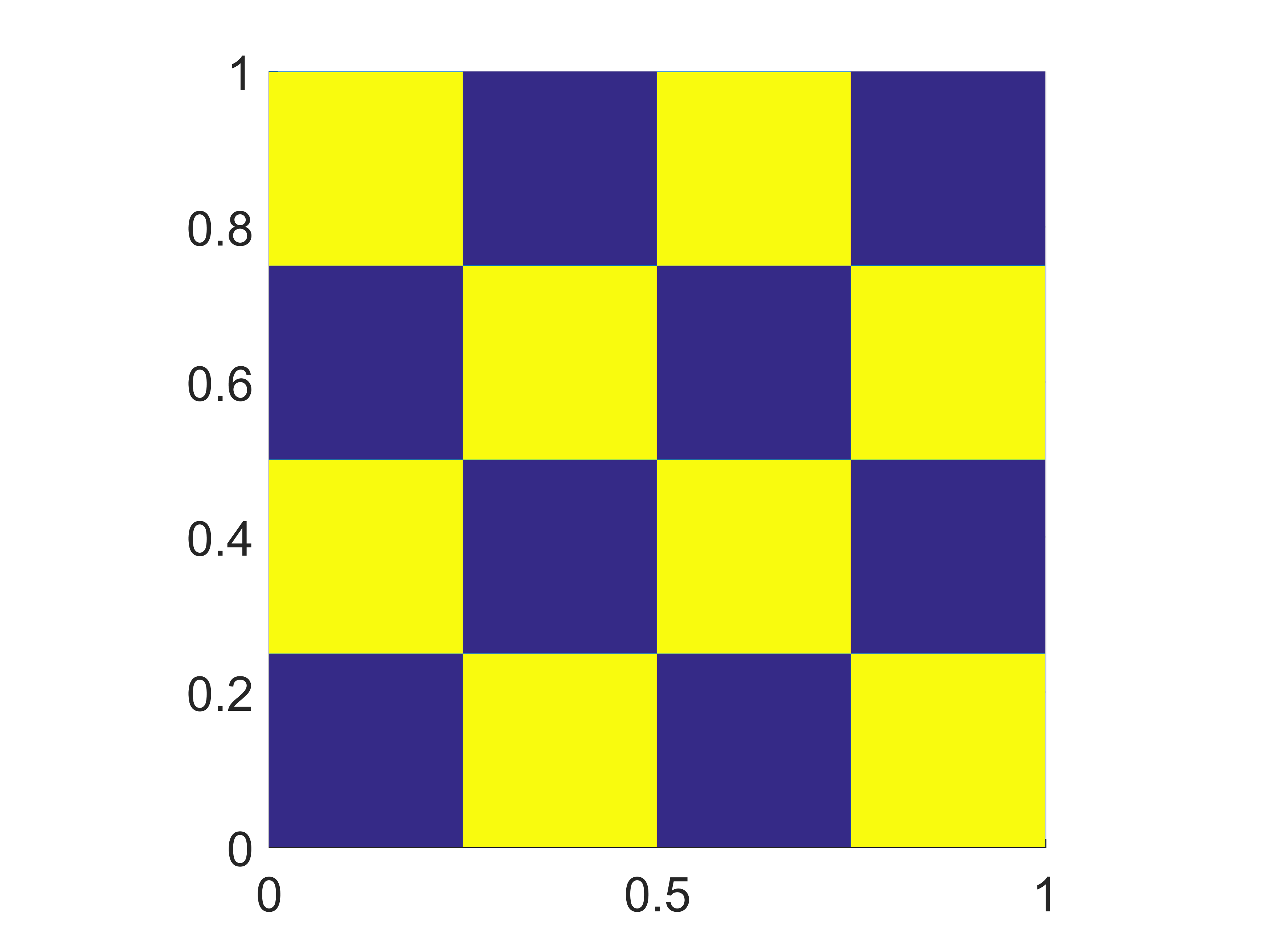}
    \caption{The layout for the lattice problem is as follows: blue rectangles represent the diffusive regions, while yellow rectangles represent the transport regions. }
    \label{fig:checkboard}
    \vspace{-0.3cm}
\end{figure}

To address this problem, we use DOM to discretize the angular domain and then apply Adaptive TFPS to discretize the spatial domain. We choose \( I = J = 32 \), meaning that the physical domain is discretized into a \( 32 \times 32 \) mesh. Subsequently, we conduct tests to show the number of adaptive TFPS basis functions in each cell, varying the  number of velocity directions (\( 4M = 4, 12, 24, 40, 60, 84\)), and threshold values (\( \delta = 10^{-1}, 10^{-2}, 10^{-3}, 10^{-4}, 10^{-5}, 10^{-10}\)). The results show that for cells within the diffusive regime, the number of basis functions is low and increases as $\delta$ decreases, as shown in \cref{tab:Example1_basis}. However, for cells in the transport regime, the number of basis functions remains constant at \( 8M \), identical to that in the full TFPS scheme.

The reduced number of basis functions in Adaptive TFPS suggests the existence of a low-rank structure within the velocity domain. This underlines the computational efficiency that is achieved through the use of Adaptive TFPS as opposed to the standard full TFPS.

\begin{table}[htbp]
\centering
\caption{The number of basis functions utilized in the physical cell located within the diffusive regime for different selections of $M$ and $\delta$.}
\label{tab:Example1_basis}
\begin{tabular}{|c|c|c|c|c|c|c|}
\hline
\diagbox{$M$}{basis}{$\delta$} & $10^{-1}$ & $10^{-2}$ & $10^{-3}$ & $10^{-4}$ & $10^{-5}$ & $10^{-10}$ \\\hline
1 & 4 & 4 & 4 & 4 & 4 & 4\\\hline
3 & 4 & 4 & 4 & 4 & 4 & 8 \\\hline
6 & 4 & 4 & 4 & 4 & 4 & 24 \\\hline
10 & 4 & 4 & 4 & 4 & 4 & 28 \\\hline
15 & 4 & 4 & 4 & 4 & 4 & 48 \\\hline
21 & 4 & 4 & 4 & 4 & 4 & 52 \\\hline
\end{tabular}
\vspace{-0.3cm}
\end{table}

To demonstrate the accuracy of Adaptive TFPS, we show \(\tilde{\phi} = \sum_{m \in \mathcal{M}} \tilde{\psi}_{m}\) and \(\tilde{\phi}_{\delta} = \sum_{m \in \mathcal{M}} \tilde{\psi}_{\delta,m}\) at cell centers for \(\delta = 10^{-1}\), \(M = 21\) in \cref{fig:Example1_center}, and define
\[
\mathrm{error} = \max_{1 \leq i,j \leq I, m \in \mathcal{M}} \left|\tilde{\psi}_{m}(x_{i-1/2},y_{j-1/2}) - \tilde{\psi}_{\delta,m}(x_{i-1/2},y_{j-1/2})\right|
\]
to quantitatively illustrate the difference between the Adaptive TFPS solution \(\tilde{\psi}\) and the TFPS solution \(\tilde{\psi}_{\delta}\). Furthermore, to quantify the computational savings achieved by Adaptive TFPS, we define
\[
\mathrm{ratio} = \frac{\sum_{C\in\mathcal{C}}|\mathcal{V}_{\delta,C}|}{\sum_{C\in\mathcal{C}}|\mathcal{V}|}
\]
which represents the proportion of the total number of basis functions utilized in Adaptive TFPS relative to the total number of basis functions in full TFPS.
\begin{figure}[htbp]
    \vspace{-0.3cm}
    \setlength{\abovecaptionskip}{0.cm}
    \centering
    \subfloat[$\tilde{\phi}$]
    {
    \includegraphics[width=0.4\textwidth]{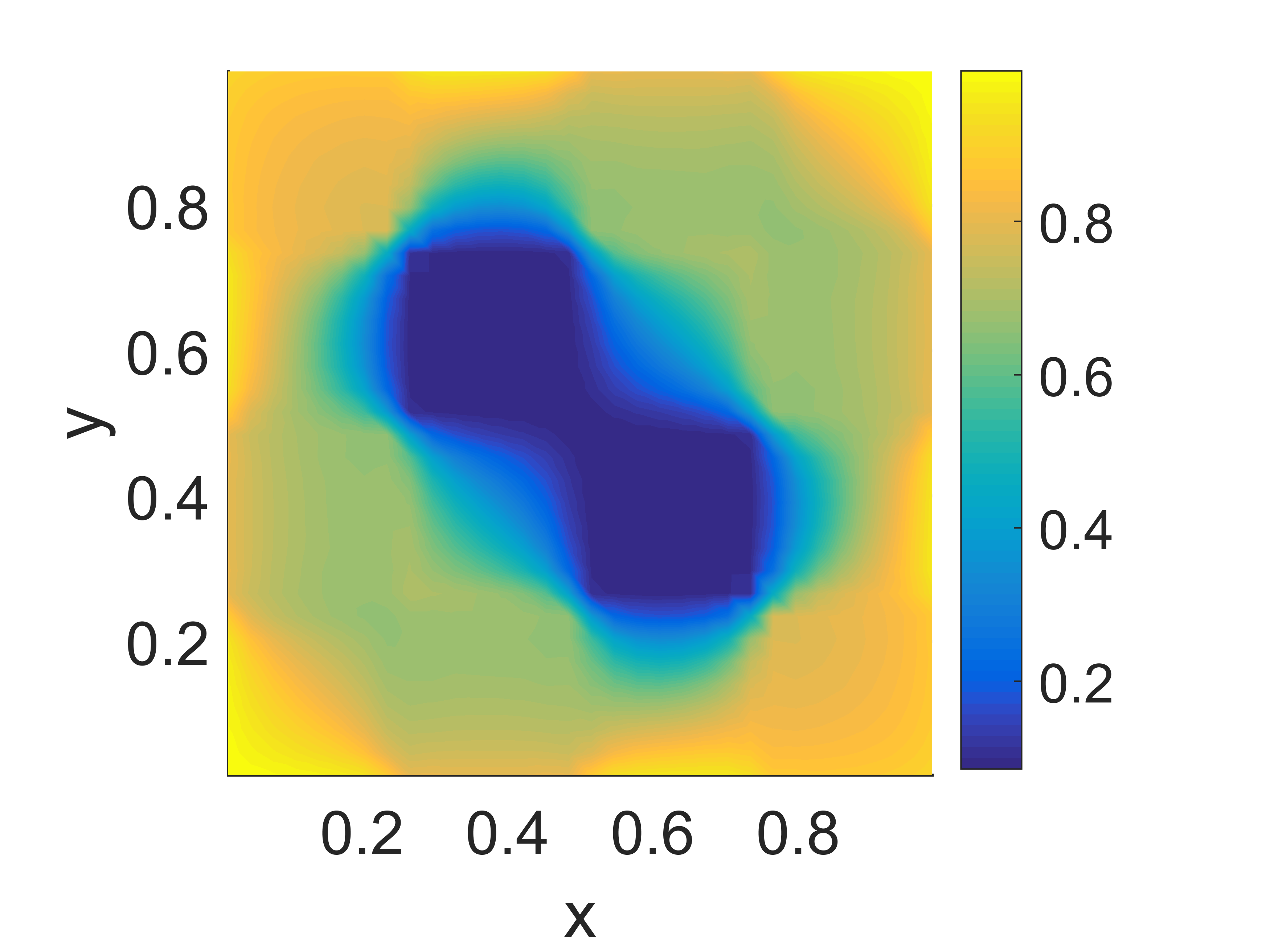}
    }
    \subfloat[$\tilde{\phi}_{\delta}$]
    {
    \includegraphics[width=0.4\textwidth]{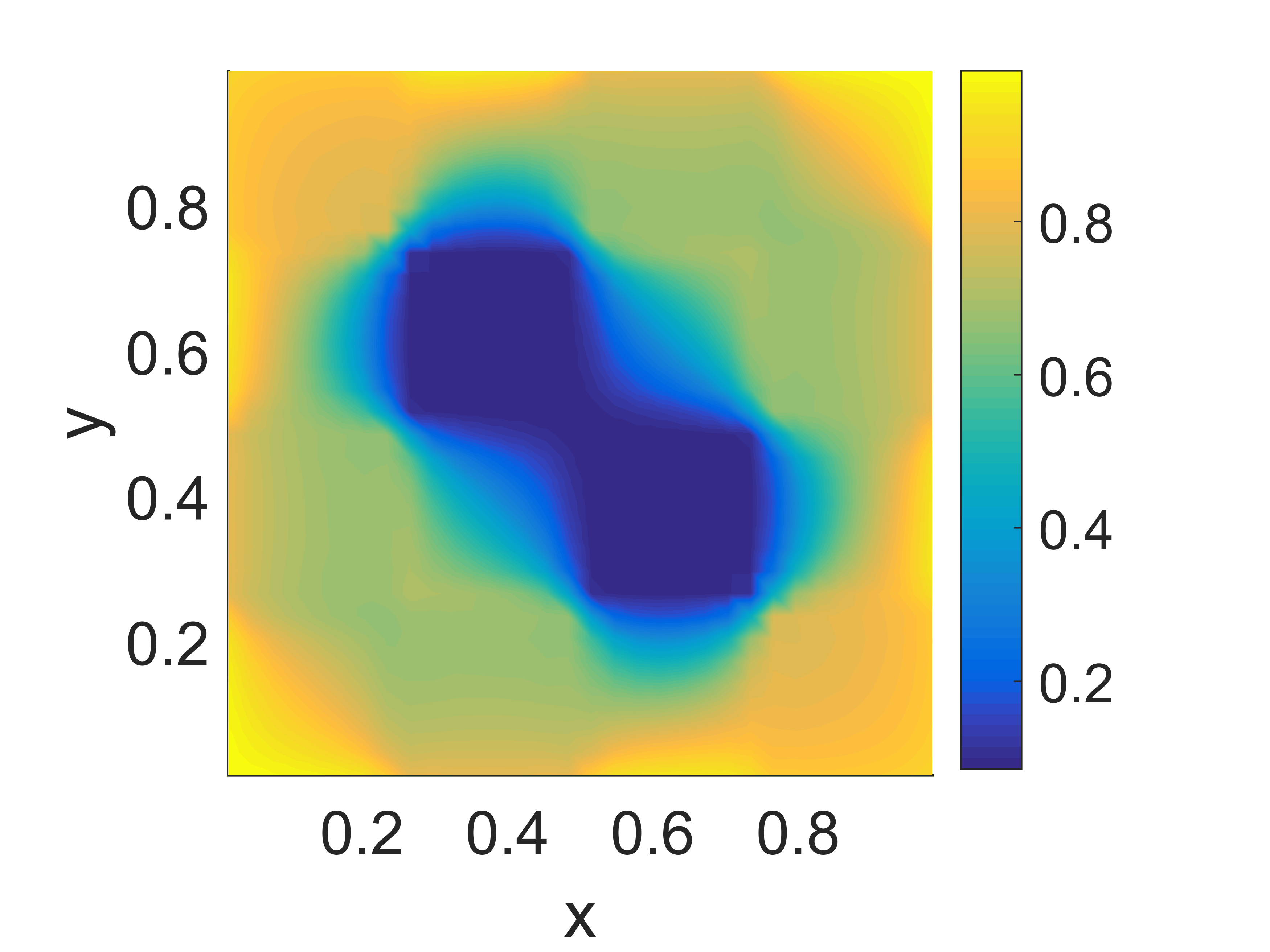}
    }
    \caption{The profile of $\tilde{\phi}$ and $\tilde{\phi}_{\delta}$ for \(M=21\) and \(\delta=10^{-1}\) for lattice problem.}
    \label{fig:Example1_center}
    \vspace{-0.3cm}
\end{figure}

We proceed to illustrate the values of \(\mathrm{error}\) and \(\mathrm{ratio}\)  with respect to \(\delta\) for various selections of \(M\) in \cref{fig:Example1_error_ratio}. The figures clearly demonstrate that \(\tilde{\psi}_{\delta}\) accurately approximates \(\tilde{\psi}\) at cell centers while achieving significant computational savings. 

\begin{figure}[htbp]
    \setlength{\abovecaptionskip}{0.cm}
    \centering
    \subfloat[$M=1$]
    {
    \includegraphics[width=0.3\textwidth]{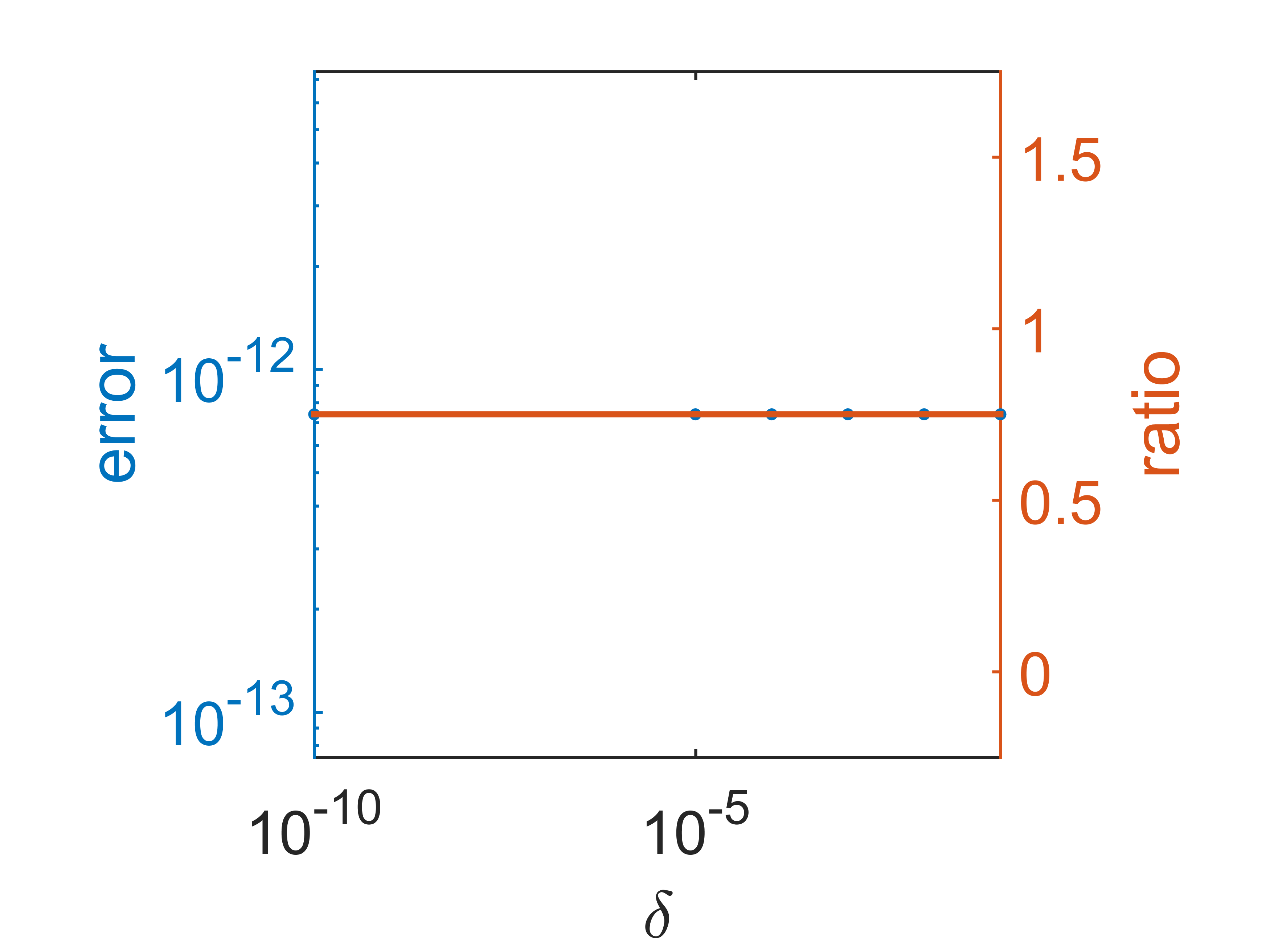}
    }
    \subfloat[$M=3$]
    {
    \includegraphics[width=0.3\textwidth]{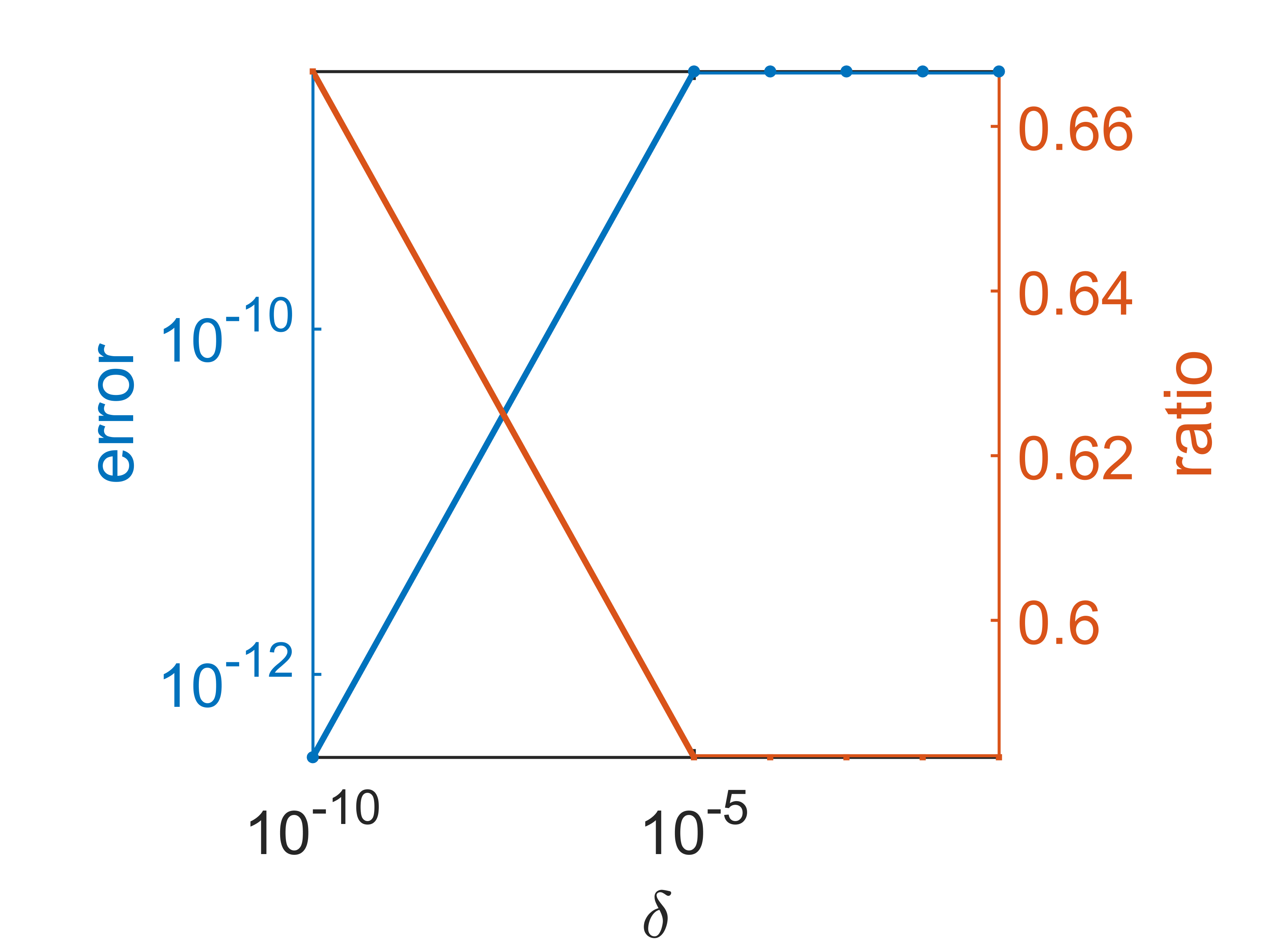}
    }
    \subfloat[$M=6$]
    {
    \includegraphics[width=0.3\textwidth]{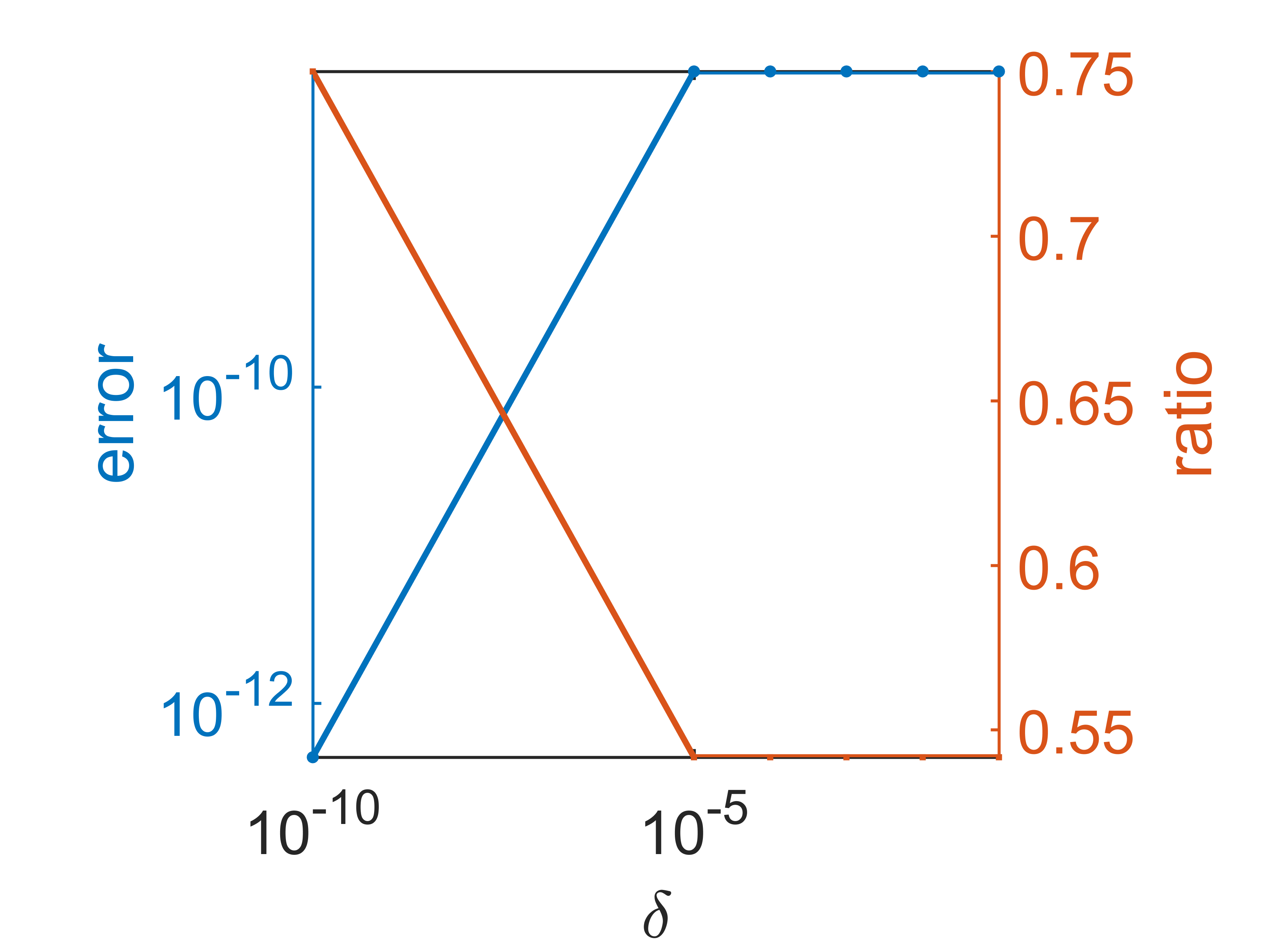}
    }
    
    \subfloat[$M=10$]
    {
    \includegraphics[width=0.3\textwidth]{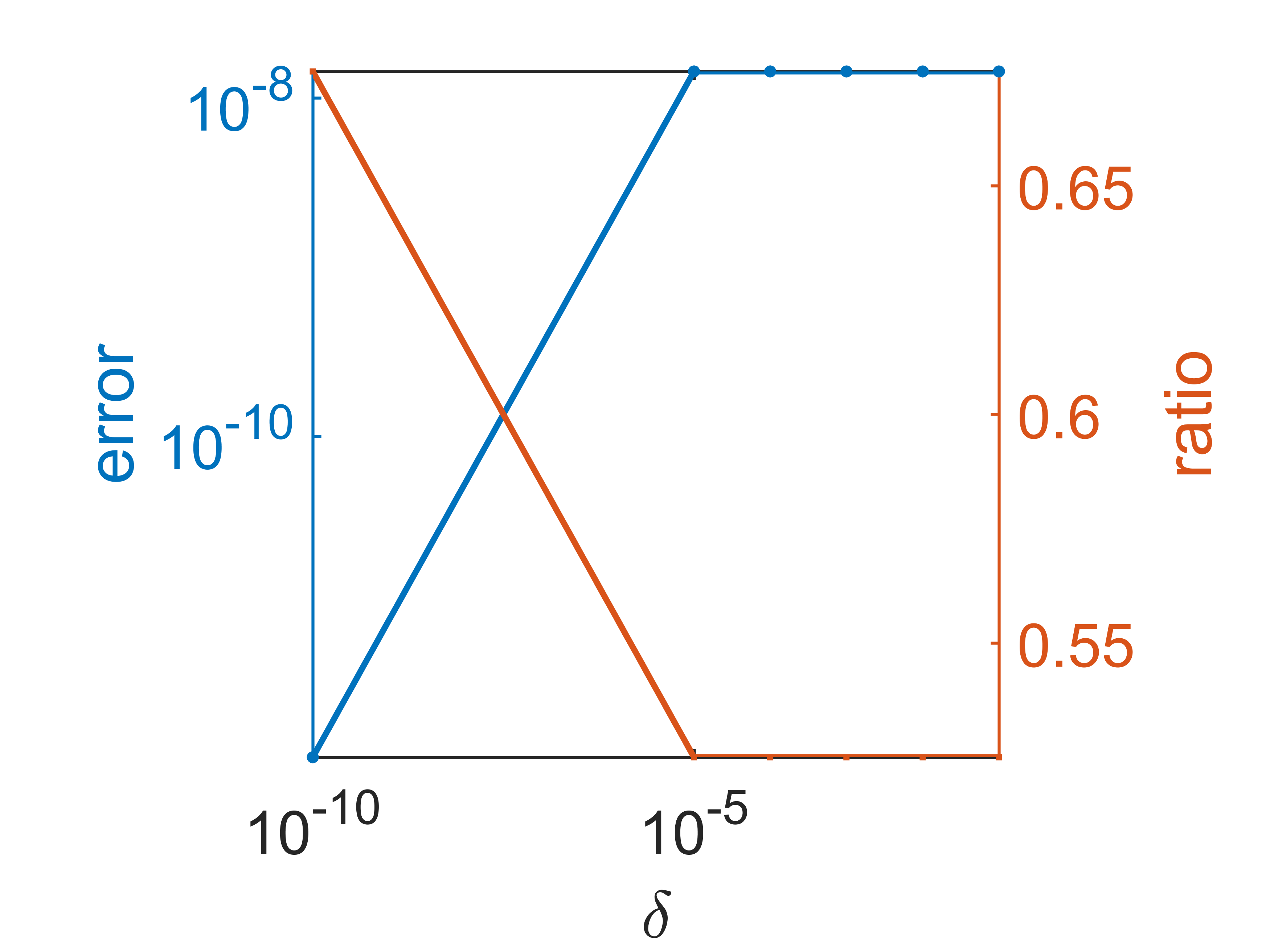}
    }
    \subfloat[$M=15$]
    {
    \includegraphics[width=0.3\textwidth]{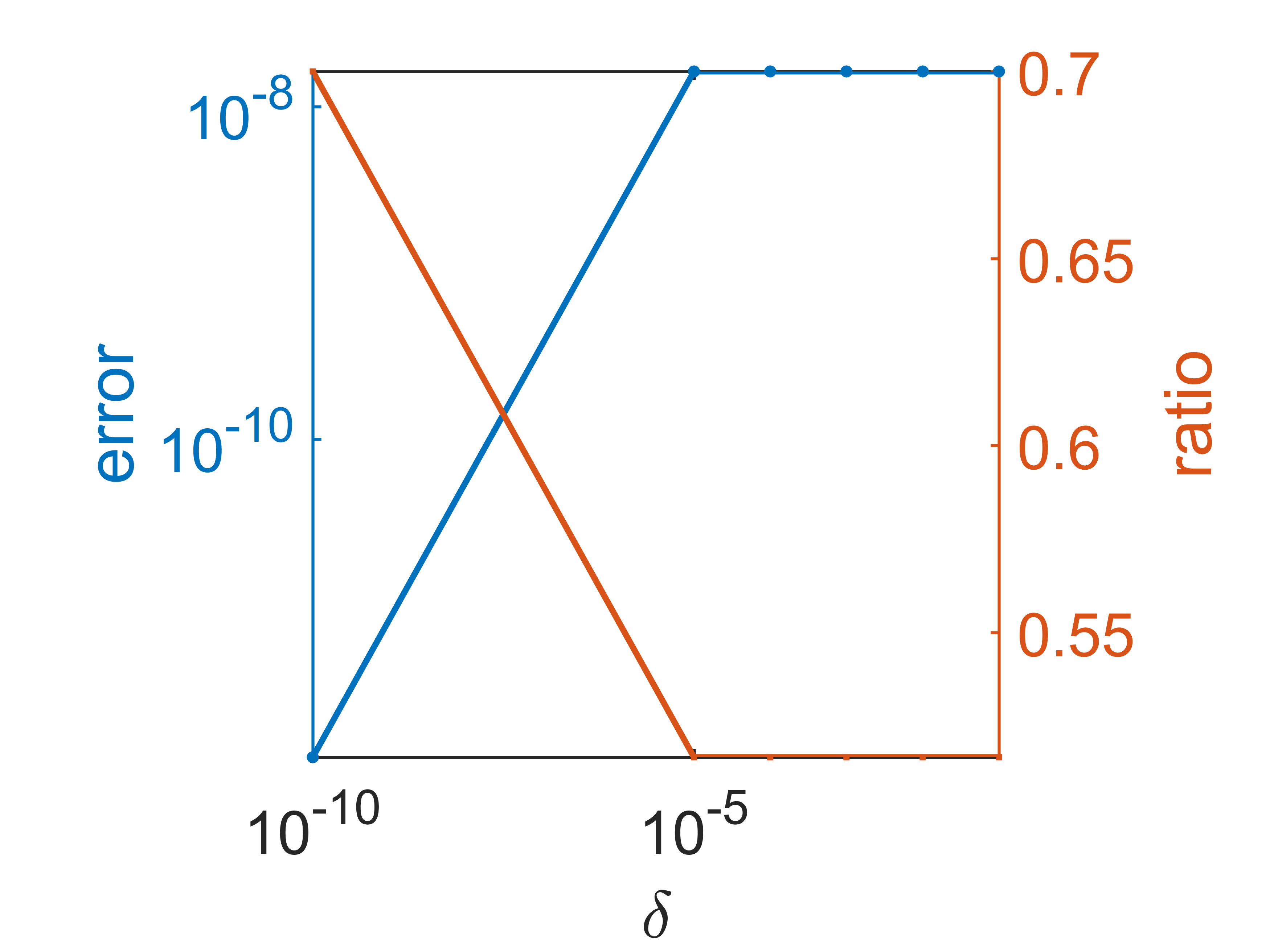}
    }
    \subfloat[$M=21$]
    {
    \includegraphics[width=0.3\textwidth]{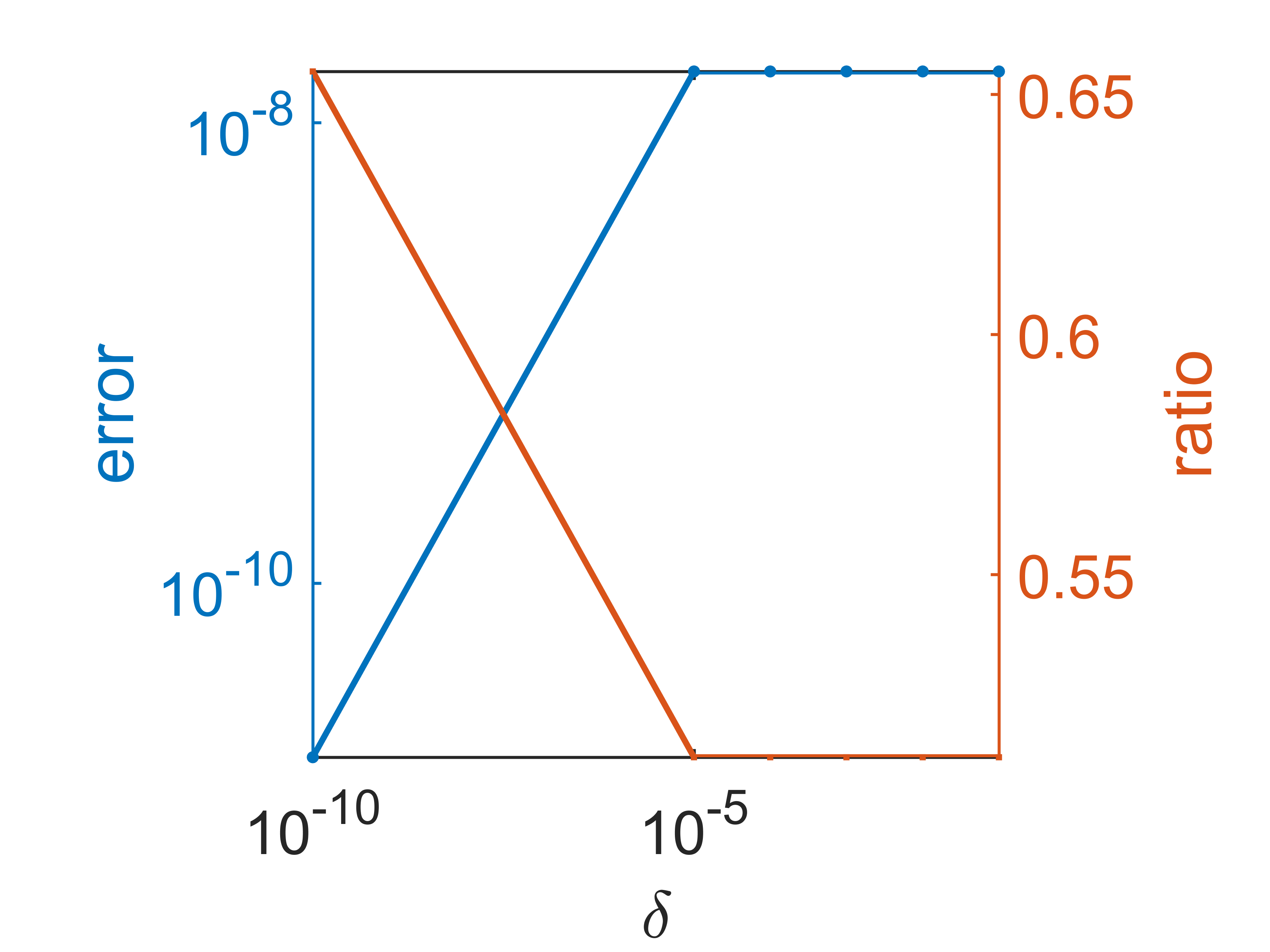}
    }
    \caption{The figure illustrates how the error between \(\tilde{\psi}_{\delta}\) and \(\tilde{\psi}\) and the ratio between the total number of basis utilized in Adaptive TFPS (compressed scheme) and in full TFPS changes with respect to $\delta$ for different choice of $M$ for lattice problem.}
    \label{fig:Example1_error_ratio}
\end{figure}

Based on prior knowledge, we understand that layers may appear at interfaces between different regions or physical boundaries. Since the compression process in Adaptive TFPS removes layer information, the accuracy of \(\tilde{\psi}_{\delta}\) compared to \(\tilde{\psi}\) cannot be guaranteed near these layers. \Cref{fig:Example1_layer} shows \(\tilde{\phi} - \tilde{\phi}_{\delta}\) at the interface layer located at \(x \in [\frac{1}{32}, \frac{7}{32}]\), \(y = \frac{1}{4} - \frac{1}{1280}\) for $M=21$. It is observed that the accuracy of Adaptive TFPS at the interface layer is not very high. However, as \(\delta\) approaches zero, the accuracy improves, indicating that more layer information is retained.

\begin{figure}[htbp]
    \vspace{-0.3cm}
    \setlength{\abovecaptionskip}{0.cm}
    \centering
    \subfloat[$\delta=10^{-1}$]
    {
    \includegraphics[width=0.3\textwidth]{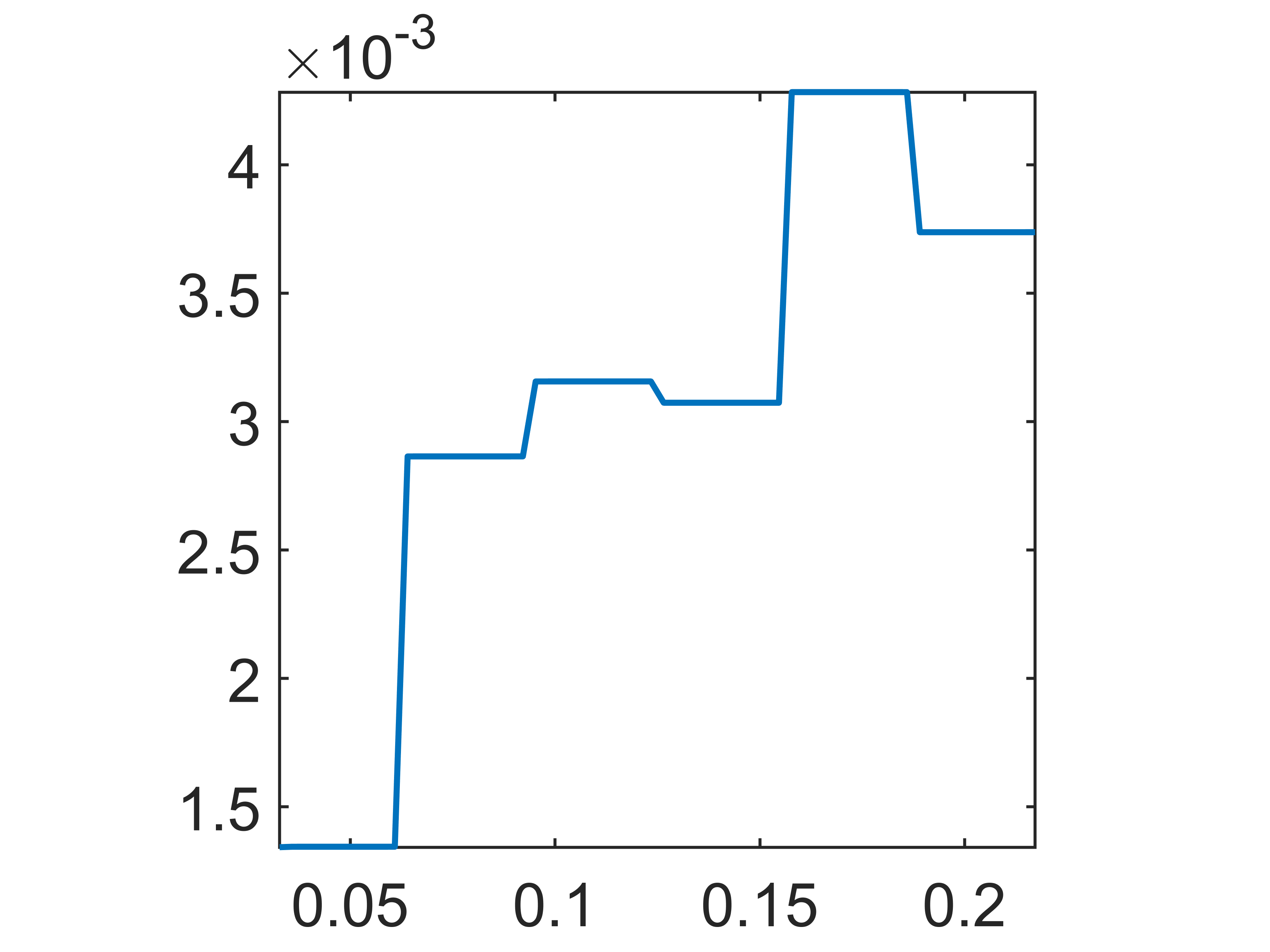}
    }
    \subfloat[$\delta=10^{-10}$]
    {
    \includegraphics[width=0.3\textwidth]{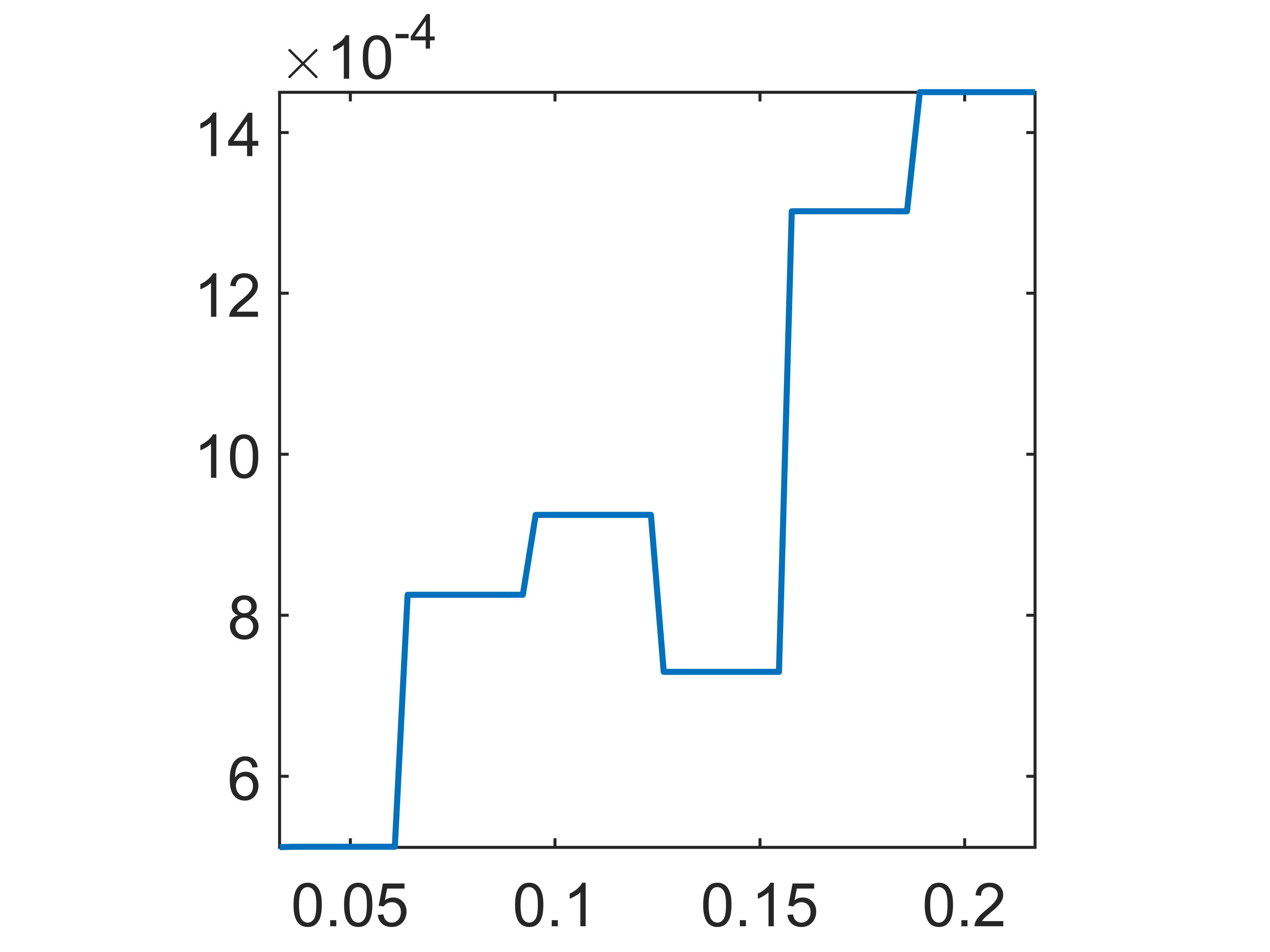}
    }
    \subfloat[$\delta=10^{-20}$]
    {
    \includegraphics[width=0.3\textwidth]{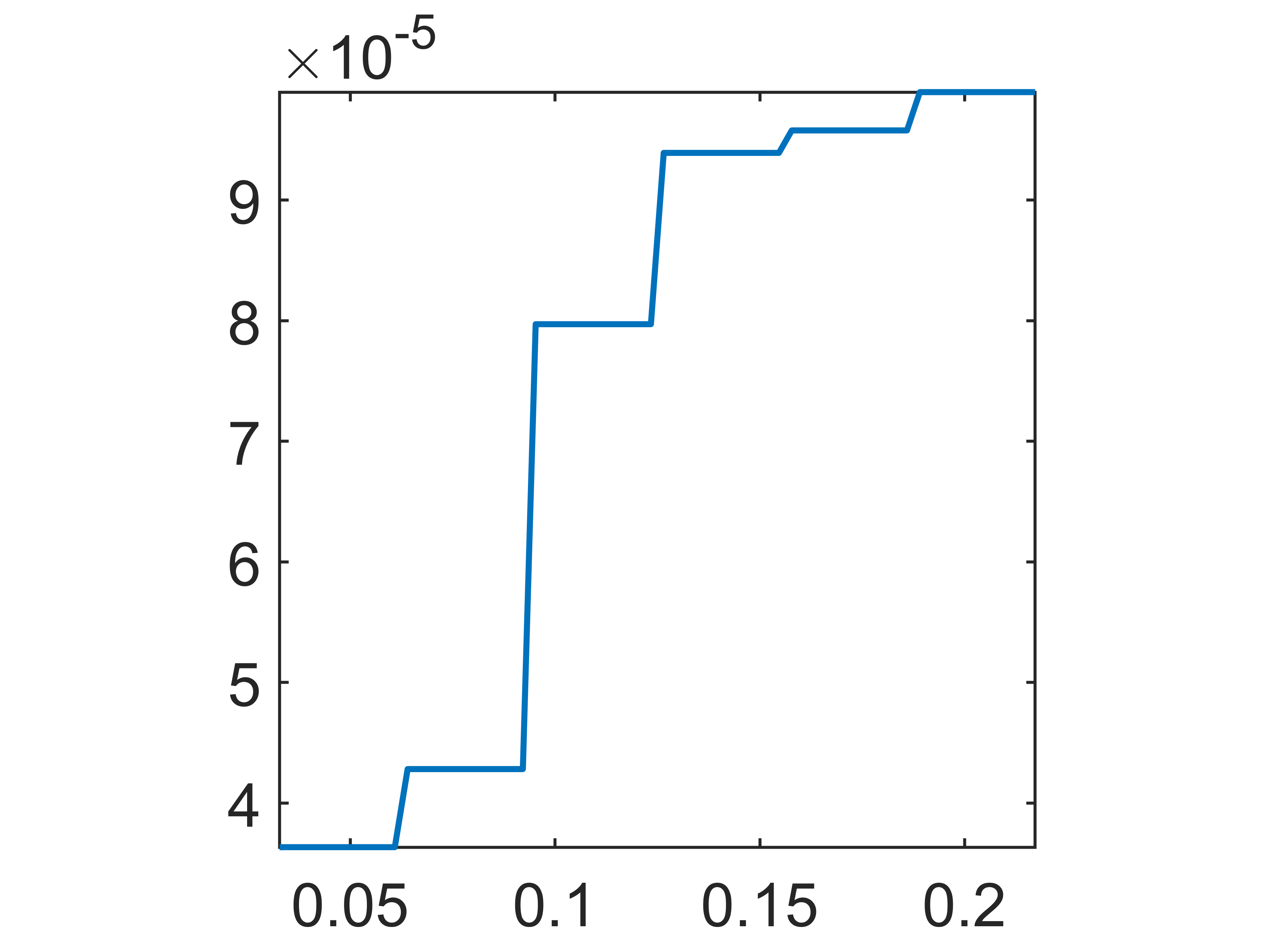}
    }
    \caption{The value of $\tilde{\phi}-\tilde{\phi}_{\delta}$ at the interface layer located at $x\in [\frac{1}{32},\frac{7}{32}]$, $y=\frac{1}{4}-\frac{1}{1280}$.}
    \label{fig:Example1_layer}
    \vspace{-0.3cm}
\end{figure}

\subsection{Buffer zone problem}

In the previous example, we investigated the lattice problem, which features sharp interfaces between different regimes. Nonetheless, such problems can also be addressed using asymptotic analysis and domain decomposition \cite{klar1995domain, bal2002coupling, golse2003domain, li2015diffusion}, thus not fully showcasing the advantage of our Adaptive TFPS approach. We now turn our attention to a more complex multiscale 2D RTE problem: the buffer zone problem, where one regime transitions gradually into another across the physical domain without sharp interfaces. In such situations, conventional analysis tools may not be applicable. The physical domain is \(\Omega=[0,1]\times [0,1]\), and the coefficients are specified as follows,
\[\sigma_{T}=\frac{1+x^{2}+y^{2}}{0.02x+0.001},\quad \sigma_{a}=\Big(0.02x+0.001\Big)\Big(0.5+x^{2}+y^{2}\Big),\quad (x,y)\in\Omega\]
\[g=0.2,\quad q=(0.02x+0.001)\sin(xy).\]
Additionally, the boundary conditions are as follows,
\[
\psi_{m}(0,y)=0,\quad c_{m}>0;\quad \psi_{m}(1,y)=0,\quad c_{m}<0;
\]
\[\psi_{m}(x,0)=0,\quad s_{m}>0;\quad \psi_{m}(x,1)=0,\quad s_{m}<0.\]

We apply DOM and Adaptive TFPS to discretize the buffer zone problem and choose \( I = J = 32 \). In \cref{fig:Example2_basis}, the number of basis functions used for each spatial cell in Adaptive TFPS is illustrated for different values of \( M \) and \( \delta \). Here, the value at the top of the colorbar is \(8M\), representing the number of full TFPS basis functions per cell. We observe that the number of basis functions varies between 4 and \( 8M \), highlighting the low-rank structure in the angular domain in the buffer zone case.

\begin{figure}[htbp]
    \vspace{-0.3cm}
    \setlength{\abovecaptionskip}{0.cm}
    \centering
    \subfloat[$M=3,\delta=10^{-1}$]
    {
    \includegraphics[width=0.3\textwidth]{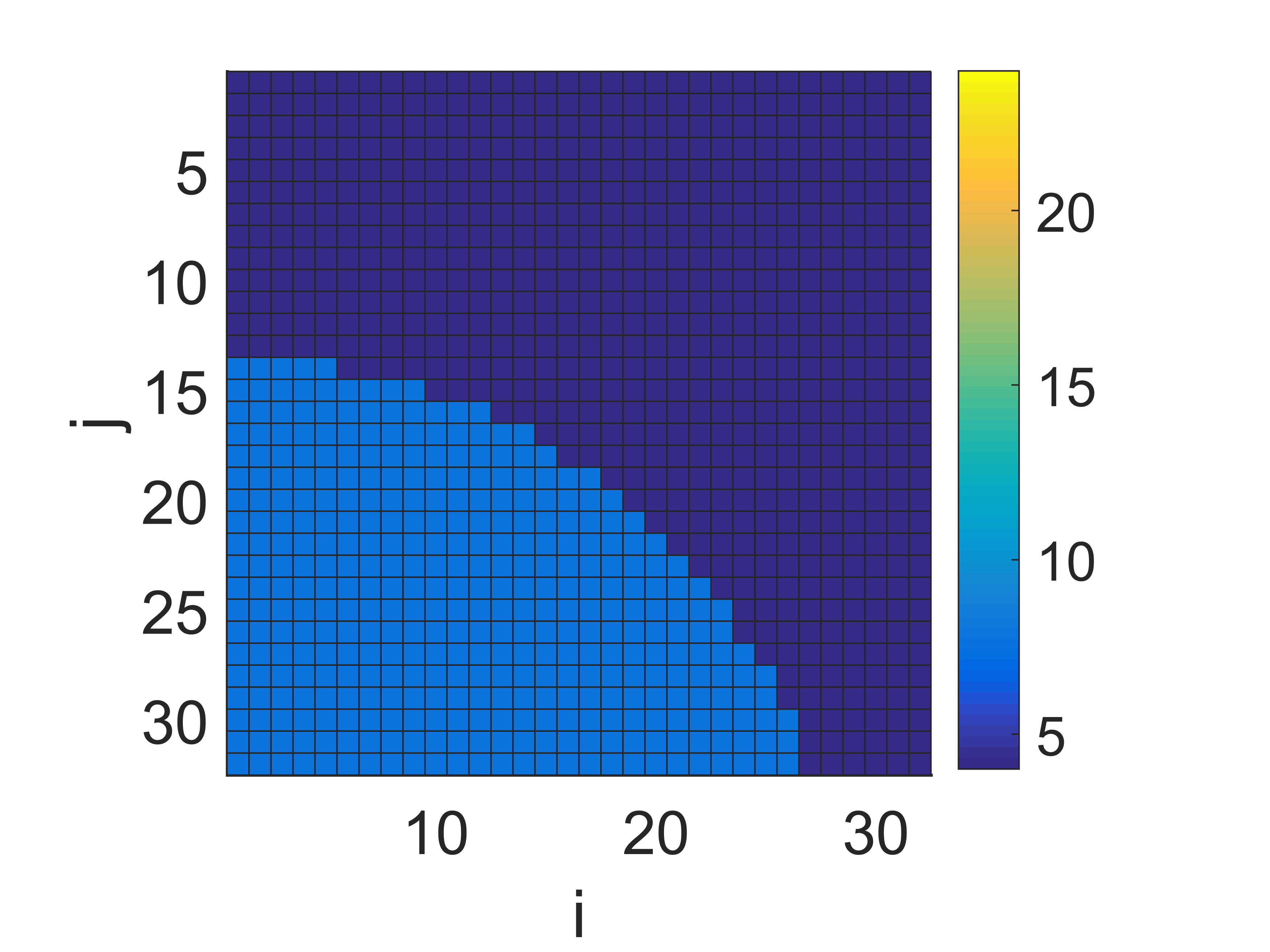}
    }
    \subfloat[$M=3,\delta=10^{-3}$]
    {
    \includegraphics[width=0.3\textwidth]{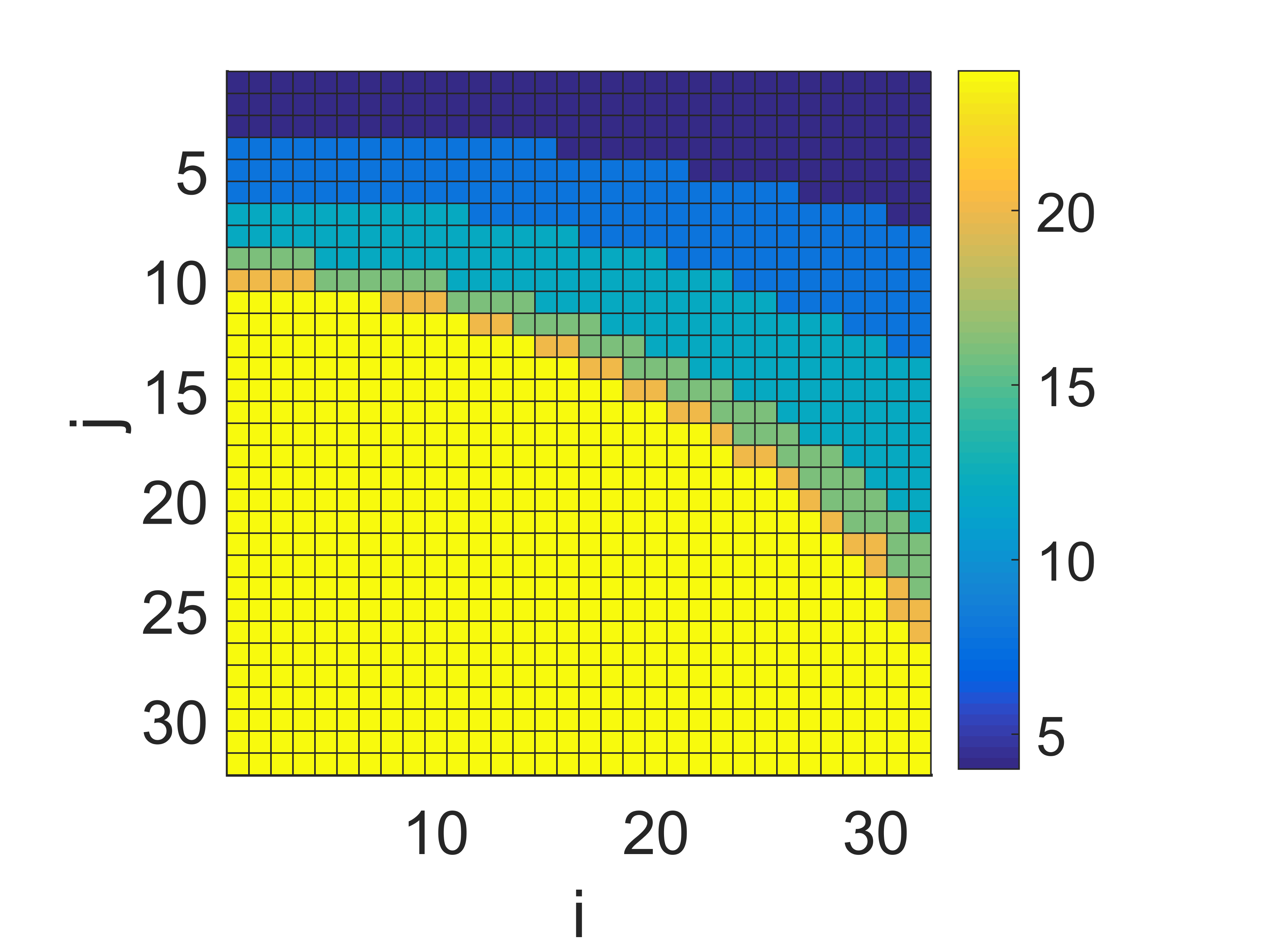}
    }
    \subfloat[$M=3,\delta=10^{-5}$]
    {
    \includegraphics[width=0.3\textwidth]{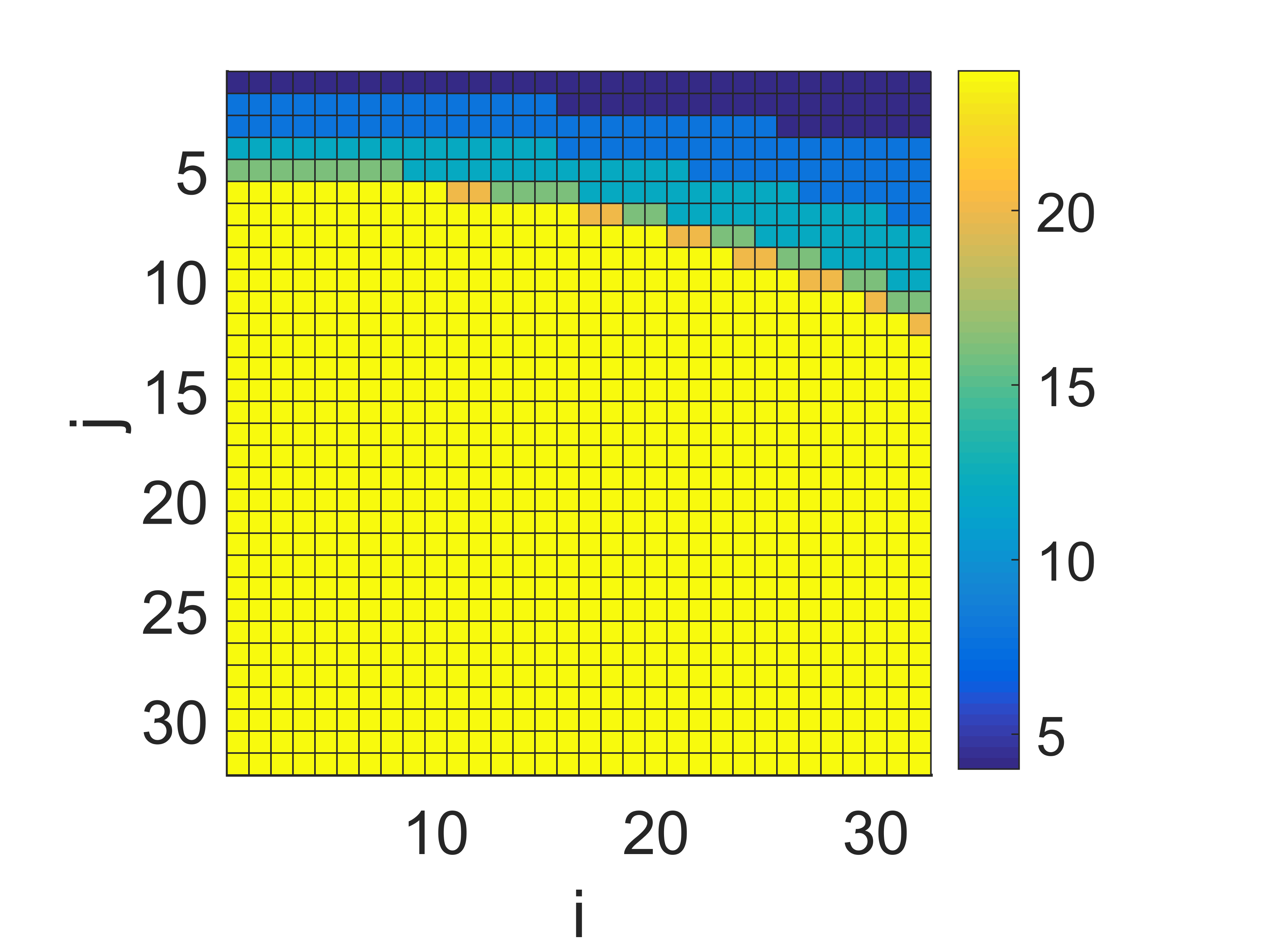}
    }

    \subfloat[$M=21,\delta=10^{-1}$]
    {
    \includegraphics[width=0.3\textwidth]{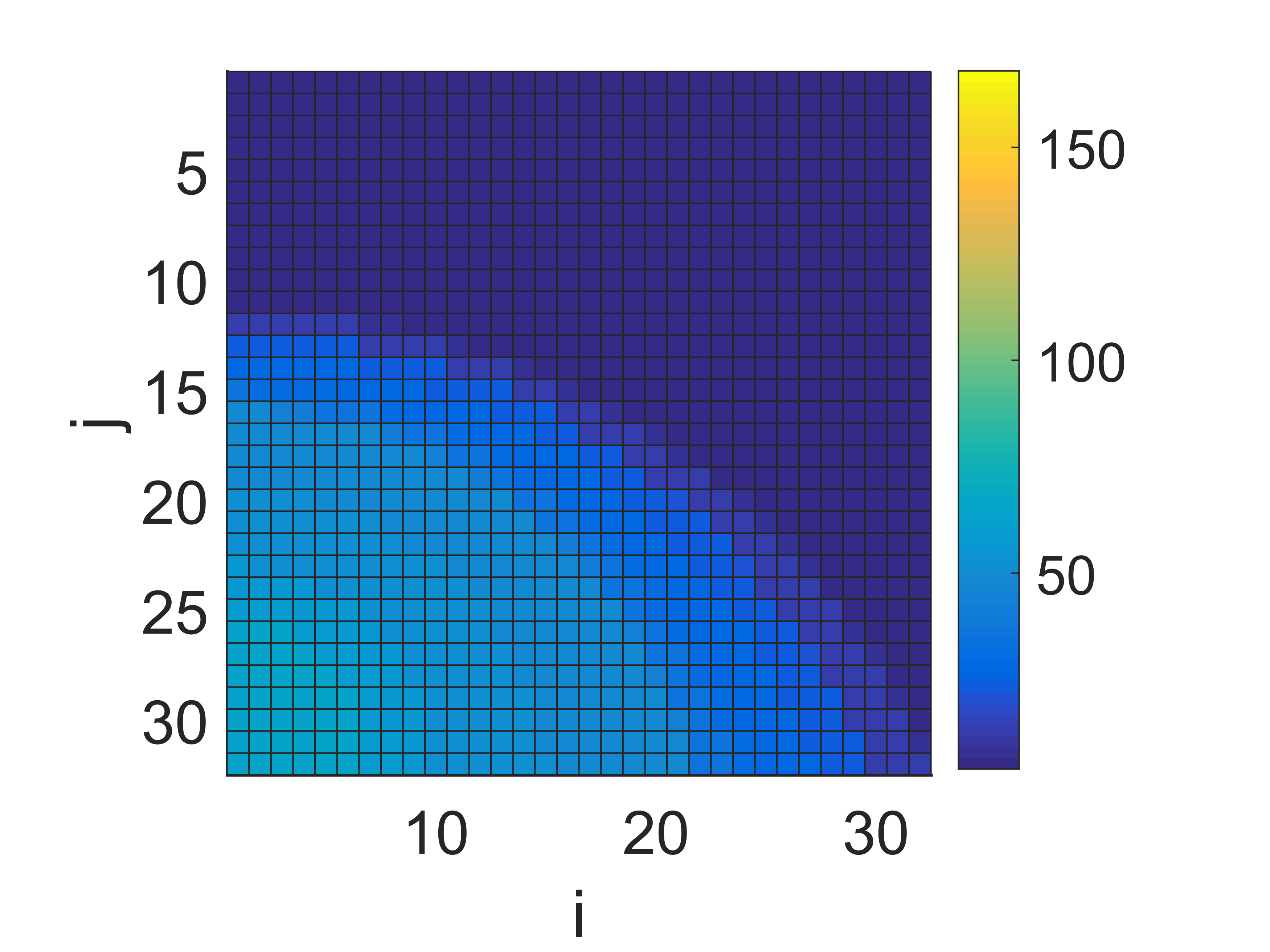}
    }
    \subfloat[$M=21,\delta=10^{-3}$]
    {
    \includegraphics[width=0.3\textwidth]{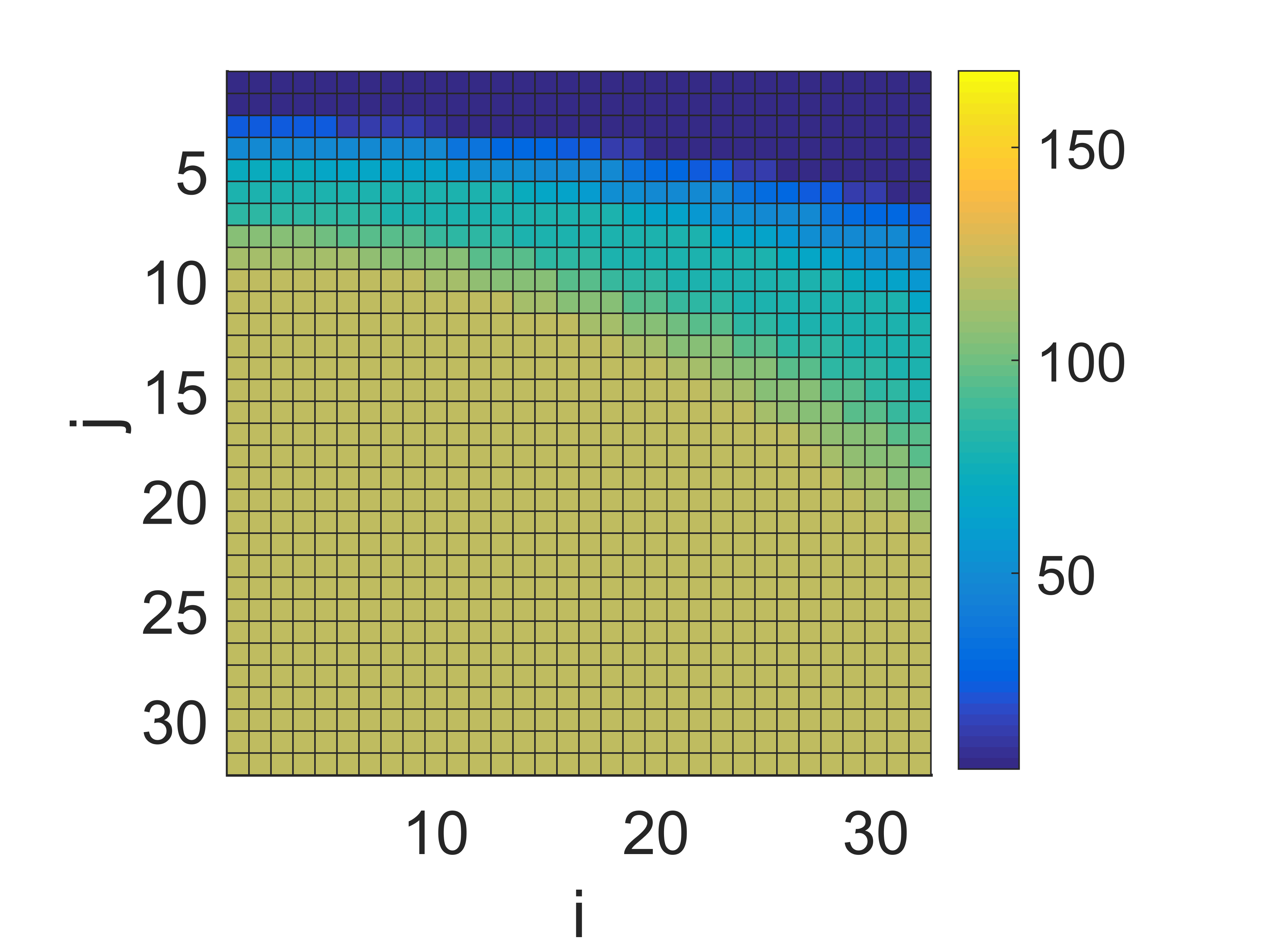}
    }
    \subfloat[$M=21,\delta=10^{-5}$]
    {
    \includegraphics[width=0.3\textwidth]{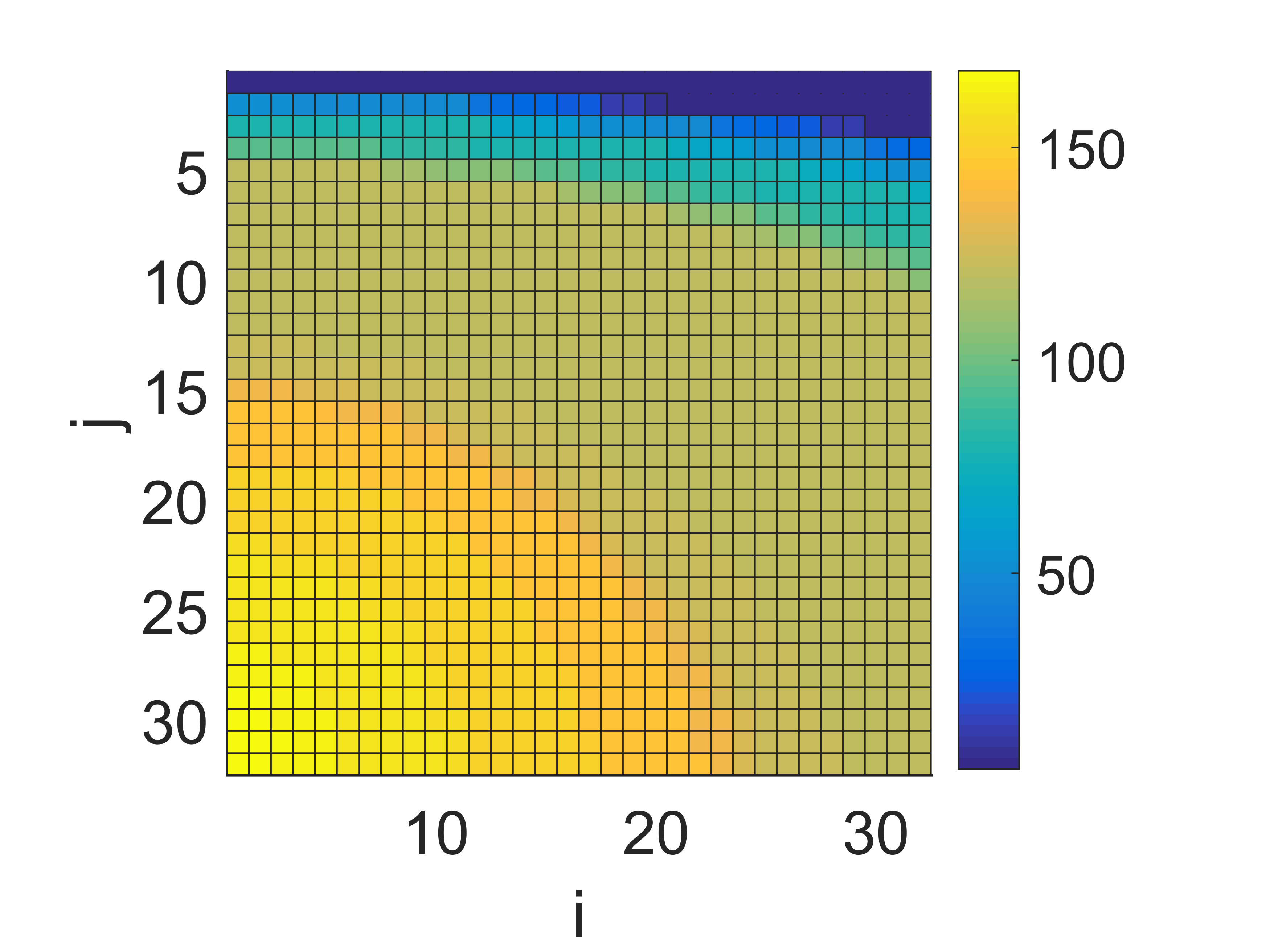}
    }
    \caption{The number of adaptive TFPS basis functions in each cell for different choices of velocity directions ($4M$) and threshold values $\delta$.}
    \label{fig:Example2_basis}
    \vspace{-0.3cm}
\end{figure}

\begin{figure}[htbp]
    \vspace{-0.3cm}
    \setlength{\abovecaptionskip}{0.cm}
    \centering
    \subfloat[$\tilde{\phi}$]
    {
    \includegraphics[width=0.4\textwidth]{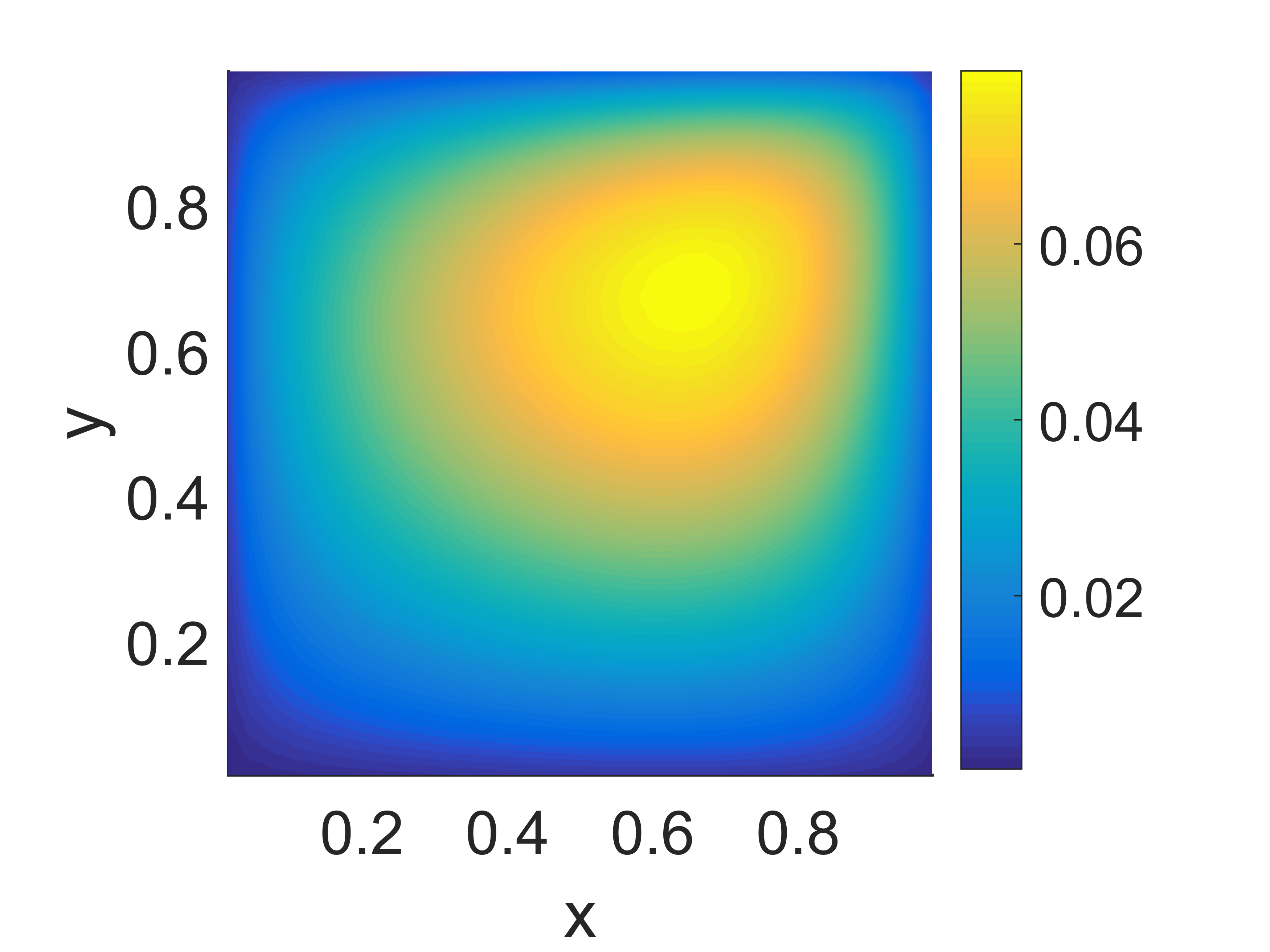}
    }
    \subfloat[$\tilde{\phi}_{\delta}$]
    {
    \includegraphics[width=0.4\textwidth]{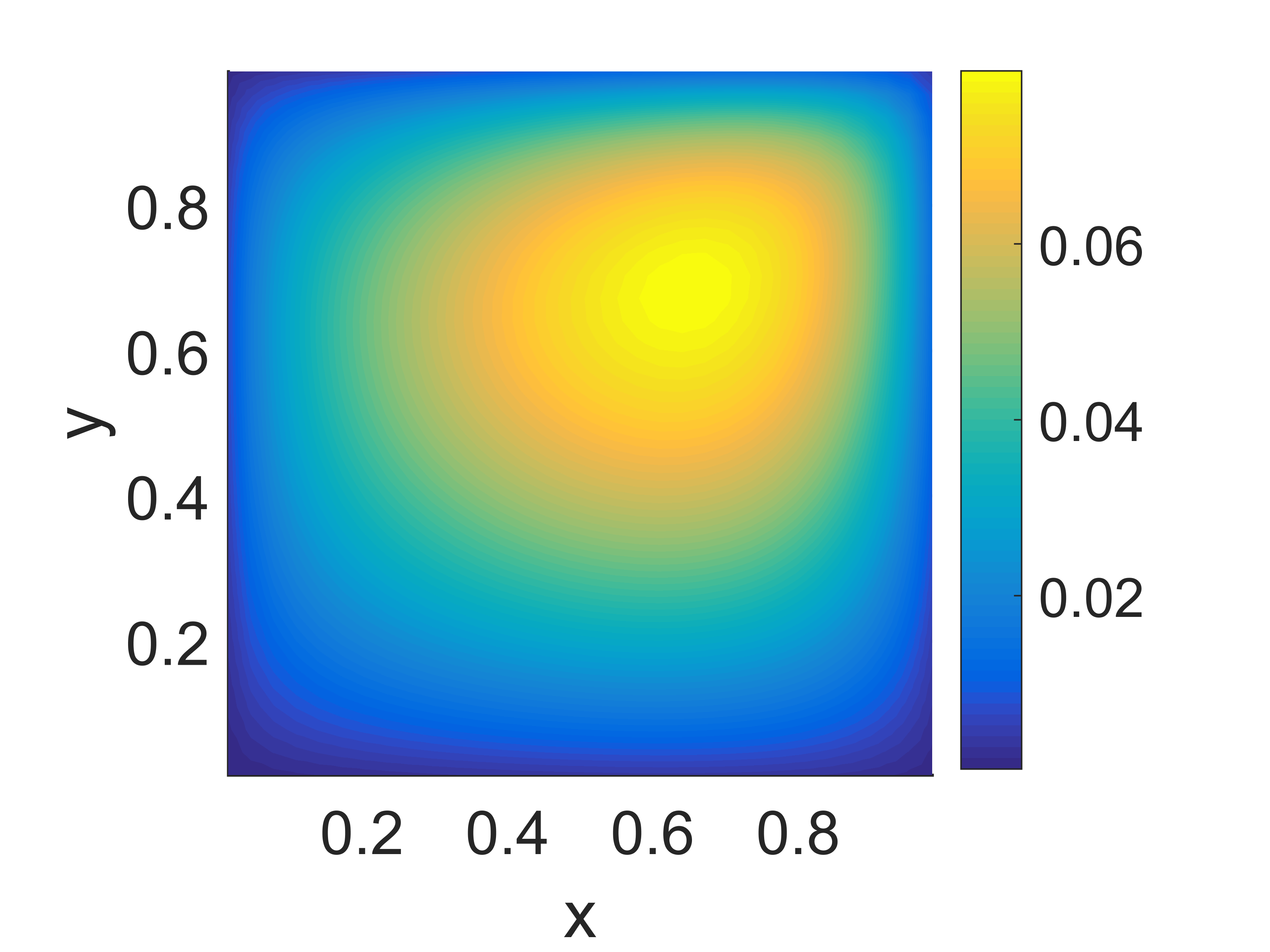}
    }
    \caption{The profile of $\tilde{\phi}$ and $\tilde{\phi}_{\delta}$ for \(M=21\) and \(\delta=10^{-1}\) for buffer zone problem.}
    \label{fig:Example2_center}
    \vspace{-0.3cm}
\end{figure}

To assess the accuracy of Adaptive TFPS at cell centers in the buffer zone problem, we present figures of \(\tilde{\phi}\) and \(\tilde{\phi}_{\delta}\) at cell centers for \(\delta = 10^{-1}\), with \(M = 21\), in \cref{fig:Example2_center}. In \cref{fig:Example2_error_ratio}, we present how the \(\text{error}\) and \(\text{ratio}\) metrics vary with different values of \(\delta\) and \(M\) for the buffer zone case. The results suggest that \(\tilde{\psi}_{\delta}\) maintains a good accuracy in approximating \(\tilde{\psi}\) at the cell centers, while also achieving substantial computational efficiency. Moreover, \cref{fig:Example2_error_ratio} reveals that Adaptive TFPS exhibits first-order convergence with respect to the tolerance \(\delta\), in alignment with the posterior analysis discussed in the earlier section.

\begin{figure}[htbp]
    \setlength{\abovecaptionskip}{0.cm}
    \centering
    \subfloat[$M=1$]
    {
    \includegraphics[width=0.3\textwidth]{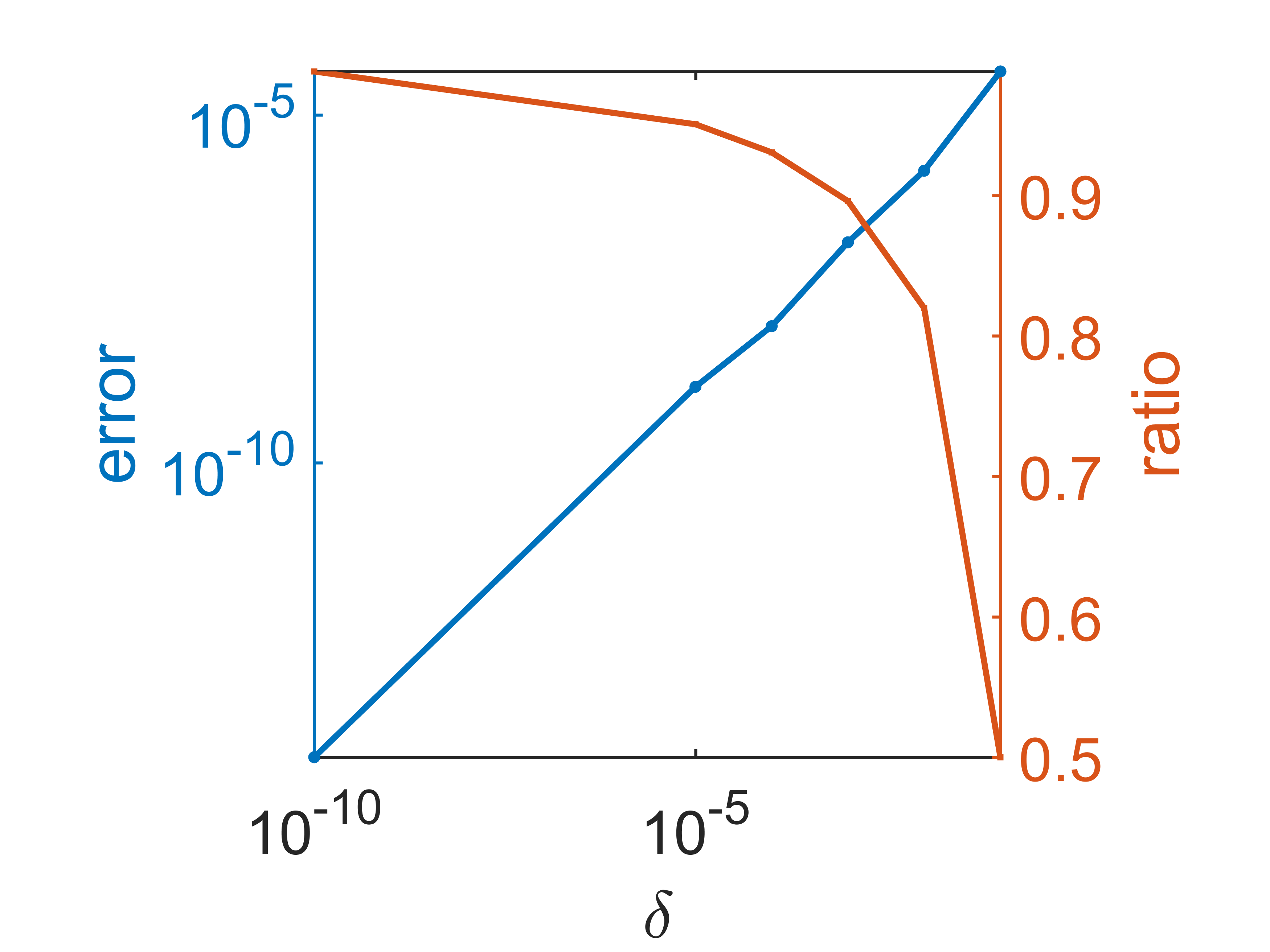}
    }
    \subfloat[$M=3$]
    {
    \includegraphics[width=0.3\textwidth]{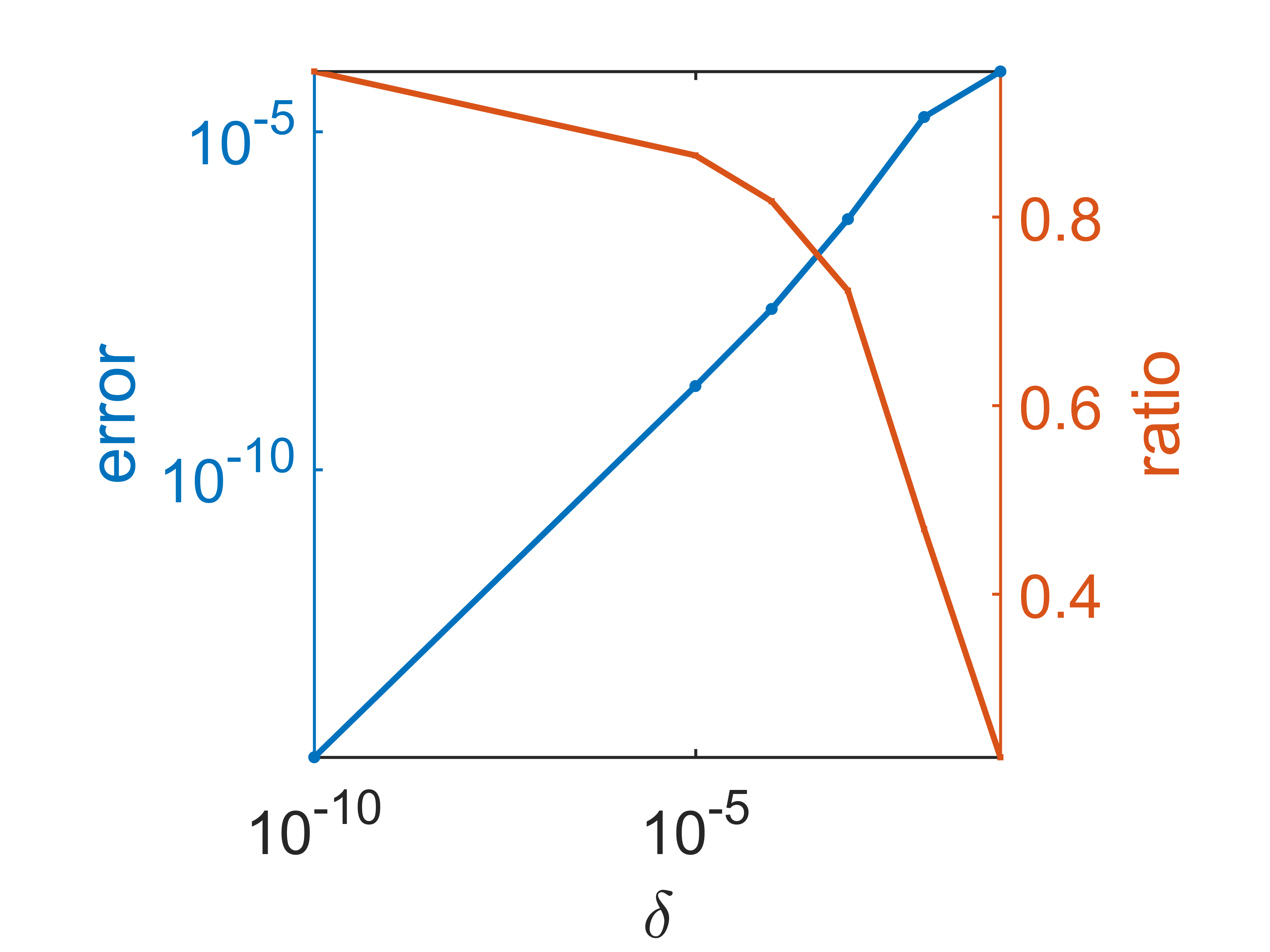}
    }
    \subfloat[$M=6$]
    {
    \includegraphics[width=0.3\textwidth]{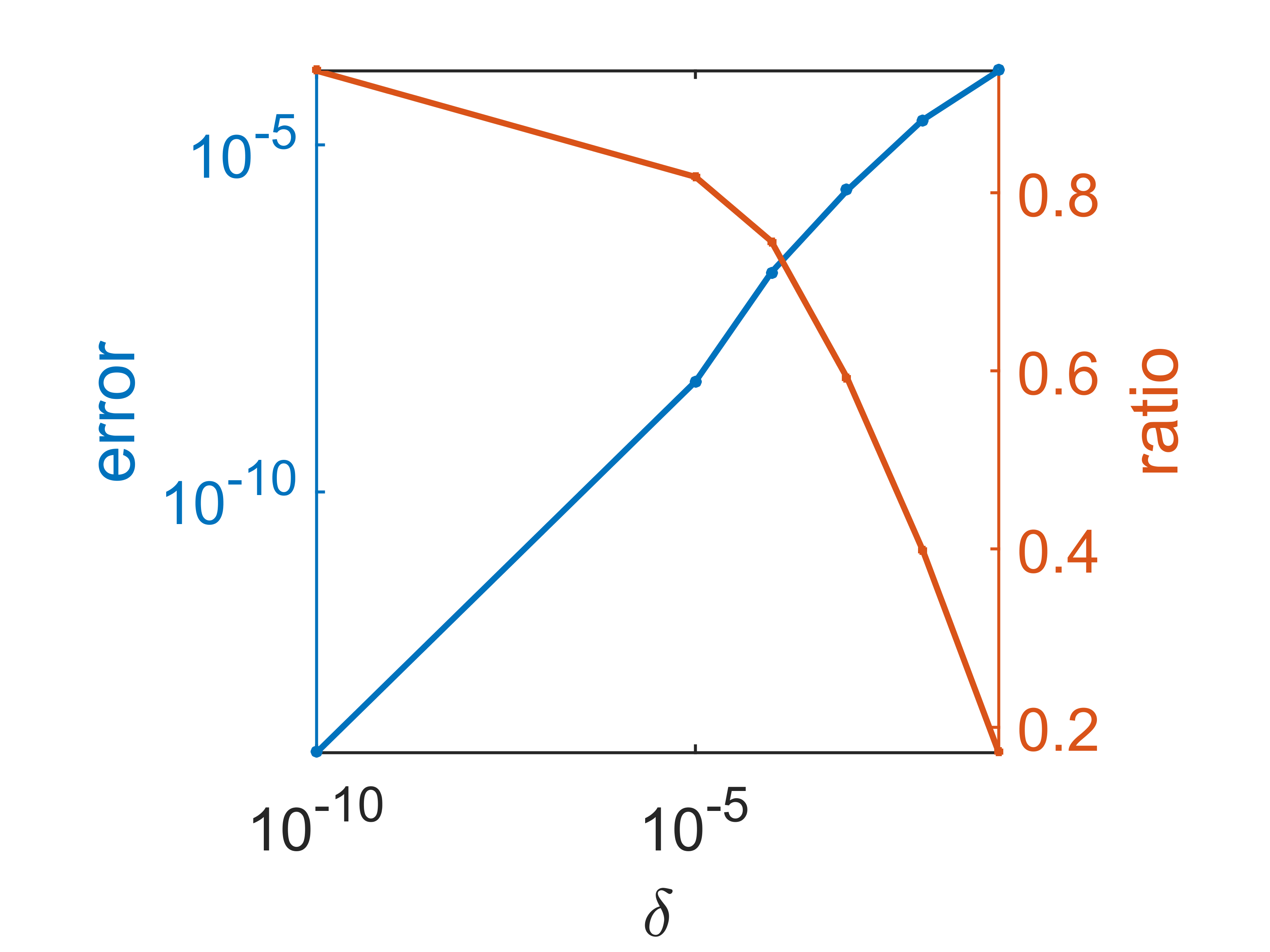}
    }

    \subfloat[$M=10$]
    {
    \includegraphics[width=0.3\textwidth]{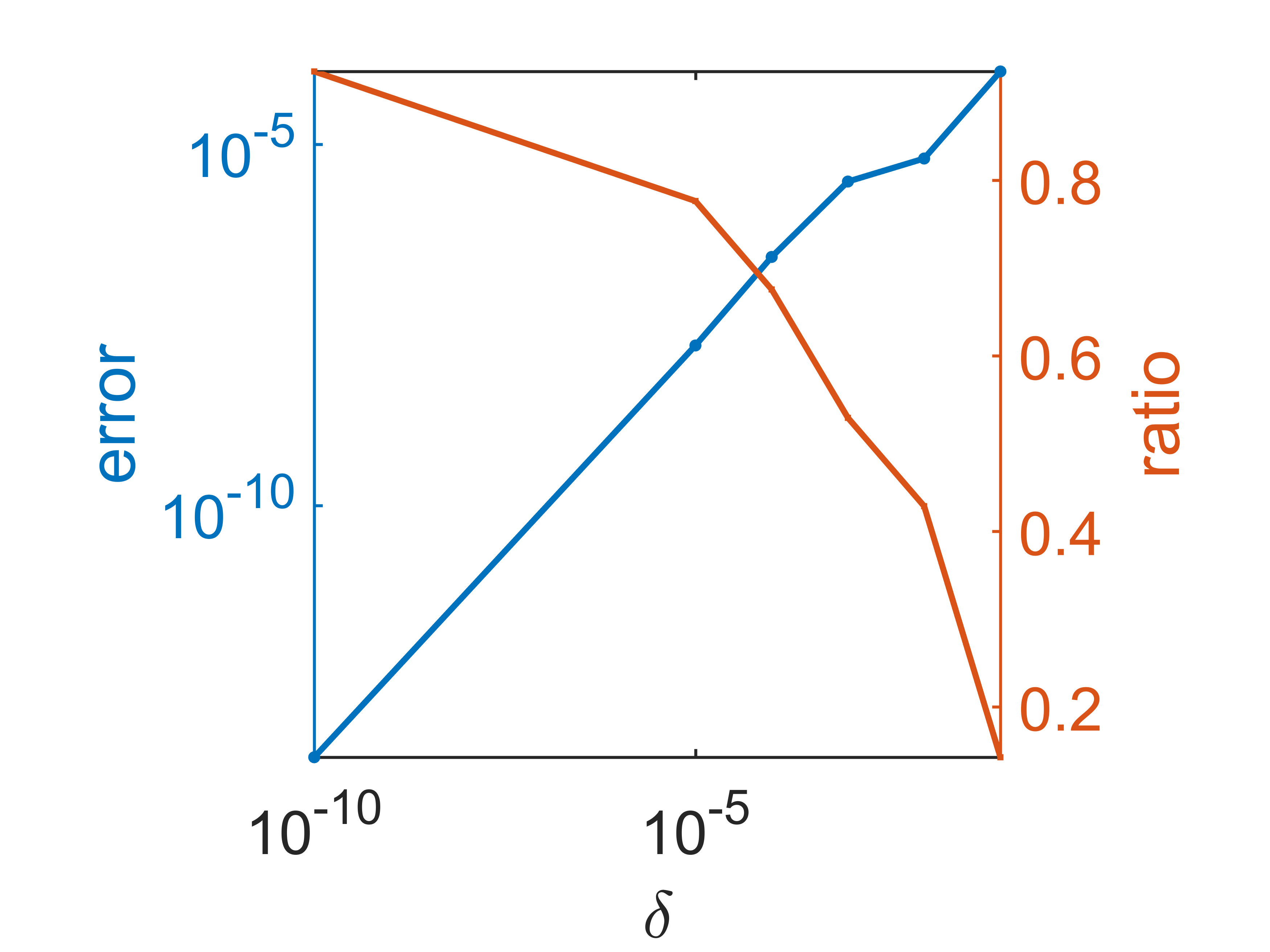}
    }
    \subfloat[$M=15$]
    {
    \includegraphics[width=0.3\textwidth]{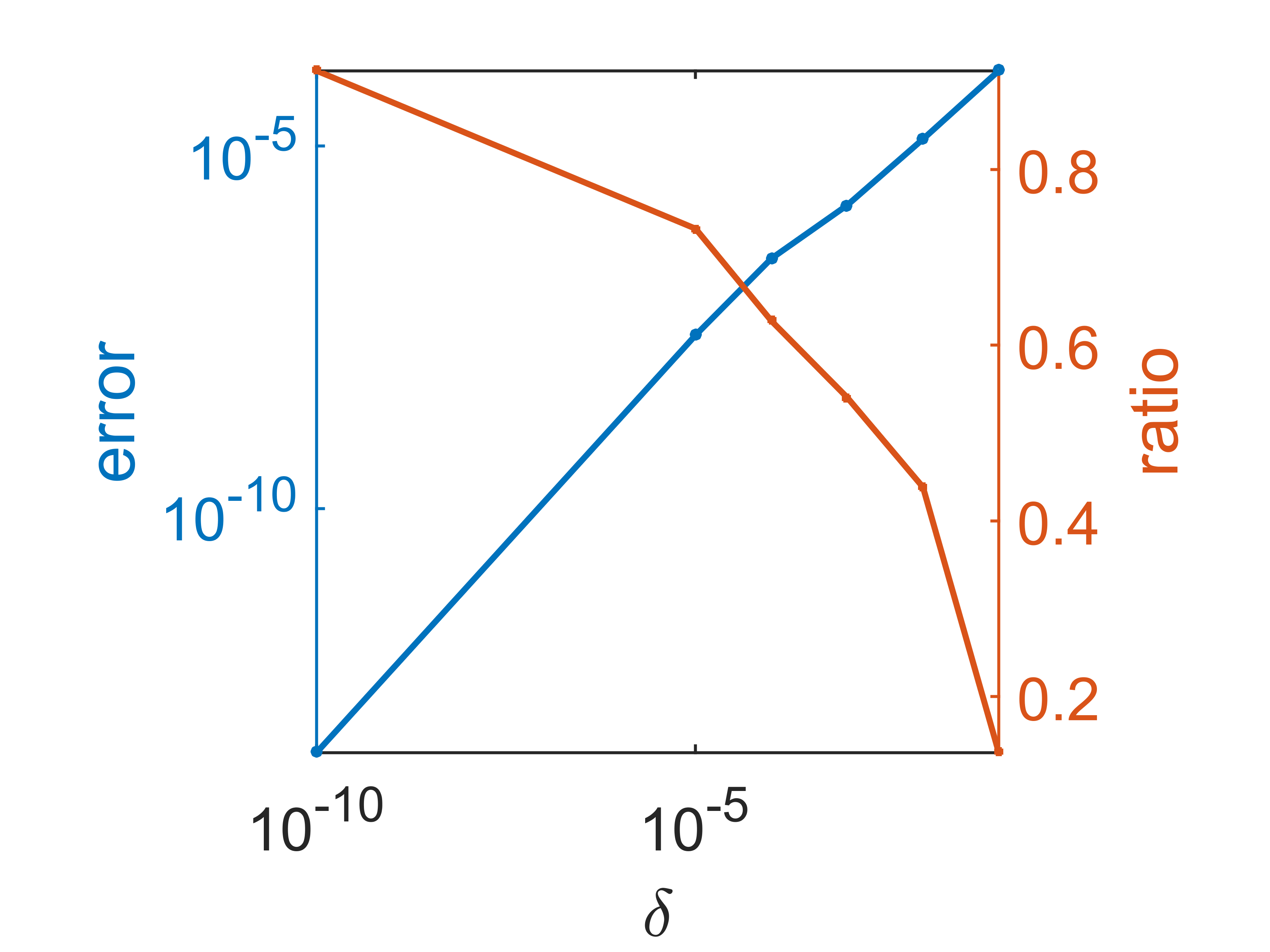}
    }
    \subfloat[$M=21$]
    {
    \includegraphics[width=0.3\textwidth]{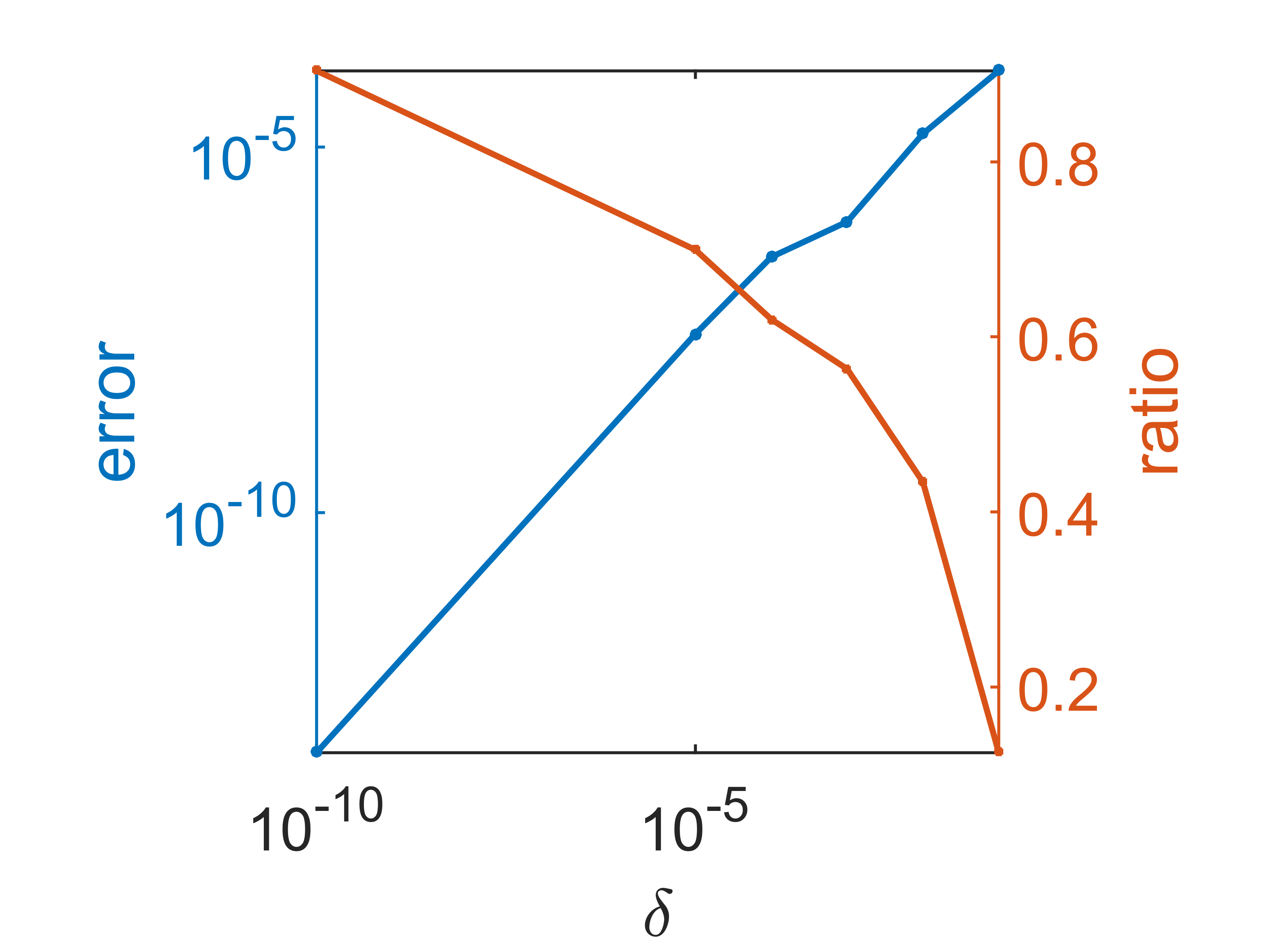}
    }
    \caption{The figure illustrates how the error between \(\tilde{\psi}\) and \(\tilde{\psi}_{\delta}\) and the ratio between the total number of basis utilized in Adaptive TFPS (compressed scheme) and in full TFPS changes with respect to $\delta$ for different choice of $M$ for buffer zone problem.}
    \label{fig:Example2_error_ratio}
    \vspace{-0.3cm}
\end{figure}

\section{Conclusion}
\label{sec:conclusion}

In this article, we propose Adaptive TFPS, a numerical scheme for solving the multiscale radiative transfer equation. In the simple case where a sharp interface exists, with one side in the diffusive regime and the other side in the non-diffusive regime, our scheme can be reduced to the domain decomposition schemes mentioned in \cite{golse2003domain,li2015diffusion}. Adaptive TFPS adaptively compresses the angular domain by considering the local optical properties of the background media. This compression works not only in the diffusion regime but also in other regimes where a low-rank structure exists, such as optically thick regimes and scattering-dominated regimes. Additionally, Adaptive TFPS offers an adjustable threshold \(\delta\), impacting the solution accuracy away from layers and the ability to capture layer information. 
Furthermore, a posterior analysis is conducted to confirm the accuracy of Adaptive TFPS.

Extending our scheme from the 2D case to the 3D case is straightforward and can be easily implemented. In our future work, we aim to expand our approach from the steady-state radiative transport equation to the time-dependent radiative transport equation. This extension will allow us to capture the dynamic behavior of radiative transport phenomena.

\appendix
\section{TFPS basis functions}
\label{appendix:basis}

For any cell \( C \in \mathcal{C} \), the explicit expressions of the TFPS basis localized in cell \( C \) are shown in \cref{eq:sec2:6x} and \cref{eq:sec2:6y}. Within these expressions, \((\lambda_{C}^{(k)}, \xi_{C}^{(k)})\) for \( 1 \leq k \leq 4M \) and $(\lambda_{C}^{(k)},\xi_{C}^{(k)})$ for $4M+1\leq k\leq 8M$ are the eigenpairs corresponding to the following two matrices respectively,
\[
\setlength{\abovedisplayskip}{3pt}
\setlength{\belowdisplayskip}{3pt}
M_{C}^{x}=D^{-1}[\frac{\sigma_{s,C}}{\sigma_{T,C}}KW-I],\quad M_{C}^{y}=S^{-1}[\frac{\sigma_{s,C}}{\sigma_{T,C}}KW-I],
\]
where $D$, $S$, $W$ and $K$ are defined as:
\[D=\diag\{c_{1},c_{2},\dots,c_{4M}\},\  S=\diag\{s_{1},s_{2},\dots,s_{4M}\},\  W=\diag\{\omega_{1},\omega_{2},\dots,\omega_{4M}\},\]
$$K=
\begin{pmatrix}
    \kappa_{1,1} & \kappa_{1,2} &\dots & \kappa_{1,4M}\\
    \kappa_{2,1} & \kappa_{2,2} &\dots & \kappa_{2,4M}\\
    \vdots\\
    \kappa_{4M,1} & \kappa_{4M,2} &\dots & \kappa_{4M,4M}\\
\end{pmatrix}
.$$
We note that \(\xi_{C}^{(k)}\) are normalized such that \(\Vert \xi_{C}^{(k)} \Vert_{\infty} = 1\).

\section{Proof of Lemma \ref{lemma:normBCD}: the upper bound of the infinity norm of \texorpdfstring{$\tilde{B}_{\delta}$}{}, \texorpdfstring{$\tilde{C}_{\delta}$}{}, \texorpdfstring{$\tilde{D}_{\delta}$}{} and \texorpdfstring{$\tilde{b}$}{}}
\label{appendix:BCD}

\paragraph{\textbf{1.Infinity norms of $\tilde{B}_{\delta}$, $\tilde{C}_{\delta}$, and $\tilde{D}_{\delta}-I_{\delta}$:}}

According to the definition of \(\tilde{B}_{\delta}\), its rows correspond to the constraints on angular flux at important velocity modes at certain grid points, while its columns correspond to the unselected basis functions. Therefore, a typical row of \(\tilde{B}_{\delta}\) corresponds to the continuity condition of velocity mode \(\chi_{\mathfrak{i}}^{(k)}\) at the point \(\mathbf{x}_{\mathrm{mid}}\), where \(\mathfrak{i}\) is an interface with the middle point \(\mathbf{x}_{\mathrm{mid}}\), and \(k \in \mathcal{V}_{\delta}^{\mathfrak{i}}\). By \eqref{eq:sec3:41}, \eqref{eq:sec5:4} and \eqref{eq:sec5:5}, we know that if \(\mathfrak{i} \in \mathcal{I}_{i}\) (\(\mathfrak{i} = C_{-} \cap C_{+}\)), then the possibly non-zero elements in the row are:
\begin{equation}\label{eq:appendix:1}   \langle\xi_{C_{-}}^{(k')},\chi_{\mathfrak{i}}^{(k)}\rangle_{\mathfrak{i}}\zeta_{C_{-}}(\mathbf{x}_{\mathrm{mid}},k'),\quad k'\in\bar{\mathcal{V}}_{\delta,C_{-}};\quad \langle\xi_{C_{+}}^{(k'')},\chi_{\mathfrak{i}}^{(k)}\rangle_{\mathfrak{i}}\zeta_{C_{+}}(\mathbf{x}_{\mathrm{mid}},k''),\quad k''\in\bar{\mathcal{V}}_{\delta,C_{+}}.
\end{equation}

For $k'\in \bar{\mathcal{V}}_{\delta}^{\mathfrak{i}}\cap\bar{\mathcal{V}}_{\delta,C_{-}}$, $\langle\xi_{C_{-}}^{(k')},\chi_{\mathfrak{i}}^{(k)}\rangle_{\mathfrak{i}}=0$, and for $k''\in\bar{\mathcal{V}}_{\delta}^{\mathfrak{i}}\cap\bar{\mathcal{V}}_{\delta,C_{+}}$, $\langle\xi_{C_{+}}^{(k'')},\chi_{\mathfrak{i}}^{(k)}\rangle_{\mathfrak{i}}=0$. Therefore, there are actually at most $12M$ non-zero elements in \cref{eq:appendix:1}. 
Given that $k'\in\bar{\mathcal{V}}_{\delta,C_{-}}$ and $k''\in\bar{\mathcal{V}}_{\delta,C_{+}}$,  it follows that $\zeta_{C_{-}}(\mathbf{x}_{\mathrm{mid}},k')\leq \delta$ and $\zeta_{C_{+}}(\mathbf{x}_{\mathrm{mid}},k'')\leq\delta$. 
Additionally, in accordance with \cref{lemma:Cinf} and the requirement that \(\Vert \xi_{C}^{(k)} \Vert_{\infty} = 1\) in the construction of eigenvectors, we deduce that $\langle\xi_{C_{-}}^{(k')},\chi_{\mathfrak{i}}^{(k)}\rangle_{\mathfrak{i}}\leq C_{\gamma,g,M,\delta,\infty}$ and $\langle\xi_{C_{+}}^{(k'')},\chi_{\mathfrak{i}}^{(k)}\rangle_{\mathfrak{i}}\leq C_{\gamma,g,M,\delta,\infty}$. Consequently, if $\mathfrak{i}\in\mathcal{I}_{i}$, there are at most $12M$ non-zero elements in this row, with each element being less than $C_{\gamma,g,M,\delta,\infty}\delta$. Similarly, we can deduce that if $\mathfrak{i}\in\mathcal{I}_{b}$, there are at most $6M$ non-zero elements in this row and each element is less than $C_{\gamma,g,M,\delta,\infty}\delta$. Therefore, we have:
\[
\setlength{\abovedisplayskip}{2pt}
\setlength{\belowdisplayskip}{2pt}
\Vert \tilde{B}_{\delta}\Vert_{\infty}\leq 12MC_{\gamma,g,M,\delta,\infty}\delta.
\]

The proof for the following is similiar, 
\[
\setlength{\abovedisplayskip}{2pt}
\setlength{\belowdisplayskip}{2pt}
\Vert \tilde{C}_{\delta}\Vert_{\infty}\leq 12MC_{\gamma,g,M,\delta,\infty},\quad \Vert \tilde{D}_{\delta}-I_{\delta}\Vert_{\infty}\leq 12MC_{\gamma,g,M,\delta,\infty}\delta.
\]

\paragraph{\textbf{2.Infinity norm of $\tilde{b}$:}}
Each element of $\tilde{b}$ corresponds to the constraints on angular flux at certain velocity modes $\chi_{\mathfrak{i}}^{(k)}$ at certain mid point $\mathbf{x}_{\mathrm{mid}}$ of the interface $\mathfrak{i}$. For an interior interface  $\mathfrak{i}=C_{-}\cap C_{+}\in\mathcal{I}_{i}$, the corresponding element is
\[
\big(\frac{q_{C_{+}}}{\sigma_{a,C_{+}}}-\frac{q_{C_{-}}}{\sigma_{a,C_{-}}}\big)\langle\mathbf{1},\chi_{\mathfrak{i}}^{(k)}\rangle_{\mathfrak{i}}.
\]
While for a boundary interface $\mathfrak{i}=C\cap\partial\Omega\in \mathcal{I}_{b}$, the corresponding element is
\[
\langle\Psi_{\Gamma,\mathfrak{i}-}(\mathbf{x}_{\mathrm{mid}}),\chi_{\mathfrak{i}}^{(k)}\rangle_{\mathfrak{i}}-\frac{q_{C}}{\sigma_{a,C}}\langle\mathbf{1},\chi_{\mathfrak{i}}^{(k)}\rangle_{\mathfrak{i}}
\]
Since 
\[
\setlength{\abovedisplayskip}{3pt}
\setlength{\belowdisplayskip}{3pt}
\big\vert\frac{q_{C_{+}}}{\sigma_{a,C_{+}}}-\frac{q_{C_{-}}}{\sigma_{a,C_{-}}}\big\vert\leq 2\big\Vert\frac{q}{\sigma_{a}}\big\Vert_{\infty},\quad \vert\langle\mathbf{1},\chi_{\mathfrak{i}}^{(k)}\rangle_{\mathfrak{i}}\vert\leq C_{\gamma,g,M,\delta,\infty}\]
\[
\setlength{\abovedisplayskip}{2pt}
\setlength{\belowdisplayskip}{2pt}
\vert\langle\Psi_{\Gamma,\mathfrak{i}-}(\mathbf{x}_{\mathrm{mid}}),\chi_{\mathfrak{i}}^{(k')}\rangle_{\mathfrak{i}}\vert\leq \Vert\Psi_{\Gamma^{-}}\Vert_{\infty}C_{\gamma,g,M,\delta,\infty},\quad \vert\frac{q_{C}}{\sigma_{a,C}}\vert\leq \big\Vert\frac{q}{\sigma_{a}}\big\Vert_{\infty}
\]
we have
\[
\setlength{\abovedisplayskip}{2pt}
\setlength{\belowdisplayskip}{2pt}
\Vert\tilde{b}\Vert_{\infty}\leq C_{\gamma,g,M,\delta,\infty}\big(2\big\Vert\frac{q}{\sigma_{a}}\big\Vert_{\infty}+\Vert\Psi_{\Gamma^{-}}\Vert_{\infty}\big)
\]



\end{document}


\maketitle

\section{Justification of Assumption \ref{assump:linear_independent1}}
\label{appendix:independence1}

The set $\{\xi_{C_{-}}^{(k_{-})},\xi_{C_{+}}^{(k_{+})}\}_{k_{-}\in\mathcal{V}_{C_{-}}^{\mathfrak{i}},k_{+}\in\mathcal{V}_{C_{+}}^{\mathfrak{i}}}$ is determined by the number of discrete velocity directions $4M$, the scattering ratio $\gamma=\frac{\sigma_{s}}{\sigma_{T}}$ and the anisotropy factor $g$ in cell $C_{-}$ and cell $C_{+}$, represented by $\gamma_{C_{-}}=\frac{\sigma_{s,C_{-}}}{\sigma_{T,C_{-}}}$, $\gamma_{C_{+}}=\frac{\sigma_{s,C_{+}}}{\sigma_{T,C_{+}}}$, $g_{C_{-}}$ and $g_{C_{+}}$, according to the definition of eigenvectors in Appendix \ref{appendix:basis}. In the following, we take vertical edge $\mathfrak{i}\in\mathcal{I}_{\mathrm{i}}$ for example and test how the rank ratio of the vectors $\{\xi_{C_{-}}^{(k_{-}},\xi_{C_{+}}^{(k_{+}}\}_{k_{-}\in\mathcal{V}_{C_{-}}^{\mathfrak{i}},k_{+}\in\mathcal{V}_{C_{+}}^{\mathfrak{i}}}$ changes with the values of $\gamma_{C_{-}}$ and $\gamma_{C_{+}}$ for different choice of $g_{C_{-}}$, $g_{C_{+}}$ ($g_{C_{-}}=0$, $g_{C_{+}}=0$ or $g_{C_{-}}=0.3$, $g_{C_{+}}=0$ or $g_{C_{-}}=0.2$, $g_{C_{+}}=-0.3$) and different number of discrete velocity directions ($4M=12$ or 24 or 40 or 60). 
Here the rank ratio is defined as follows:
\[
\mathrm{rank\ ratio}=\frac{\mathrm{rank}\big(\{\xi_{C_{-}}^{(k_{-})},\xi_{C_{+}}^{(k_{+})}\}_{k_{-}\in\mathcal{V}_{C_{-}}^{\mathfrak{i}},k_{+}\in\mathcal{V}_{C_{+}}^{\mathfrak{i}}}\big)}{|\mathcal{V}_{C_{-}}^{\mathfrak{i}}|+|\mathcal{V}_{C_{+}}^{\mathfrak{i}}|}
\]
The results indicate that the rank ratio is consistently 1 regardless of the values of $\gamma_{C_{-}}$, $\gamma_{C_{+}}$, $g_{C_{-}}$, $g_{C_{+}}$ or $M$, as illustrated in \cref{fig:rank2}. This finding suggests that the vectors $\{\xi_{C_{-}}^{(k_{-}},\xi_{C_{+}}^{(k_{+}}\}_{k_{-}\in\mathcal{V}_{C_{-}}^{\mathfrak{i}},k_{+}\in\mathcal{V}_{C_{+}}^{\mathfrak{i}}}$ are linear independent.
\begin{figure}[htbp]
    \setlength{\abovecaptionskip}{0.cm}
    \centering
    \includegraphics[width=0.5\textwidth]{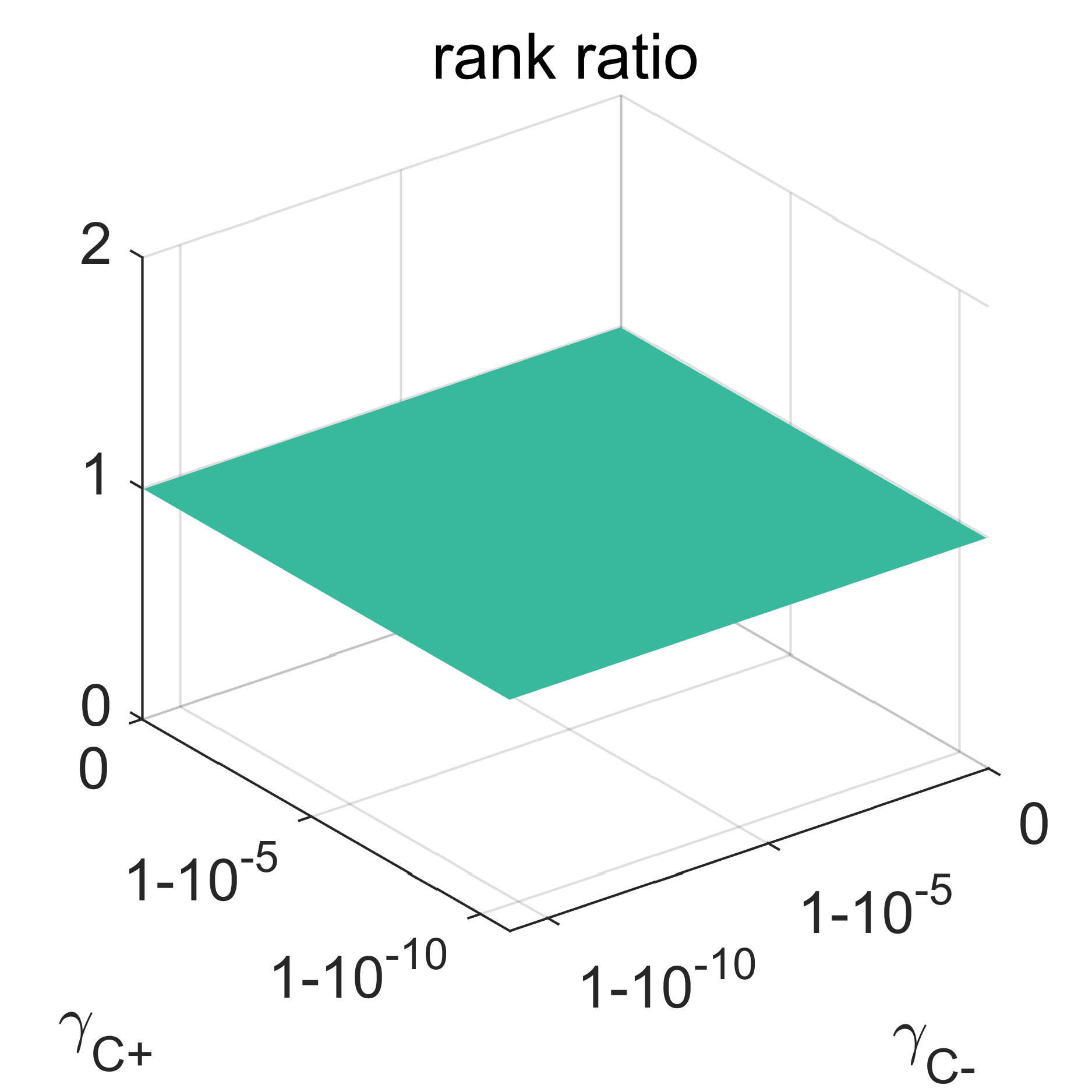}
    \caption{The rank ratio of $\{\xi_{C_{-}}^{(k_{-}},\xi_{C_{+}}^{(k_{+}}\}_{k_{-}\in\mathcal{V}_{C_{-}}^{\mathfrak{i}},k_{+}\in\mathcal{V}_{C_{+}}^{\mathfrak{i}}}$ v.s. $\gamma_{C_{-}}$ and $\gamma_{C_{+}}$ for different choice of $g_{C_{-}}$, $g_{C_{+}}$ ($g_{C_{-}}=0$, $g_{C_{+}}=0$ or $g_{C_{-}}=0.3$, $g_{C_{+}}=0$ or $g_{C_{-}}=0.2$, $g_{C_{+}}=-0.3$), and $M$ ($4M=12$ or 24 or 40 or 60).}
    \label{fig:rank2}
\end{figure}

\section{Justification of Assumption \ref{assump:linear_independent2}}
\label{appendix:independence2}

In the following, we take vertical edge $\mathfrak{i}\in\mathcal{I}_{\mathrm{b}}$ at left physical boundary for example, and test how the rank ratio of the vectors $\{\xi_{C,\mathfrak{i}-}^{(k)}\}_{k\in\mathcal{V}_{C}^{\mathfrak{i}}}$ changes with the values of $\gamma_{C}$ for different choice of $g_{C}$ ($g_{C}=0.3$ or -0.3) and different number of discrete velocity directions ($4M=12$ or 24 or 40 or 60). Here the rank ratio is defined as:
\[
\mathrm{rank\ ratio}=\frac{\mathrm{rank}\big(\{\xi_{C,\mathfrak{i}-}^{(k)}\}_{k\in\mathcal{V}_{C}^{\mathfrak{i}}}\big)}{|\mathcal{V}_{C}^{\mathfrak{i}}|}
\]

The results indicate that the rank ratio is consistently 1 regardless of the values of $\gamma_{C}$, $g_{C}$ or $M$, as illustrated in \cref{fig:rank1}. This finding suggests that the vectors $\{\xi_{C,\mathfrak{i}-}^{(k)}\}_{k\in\mathcal{V}_{C}^{\mathfrak{i}}}$ for $\mathfrak{i}\in\mathcal{I}_{\mathrm{b}}$ are linear independent. 
\begin{figure}[htbp]
    \setlength{\abovecaptionskip}{0.cm}
    \centering
    \includegraphics[width=0.5\textwidth]{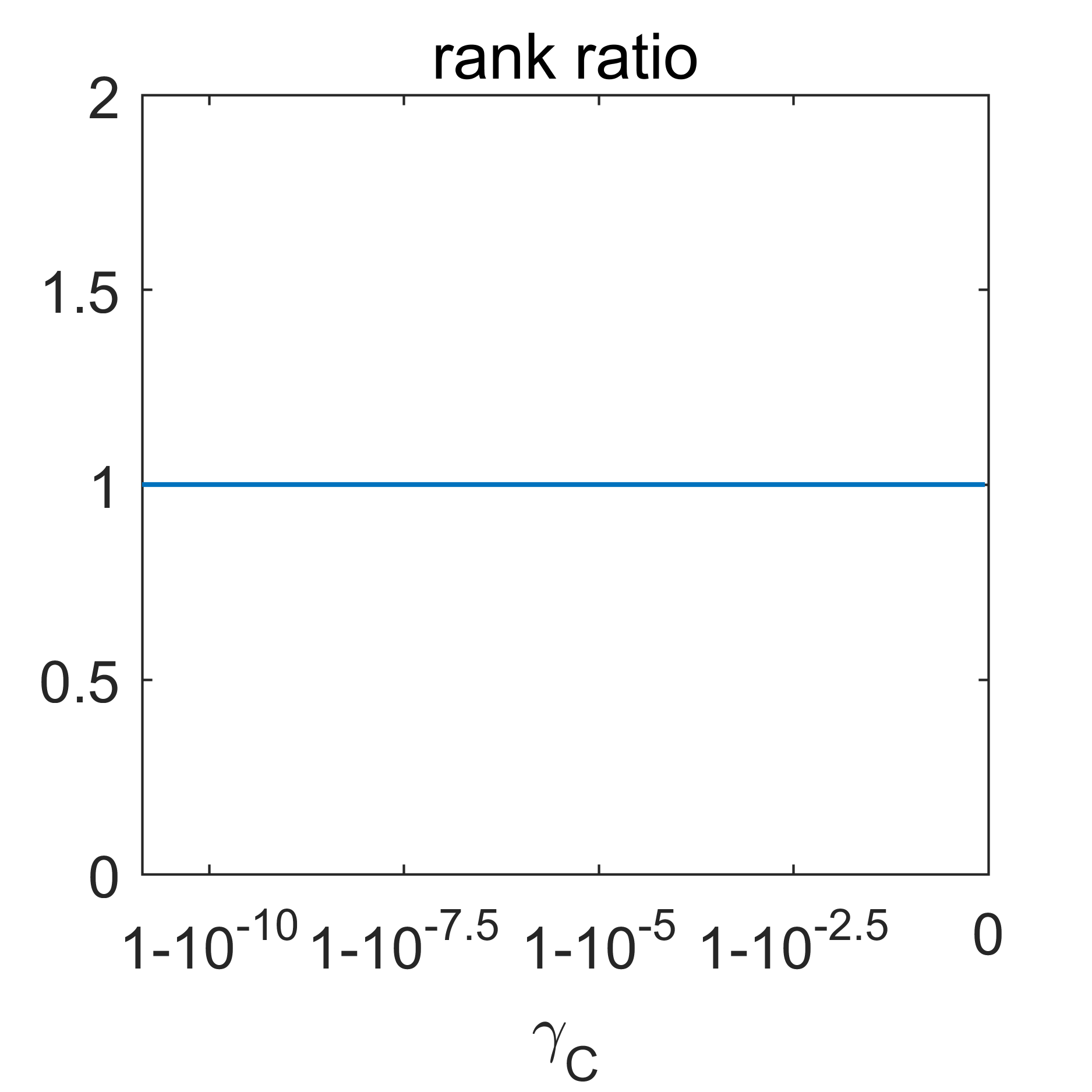}
    \caption{The rank ratio of $\{\xi_{C,\mathfrak{i}-}^{(k)}\}_{k\in\mathcal{V}_{C}^{\mathfrak{i}}}$ v.s. $\gamma_{C}$ for different choice of $g_{C}$ ($g_{C}=0.3$ or -0.3) and $M$ ($4M=12$ or 24 or 40 or 60).}
    \label{fig:rank1}
    \vspace{-0.3cm}
\end{figure}

\section{Justification of Assumption \ref{assump:bound}: boundedness of \texorpdfstring{$C_{\gamma,g,M,\delta,2}$}{}}
\label{appendix:boundedness}

As shown in \eqref{def:C2}, the value of $C_{\gamma,g,M,\delta,2}$ is determined by the maximum $L_{2}$ norm of $E_{\delta,\mathfrak{i}}^{-1}$ for all $\mathfrak{i}\in\mathcal{I}$, while $E_{\delta,\mathfrak{i}}$ is determined by the value of $M$, $\delta$, as well as the local values of $g$ \cite{henyey1941diffuse} and $\gamma$ near interface $\mathfrak{i}$. To demonstrate the boundedness of $C_{\gamma,g,M,\delta,2}$ with respect to $\gamma$, $g$, $M$, and $\delta$, we take an interior vertical edge $\mathfrak{i}=C_{-}\cap C_{+}$ for example and examine how $\mathop{\max}_{\delta} E_{\delta,\mathfrak{i}}^{-1}$ changes with the values of $\gamma_{C_{-}}$ and $\gamma_{C_{+}}$, while considering different choices of $g_{C_{-}}$, $g_{C_{+}}$, and $M$. The results are shown in \cref{fig:Crho}. 
We observe that in all our experimental configurations for \(g_{C_{-}}\), \(g_{C_{+}}\) (\(g_{C_{-}}=0\), \(g_{C_{+}}=0\) or \(g_{C_{-}}=0.3\), \(g_{C_{+}}=0\) or \(g_{C_{-}}=0.2\), \(g_{C_{+}}=-0.3\)), and \(M\) (\(4M=12\) or 24 or 40 or 60), \(\mathop{\max}_{\delta} E_{\delta,\mathfrak{i}}^{-1}\) is bounded above by 2.5. This observation suggests that \(E_{\delta,\mathfrak{i}}^{-1}\) is uniformly bounded with respect to the values of \(M\), \(\delta\), and the local values of \(g\) and \(\gamma\) near the interface \(\mathfrak{i}\). Therefore, \(C_{\gamma,g,M,\delta,2}\) should be uniformly bounded with respect to the values of \(M\), \(\delta\), \(g\), and \(\gamma\).

\begin{figure}
     \setlength{\abovecaptionskip}{0.cm}
    \centering
    \includegraphics[width=0.9\textwidth]{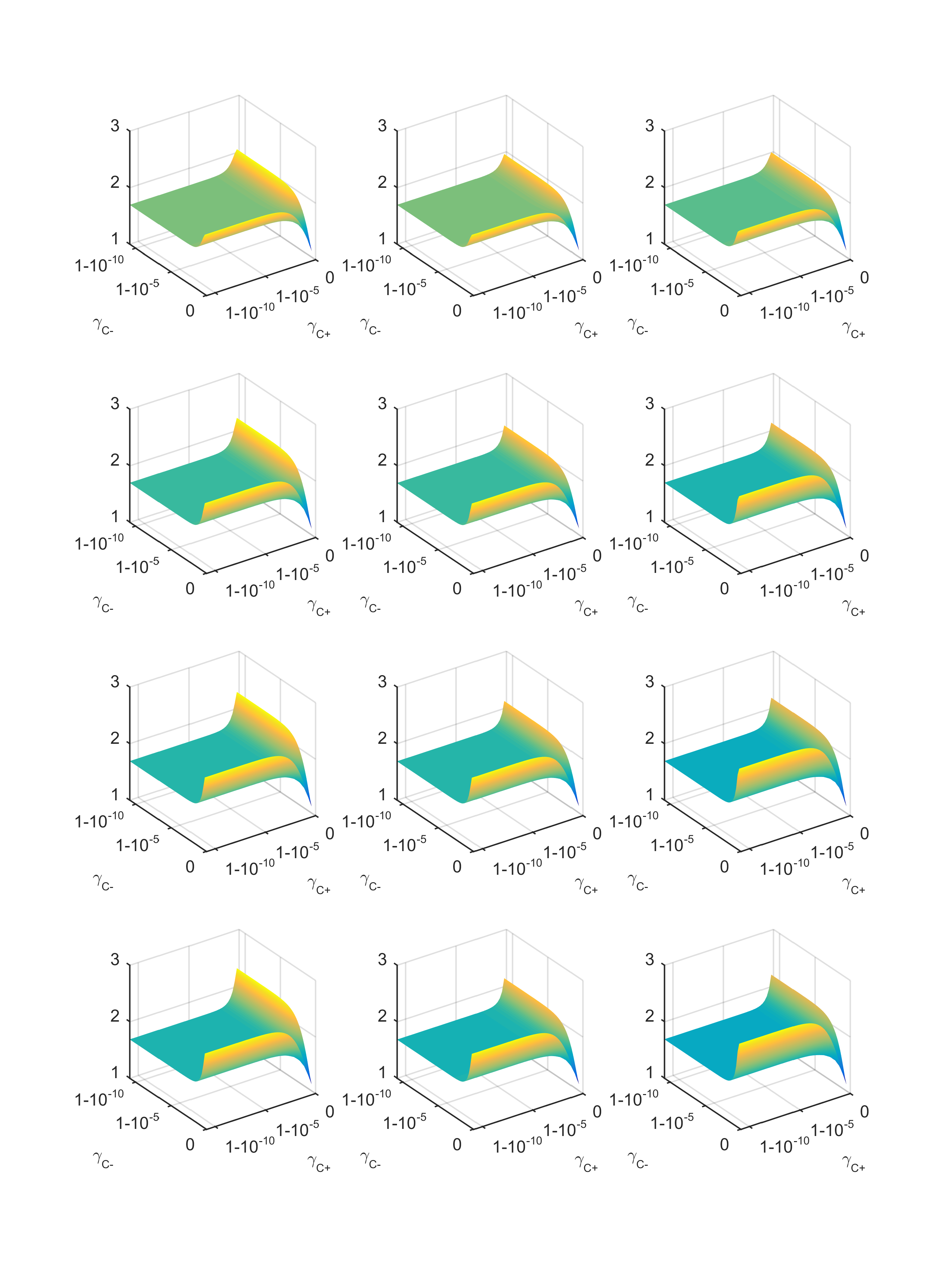}
    \caption{The value of $\mathop{\max}\limits_{\delta} \Vert E_{\delta,\mathfrak{i}}^{-1}\Vert_{2}$ v.s. $\gamma_{C_{-}}$ and $\gamma_{C_{+}}$ for different choice of $g_{C_{-}}$ and $g_{C_{+}}$, and different number of $M$. Left column: $g_{C_{-}}=g_{C_{+}}=0$. Middle column : $g_{C_{-}}=0.3$, $g_{C_{+}}=0$. Right column: $g_{C_{-}}=0.2$, $g_{C_{+}}=-0.3$. From the top row to the bottom row, $M$ changes from 3 to 6, 10, and 15.}
    \label{fig:Crho}
\end{figure}

\bibliographystyle{siamplain}
\bibliography{references}